\documentclass[11pt,reqno]{amsart}

\usepackage{amsrefs}
\usepackage{amsmath}
\usepackage{amsthm}
\usepackage{graphicx}
\usepackage{subcaption}
\usepackage{tikz}
\usetikzlibrary{trees,snakes,external,backgrounds,patterns}
\usepackage{amsfonts}
\usepackage{amssymb}
\usepackage{color}
\usepackage{bm}
\usepackage{hyperref}
\usepackage{enumitem}
\usepackage[margin=1.1in]{geometry}
\usepackage{relsize}
 \usepackage{floatrow}

\numberwithin{equation}{section}

\newtheorem{theorem}{Theorem}[section]
\newtheorem{lemma}[theorem]{Lemma}
\newtheorem{proposition}[theorem]{Proposition}

\newtheorem{corollary}[theorem]{Corollary}

\theoremstyle{definition}
\newtheorem{definition}[theorem]{Definition}
\newtheorem{remark}[theorem]{Remark}

\newtheorem{conjecture}[theorem]{Conjecture}

\newtheorem{assumption}{Assumption}

\def\E{{\mathbb E}}

\def\R{{\mathbb R}}

\def\N{{\mathbb N}}
\def\P{{\mathbb P}}

\def\VV{{\mathbb V}}

\def\T{{\mathbb T}}
\def\F{{\mathcal F}}

\newcommand{\Amc}{{\mathcal{A}}}

\newcommand{\Dmc}{{\mathcal{D}}}
\newcommand{\Emc}{{\mathcal{E}}}
\newcommand{\Fmc}{{\mathcal{F}}}

\newcommand{\Imc}{{\mathcal{I}}}
\newcommand{\Jmc}{{\mathcal{J}}}

\newcommand{\Lmc}{{\mathcal{L}}}
\newcommand{\Mmc}{{\mathcal{M}}}
\newcommand{\Nmc}{{\mathcal{N}}}
\newcommand{\Omc}{{\mathcal{O}}}
\newcommand{\Pmc}{{\mathcal{P}}}
\newcommand{\Qmc}{{\mathcal{Q}}}

\newcommand{\Smc}{{\mathcal{S}}}
\newcommand{\Tmc}{{\mathcal{T}}}

\newcommand{\Vmc}{{\mathcal{V}}}

\newcommand{\Ymc}{{\mathcal{Y}}}

\newcommand{\Nbf}{{\mathbf{N}}}

\newcommand{\muhat}{{\hat{\mu}}}
\newcommand{\phihat}{{\hat{\phi}}}

\newcommand{\xihat}{{\hat{\xi}}}
\newcommand{\Yhat}{{\hat{Y}}}

\newcommand{\shat}{\hat{s}}
\newcommand{\ihat}{\hat{i}}
\newcommand{\sigmahat}{\hat{\sigma}}
\newcommand{\rhat}{\hat{r}}
\newcommand{\Bbar}{{\bar{B}}}
\newcommand{\ebar}{{\bar{e}}}

\newcommand{\ibar}{{\bar{i}}}

\newcommand{\omegabar}{{\bar{\omega}}}

\newcommand{\sbar}{{\bar{s}}}

\newcommand{\wbar}{{\bar{w}}}
\newcommand{\Xbar}{{\bar{X}}}

\newcommand{\Erdos}{Erd\H{o}s-R\'enyi}

\newcommand{\set}[1]{\left\{#1\right\}}

\newcommand{\one}{{\boldsymbol{1}}}

\newcommand{\ftil}{{\tilde{f}}}
\newcommand{\xitil}{{\tilde{\xi}}}
\newcommand{\psitil}{{\tilde{\psi}}}
\newcommand{\dtil}{ \tilde{d}}

\newcommand{\vtil}{{\tilde{v}}}
\newcommand{\Tmctil}{{\tilde{\Tmc}}}
\newcommand{\mapvtil}{\tilde{\phi}}

\def\root{\varnothing}
\newcommand{\stateS}{\mathcal{X}}
\newcommand{\stateSS}{\bar{\stateS}}
\newcommand{\natzero}{\N_0}
\newcommand{\law}{\Lmc}
\DeclareSymbolFont{symbolsC}{U}{pxsyc}{m}{n}
\DeclareMathSymbol{\coloneqq}{\mathrel}{symbolsC}{"42}

\newcommand{\V}{\mathbb{V}}

\newcommand{\qhat}{\hat{q}}

\usepackage{mathtools}

\newcommand{\extra}{\star}

\newcommand{\Finfty}{\Fmc}
\newcommand{\infects}{\Imc}

\newcommand{\offspring}{\theta}
\newcommand{\sizeb}{{\hat{\offspring}}}

\newcommand{\x}{x}
\newcommand{\y}{y}
\newcommand{\f}{F}
\newcommand{\z}{z}

\newcommand{\condi}{r}

\newcommand{\project}{\eta}

\title[Exact description of limiting SIR and SEIR dynamics on locally tree-like graphs]{Exact description of limiting SIR and SEIR dynamics on locally tree-like graphs}

\author[Cocomello]{Juniper Cocomello}           
\author[Ramanan]{Kavita Ramanan}
                \thanks{K. Ramanan was supported in part by ARO Grant W911NF2010133 and both authors were supported by the Office of Naval Research under the Vannevar Bush Faculty Fellowship N0014-21-1-2887.} 
\address{Division of Applied Mathematics, Brown University, 182 George Street, Providence, RI 02912} 
\email{juniper\_cocomello@brown.edu, kavita\_ramanan@brown.edu}

\thanks{\textit{Acknowledgements:} The second author would like to thank the Simons Institute and the organizers of the workshop on \href{https://simons.berkeley.edu/workshops/graph-limits-nonparametric-models-estimation}{“Graph Limits, Nonparametric Models, and Estimation”} for inviting her to the workshop to present a preliminary version of these results.}
\begin{document}
\begin{abstract}
We study the Susceptible-Infected-Recovered (SIR) and the Susceptible-Exposed-Infected-Recovered (SEIR) models of epidemics, with possibly time-varying rates,  on a class of networks that are locally tree-like, which includes  sparse \Erdos  \ random graphs, random regular graphs, and other configuration models.
We identify tractable systems of ODEs that exactly describe the dynamics of the SIR and SEIR processes in a suitable asymptotic regime in which the population size goes to infinity.  Moreover, in the case of constant recovery and infection rates, we  characterize the outbreak size  as the unique zero of an explicit functional. We use this to show that  a (suitably defined) mean-field prediction always overestimates the outbreak size, and that the outbreak sizes for SIR and SEIR processes with 
 the same initial condition and constant infection and recovery rates coincide.   In contrast,  we show that the outbreak sizes for SIR and SEIR processes with the same 
 time-varying infection and recovery rates can  in  general be quite different. 
 We also demonstrate via simulations the efficacy of our approximations for populations of moderate size. 
\end{abstract}


\maketitle
%%%%%%%%%%%%%%%%%%%%%%%%%%%%%%%%%%%%%%%%%%%%%%%%%%%%%%%%%%%%%%%%%%%%%%

\section{Introduction}

\noindent 
{\bf Models, Results and Proof Techniques. } 
The Susceptible-Infected-Recovered (SIR)  model has been extensively used to study the spread of infectious diseases, computer viruses, information, and rumors.  In this model, each individual in a population is represented as being in one of three states: susceptible (does not have the disease, but could catch it), infected (has the disease and can spread it to susceptible individuals), or recovered (no longer has the disease and cannot be reinfected). 
The dynamics are governed by a graph, which describes the contact network of individuals in a population, and two strictly positive parameters: $\rho$, the rate at which an infected individual recovers, and $\beta$, the rate at which an infected individual transmits the disease to a susceptible individual with whom it is in  immediate contact.  Random graphs are used to model variation and uncertainly in real-world contact networks. We are interested in the global dynamical behavior that arises from the local interactions of disease spreading. In particular, we are interested in the following questions: how many individuals are at each of the states at any time $t\in[0,\infty)$? What is the total size of the outbreak, that is,  how many individuals were infected during the course of the epidemic?  Answering these questions for large populations can be analytically challenging, and simulations are computationally expensive and do not easily allow for rigorous characterization of qualitative behavior. 

In the present work, we study a continuous-time stochastic SIR process, and a related epidemic model, the Susceptible-Exposed-Infected-Recovered (SEIR) process, where an additional state is considered for individuals that have been exposed to the pathogen but have not yet become infectious.  The SEIR process has been widely used to model the spreading of diseases including the recent SARS-CoV-2 pandemic, for instance, see \cite{Suwardi2020stability, girardi2023anseir ,mwalili2020SEIR}. 
 In both cases, we  allow for the (infection and recovery)  transition  rates to be time-dependent, so as to model effects due to seasonal variations, changes in the virulence of a disease,  developments in treatment options, and changes in public health policies, which are of significant interest in practice \cite{morris2021optimal,lopez2021modified,fisman2007seasonality}.
While the majority of works have considered SIR processes on dense networks, we consider these processes on  sparse networks (i.e., where each individual is connected to a bounded number of individuals), which more faithfully describe real-world networks. We provide tractable approximations for the evolution of fractions  of individuals in  each of the states of the epidemic in terms of a coupled system of   ordinary differential equations (ODEs), see \eqref{eq:f_If_SF_I} and \eqref{eq:gggG-Odes}, and for the outbreak size, which is the final fraction of individuals ever infected.  
Moreover, we  show that these approximations are asymptotically exact, as the size of the population increases to infinity,  when the graph governing the dynamics is locally-tree like. 
More precisely, we consider a broad class of sparse (random)  graph sequences, including sparse \Erdos \ random graphs, random regular graphs, and certain configuration models, 
which are known to  converge in a certain local (weak)   sense to  a random limit that belongs to the class of  unimodular Galton-Watson (UGW) trees; see  Theorem \ref{thm:MAIN_SIR} and Theorem \ref{thm:SEIR_main}.  We refer the reader to Definition \ref{def:UGW} for the definition of a UGW tree, and to \cite{vanderHofstad2023vol2, Aldous2004objective} for an extensive account of local  convergence of graphs. 

Our proof technique  starts by appealing to  a general result in \cite{Ganguly2022hydrodynamic} that shows that for a general class of interacting particle systems that includes the SIR and SEIR processes, the  sequence of empirical measures (equivalently,  fractions of individuals in different states) on any  locally converging sequence of graphs
converges  to the law of the marginal evolution of the root node  in the limit UGW tree  (see also \cite{Lacker2023local,RamICM22} for related results). 
The key step is then to provide a tractable characterization of the  root marginal dynamics of this infinite-dimensional process. 
% (see also \cite{Lacker2023local} for related results for interacting diffusions.  
While for general particle systems the marginal dynamics of the root,  or even of the  root and its neighborhood,  could be  non-Markovian,   
a key step in our proof is to show that  for the SIR and SEIR models, the dynamics simplifies. In fact, we can deduce from our proof that the evolution of the pair of vertices 
consisting of the root and an  offspring is in fact Markovian (see Remark \ref{rem-Markov}).  
 The proof of the latter property relies crucially on certain conditional independence relations that we identify  (see Proposition \ref{prop:cond_ind_prop}) and a certain projection lemma (Proposition \ref{prop:filtering})  in the spirit of \cite[Lemma 9.1]{Ganguly2022nonmarkovian}. 
   These properties are combined with  symmetry properties of the dynamics  (identified in Proposition \ref{prop:symmetry}) to obtain a very tractable description of the evolution of the marginal law of the root in terms of the abovementioned systems of ODEs.  
  
For both the SIR and SEIR models, the associated system of ODEs is then analyzed to characterize the 
 outbreak size in terms of  the moment generating function of the offspring distribution  of the limiting UGW tree, evaluated at the unique zero of an explicitly given functional; see Theorem \ref{thm:SIR-outbreak}.
In the case of constant recovery and infection rates, 
we obtain a simpler characterization of the outbreak size and use it to show 
that the (suitably defined) mean-field prediction always overestimates the outbreak size.
In this setting, we also show 
that although the transient dynamics can be different, 
 the outbreak sizes  for the SIR and SEIR models coincide when they have the same rates and initial conditions. 
In particular, this shows that in this case the outbreak size for the SEIR model does not depend on the rate at which an exposed individual becomes infectious.
In contrast,   we  show that when the rates are time-varying, the outbreak sizes of the corresponding SIR and SEIR processes no longer coincide and can be vastly different even when the (time-varying) ratios of the infection  rate to the recovery rate coincide.  
 For both transient dynamics and the outbreak size, we compare our results with numerical simulations to  demonstrate the  efficacy of  these approximations for populations of even moderate size.   We also show how the ODEs can be used to study the impact of the amplitude and phase of periodically varying rates on the  outbreak size. 

When the infection and recovery rates are constant in time, traditional techniques to analyze the outbreak size of the SIR process exploit a reformulation of the final outbreak size in terms of a bond percolation problem. 
However, it is not apparent if such a simple correspondence exists when the infection and recovery rates are time-varying, and unlike our approach,  percolation-based arguments provide limited insight into the dynamics
of the epidemic process.
Furthermore, our general approach can also be applied  to obtain analogous results for other more general epidemic processes including a class of compartmental models and processes with general recovery distributions; 
 see Remark \ref{rmk:generalize_times} and Remark \ref{rmk:general_process}.  It would also be of interest to investigate  to what extent an analogous approach can be used to provide alternatives to mean-field approximations for other classes of models, for instance, such as those described in \cite{RamQuesta22}. We defer a complete investigation to future work.  \\

\noindent 
{\bf Discussion of Prior Work. }
Understanding epidemic dynamics on networks is an active area of contemporary research. The deterministic SIR model, introduced in \cite{kermack1927contribution}, is a system of coupled ODEs that describes the evolution over time of the fraction of individuals in each of the states of the epidemic, in a population where everyone can come into contact with everyone else.   This is known as the mean-field approximation. The mean-field dynamics are known to emerge as the large $n$ limit of the SIR process defined on the complete graph on $n$ vertices, when the infection rate scales like $\Omc(1/n)$.
The mean-field approximation provides a dramatic reduction of dimensionality, as it captures the global behavior of a size $n$ population by a coupled system of two ODEs. However, most real-world contact networks are sparse, in the sense that the average  number of neighbors for an individual in the network remands bounded even when the population size grows.

Because of this, and the application-driven need to understand epidemic dynamics on more realistic networks, the study of  SIR dynamics on a range of more realistic sparse network structures is an active area of research.  
The work of \cite{schutz2008exact} derives equations for the expected number of individuals in each SIR state on a cycle graph, and compares these results with the corresponding quantities associated with the SIR model on the complete graph, as well as the scaled dynamics that result in the mean-field approximation.
An SIR model on the $\kappa$-regular tree (the infinite tree where every vertex has $\kappa$ neighbors) with general recovery times and time-dependent rates was studied in \cite{gairat2022discrete}. The latter work derives the asymptotic limit, as $\kappa$ goes to infinity, of  the evolution of the fraction of susceptible individuals over time, which recovers the mean-field approximation. 
Differential equations to approximate the fraction of susceptible, infected and recovered individuals for the continuous-time SIR model on configuration model graphs were derived heuristically in 
\cite{Volz2008SIR, Volz2009epidemic}  
and shown to be asymptotically exact, as the population size goes to infinity, in \cite{Decreusefond2012large, janson2015law}.  
In very recent work \cite{hall2023exact}, the authors obtain an explicit representation of the marginal distribution of each node on a finite tree by solving a coupled system of ODEs. This representation is shown to provide an upper bound on the probability that a node is susceptible on general graphs or with more than one initial infection. However, they show, via simulations, that this upper bound is generally not very tight. 

 Existing mathematically rigorous work on the SEIR process focuses on studying the deterministic dynamics that arise in the mean-field regime, see \cite{li1999global}.  
 To the best of our knowledge, not  much is known rigorously about  SEIR processes on  sparse graphs or corresponding  limits. 
In \cite{Zhao2013SEIR}, the authors present an ODE system that they heuristically argue should approximate the fraction for large populations size.
 However, their approximation is not compared with simulations, and it differs from our ODE system, which  is asymptotically exact as the population size approaches infinity. 
 % These differential equations are  similar though not equivalent to the approximation we prove to be asymptotically exact.
 \\

\noindent 
{\bf Organization of the Paper. } 
The rest of the paper is structured as follows.  In Section \ref{sec:notation} we introduce some common notation used throughout the paper. 
In Section \ref{sec:transient} we define the SIR and SEIR processes and state our characterization of the large-population limit of epidemic dynamics (see Theorem \ref{thm:MAIN_SIR} and Theorem \ref{thm:SEIR_main}). In Section \ref{sec:outbreak} we provide a characterization of the outbreak size in the large-population limit (see Theorem \ref{thm:SIR-outbreak} and Theorem \ref{thm:SEIR-outbreak}). The proofs of our results are provided in Section \ref{sec:main_proofs}. They rely on a  conditional independence property that is proven in Section \ref{sec:cond_ind_proof} and some auxiliary results on the SEIR process that are relegated to Appendix \ref{sec:proof-SEIR}. Additionally,  the proof of the well-posedness of the limit ODE systems are  given in Appendix \ref{sec:odes-well-posed}. 

\subsection{Notation}
\label{sec:notation} We briefly overview common notation used throughout the paper. We use $G=(V,E)$ to denote a graph with vertex set $V$ and edge set $E$. When clear from context, we identify a graph with its vertex set, and so for a vertex $v$ we might write $v\in G$ instead of the more accurate $v\in V$. We let $|G|:=|V|$ denote the number of vertices of $G$.  Given $A\subset V$, we let $\partial^G A :=\{w\in V\ : \ \{w,v\}\in E, \ v\in A, w\in V\setminus A\}$ be the \textit{boundary} of $A$. In the case where $A=\{v\}$ is a singleton, we write $\partial^G_v:=\partial^G\{v\}$, and refer to it as the set of \textit{neighbors} of $v$. The \textit{degree} of a vertex is defined as $d_v^G:=|\partial_v^G|$. When unambiguous, we omit the dependence on $G$ from our notation, and write $d_v$ and $ \partial_v$. For $v,w\in G$, we write $v\sim w$ to mean $v\in\partial_w$. Given a set $\Ymc$, a configuration $y\in \Ymc^V$ 
 and $A\subset V$, we write $y_A:= \{y_v\ : \ v\in A\}$, and in the special case when $|A|=2$, $y_{v,w}:=y_{\{v,w\}}$.
 
We let $\natzero=\{0,1,2,...\}$, and let $\Pmc(\natzero)$ be the set of probability measures on $\natzero$. We identify probability measures on $\natzero$ with their probability mass functions. In particular, for $\zeta\in\P(\natzero)$ and $k\in\natzero$, we write $\zeta(k)=\zeta(\{k\})$. For $k\in \R$, we let $\delta_k$ be the Dirac measure at $k$. Given a probability space $(\Omega,\Fmc,\P)$, we denote by $\law(Y)$ the law of a $\Omega-$valued random variable $Y$.

\section{Results on Transient Dynamics}\label{sec:transient}
In Section \ref{sec:SIR} 
we precisely define the SIR process and in Section \ref{sec:ODEs} state the main result that describes the  limiting dynamics on converging sequences of locally tree-like graphs in terms of  solutions to systems of ODEs. In 
 Section \ref{sec:SEIR} we define  the SEIR process and state the corresponding convergence result. 
 
\subsection{SIR Model} \label{sec:SIR}
Fix a graph $G=(V,E)$, the (time-varying) \textit{infection rate}  ${\beta}:[0,\infty)\rightarrow(0,\infty)$, and the (time-varying) \textit{recovery rate} ${\rho}:[0,\infty)\rightarrow(0,\infty)$. We write $\beta_t$ (resp. $\rho_t$) for the value of $\beta$ (resp. $\rho$) at time $t\in[0,\infty)$. The SIR process on $G$, denoted by $X^G$, is  a continuous-time locally interacting Markov chain with the following dynamics.
At any time $t$, each individual $v$ has a state $X^G_v(t)$ in the space $\stateS:=\set{S,I,R}$. The initial states $X^G_v(0)$ are i.i.d. with $\P(X^G_v(0)=S)=s_0$ and $\P(X^G_v(0)=I)=i_0:=1-s_0$ for some $s_0\in (0,1)$.  Given $y\in \emptyset\cup(\cup_{k=1}^\infty \stateS^k)$, representing the configuration of the neighbors of a vertex, we denote by $\infects(y)$ the number of elements of $y$ that are equal to $I$.  At time $t$, each individual $v\in G$ jumps from $S$ to $I$ (i.e., becomes infected) at rate ${\beta_t} I(X^G_{\partial_v}(t-))$, and from $I$ to $R$ (i.e., recovers) at rate ${\rho_t}$. We impose the following mild assumptions on the recovery and infection rate functions.
\begin{assumption}\label{assu:beta_rho}
    The functions $\beta$ and $\rho$ are continuous and there exist $c_1,c_2 \in(0,\infty)$ such that
    \begin{align}\label{eq:assumption_bdd}
        \begin{split}
             c_1<\liminf_{t\rightarrow\infty }\min(\rho_t,\beta_t)< \limsup_{t\rightarrow \infty}\max(\rho_t,\beta_t)<c_2.
        \end{split}
    \end{align}
\end{assumption}
Throughout the paper, we assume that Assumption \ref{assu:beta_rho} holds. 
\begin{remark}\label{rmk:monotonicty}
    If we equip $\stateS$ with the total ordering given by $
    S< I< R$, then  the SIR process is monotonic in the sense that for every $v\in G$ and $s,t\in[0,\infty)$, if $s\leq t$ then  $X_v^G(s)\leq X_v^G(t)$.
\end{remark}

Next, we describe the class of graph sequences that we consider, as well as an associated probability measure on $\natzero$ that characterizes the corresponding local limit.

\begin{assumption} \label{assu:graph_sequence} Suppose the sequence of graphs $\{G_n\}_{n\in\N}$ and  $\offspring \in \Pmc(\natzero)$ satisfy one of the following:  
\begin{enumerate}
        \item \textit{(\Erdos)}. There exists $c>0$ such that, for every $n\in\natzero$, $G_n$ is a \Erdos \ random graph $\text{ER}(n,c/n) $, and
    $\offspring$ is the Poisson distribution with mean $c$.
    \item \textit{(Configuration Model)}. For each $n$, let $\{d_{i,n}\}_{i=1}^n$ be a graphical sequence, such that $\sum_{i=1}^n \delta_{d_{i,n}}$ converges weakly to $\offspring$ as $n\rightarrow\infty$, and $\offspring$ has finite third moment. Let $G_n$ be a graph uniformly chosen among graphs on $n$ vertices with degree sequence $\{d_{i,n}\}_{i=1}^n$. We write $G_n=\text{CM}_n(\offspring)$. 
\end{enumerate}
\end{assumption}
% By taking $\offspring=\delta_\kappa$ with $\kappa\in\N\setminus\{1\}$ the configuration model reduces to the random $\kappa$-regular graph. 
\begin{remark}
    The only place where we use the assumption that $\offspring$ has a finite third moment is in  Proposition \ref{prop:f_If_SF_I} below (and the corresponding result for the SEIR process, Proposition \ref{prop:gggG}). Every result in this paper holds by replacing the assumption that $\offspring$ has finite third moment in Assumption \ref{assu:graph_sequence}(2) with the assumption that $\offspring$ has finite second moment and that the system \eqref{eq:f_If_SF_I}-\eqref{eq:ffF-initial} (and the corresponding system for the SEIR process \eqref{eq:gggG-Odes}-\eqref{eq:gggG-Odes_initial}) has a unique solution on $[0,\infty)$.
\end{remark}
We refer the reader to \cite[Chapter 5 and Chapter 7]{hofstad2016vol1} for an extensive account of random graphs, including precise definitions and well-known properties of the graphs in Assumption \ref{assu:graph_sequence}. The class of graphs we consider is locally tree-like, in a sense that we now make precise. Given $\offspring\in\Pmc(\natzero)$ with finite first moment, we define its \textit{size-biased distribution} $\sizeb\in\Pmc(\natzero)$ by
\begin{equation}\label{eq:sizab}
        \sizeb(k)= \frac{(k+1)\offspring(k+1)}{\sum_{j=0}^\infty j \offspring(j)}, \qquad \text{for } k\in\natzero.
        \end{equation}

        \begin{definition}
        \label{def:UGW}
The unimodular Galton-Watson tree with offspring distribution $\offspring$,  denoted by UGW$(\offspring$), is a rooted random tree where the root has a number of children distributed like $\offspring$ and every vertex in subsequent generations has a number of children distributed like $\sizeb$, independently of the degree of vertices in the same or previous generations. 
\end{definition} 

It is well known that if $\{G_n\}_{n\in\N}$ and $\offspring$ satisfy Assumption \ref{assu:graph_sequence}, then $G_n$ converges in a local sense (\textit{local weak convergence in probability}, as defined in \cite[Definition 2.11]{vanderHofstad2023vol2}; see also \cite[Definition 2.2]{Lacker2023local}  and \cite[Section 2.4]{Ganguly2022hydrodynamic}) to a GWT$(\offspring)$ tree. This is established, for instance,  in \cite[Theorem 2.18 and Theorem 4.1]{van2009randomII}.  

%%%%%%%%%%%%%%%%%%%%%%%%%%%%%%%%%%%%%%%%%%%%%%%%
\subsection{Asymptotic Characterization of SIR dynamics}\label{sec:ODEs}
Our first result is the limit characterization (as the graph size goes to infinity)  of the evolution of the fractions of individuals that, at each time, are in each of the states $\{S,I,R\}$. Given a finite graph $G$, for $t\in[0,\infty)$ we define 
\begin{align}\label{def:fractions}
    \begin{split}
        & s^{G}(t):=\frac{1}{|G|}\sum_{v\in G} \one_{\set{X^G_v(t)=S}},
        \\ & i^{G}(t):=\frac{1}{|G|}\sum_{v\in G} \one_{\set{X^G_v(t)=I}}.
    \end{split}
\end{align}
 We start by establishing the existence and uniqueness of the solution to a certain system of ODEs that will be used to describe the limit.    As is standard practice, we use the dot notation for derivatives with respect to time, and prime notation for derivatives in space. 
\begin{proposition}\label{prop:f_If_SF_I}
    Suppose that $\offspring\in\Pmc(\natzero)$ has finite third moment and let $s_0\in(0,1)$. Then there exists a unique solution $(f_S,\ f_I, \ F_I)$ to the following system of ODEs:
    
\begin{equation} \label{eq:f_If_SF_I}
    \begin{cases}
         \dot f_S=f_Sf_I{\beta} \left(1-\frac{\sum_{k=0}^\infty k \sizeb(k)e^{-k F_I }}{\sum_{j=0}^\infty\sizeb(j) e^{-j  F_I}}\right),
        \\ \dot f_I= f_Sf_I {\beta} \frac{\sum_{k=0}^\infty k \sizeb(k)e^{-k F_I }}{\sum_{j=0}^\infty\sizeb(j) e^{-jF_I}} - f_I({\rho}+{\beta}-{\beta} f_I),
        \\  \dot F_I =\beta f_I,
    \end{cases}
\end{equation}
    with initial conditions 
    \begin{equation} \label{eq:ffF-initial}
    \begin{cases}
         f_S(0)= s_0,
        \\  f_I(0)= 1-s_0,
        \\  F_I(0) =0.
    \end{cases}
    \end{equation}
\end{proposition}
The proof of Proposition \ref{prop:f_If_SF_I} uses standard arguments and is thus  relegated to Appendix \ref{sec:odes-well-posed}. Given $\zeta \in\Pmc(\natzero)$ and $x\in (-\infty,0]$, we define its Laplace transform as follows: 
\begin{equation}\label{eq:mgf}
M_\zeta(x):=\sum_{k\in\natzero}\zeta(k)e^{kx}. 
\end{equation}
% that is, $M_\nu$ is the moment generating function of $\nu$.
 Given $f_S,\ f_I,\ F_I$ as in Proposition \ref{prop:f_If_SF_I}, for 
 $t \in [0,\infty)$ we define \begin{align}\label{eq:s_i_infinity}
        \begin{split}
            & s^{(\infty)}(t):=s_0 M_\offspring(-F_I(t))
            \\ & i^{(\infty)}(t):=e^{-\int_0^t\rho_u du} \left( i_0+ s_0\int_0^t M'_\offspring (-F_I(t)) e^{\int_0^u\rho_sds}\beta_u f_I(u)du \right).
        \end{split}
    \end{align}
We now state our main result for the SIR model. 

\begin{theorem} \label{thm:MAIN_SIR}
    Suppose that a sequence of random graphs $\{G_n\}_{n\in\N}$  and $\offspring\in\Pmc(\natzero)$ satisfy Assumption \ref{assu:graph_sequence}. Let $\sizeb$ be the size-biased version of $\offspring$, as defined in \eqref{eq:sizab}. Suppose that $s^{G_n}(0)\rightarrow s_0 \in (0,1)$ and let $s^{(\infty)}$ and $i^{(\infty)}$ be as defined in \eqref{eq:s_i_infinity}. Then, as  $n\rightarrow\infty $ we have

    \begin{align}
        \begin{split}
            & s^{G_n}(t)\xrightarrow{p}s^{(\infty)}(t),
        \\ &i^{G_n}(t) \xrightarrow{p} i^{(\infty)}(t),
        \end{split}
    \end{align}
    uniformly for $t\in[0,\infty)$.
\end{theorem}

The proof of Theorem \ref{thm:MAIN_SIR} is given in Section \ref{sec:proof_transient}. It relies on a hydrodynamic limit result established in \cite[Corollary 4.7]{Ganguly2022nonmarkovian}, which shows that the fraction of individuals in any state $a\in\stateS$ in the SIR process on $G_n$ converges to $\P(X_\root^\Tmc(t)=a)$, where $X^\Tmc$ is the SIR process on $\Tmc=\text{UGW}(\offspring)$, and $\root$ is the root vertex.
We then show that the trajectories of $X^{\Tmc}$ satisfy a certain conditional independence property (Proposition \ref{prop:cond_ind_prop}). We combine this property with symmetry properties of the dynamics (see Proposition \ref{prop:symmetry}) to characterize  $\law(X_\root^\Tmc)$ in terms of a system of ODEs. In particular, 
for $a=S$ or $a=I$,  the probability $\P(X^\Tmc_\root(t)=a)$ is equal to $s^{(\infty)}(t)$ or  $i^{(\infty)}(t)$, respectively,  as defined in  \eqref{eq:s_i_infinity}.  As mentioned in the Introduction, Proposition \ref{prop:cond_ind_prop} can be seen as a substantial refinement 
in the case of the SIR process $X^{\Tmc}$ of a certain general Markov random field property that holds for more general interacting particle systems; see  \cite[Theorem 3.7]{Ganguly2022interacting}.

In Figure \ref{fig:ODEs}, we compare simulations of the evolution of the SIR process on certain \Erdos \ random graphs and random $3$-regular graphs of size $n=250$ with the theoretical prediction from Theorem \ref{thm:MAIN_SIR}. The plots illustrate that even in systems of moderate size, the theoretical prediction closely tracks the simulations.
\begin{figure}
     \centering
     \begin{subfigure}[b]{0.49\textwidth}
         \centering
         \includegraphics[width=\textwidth]{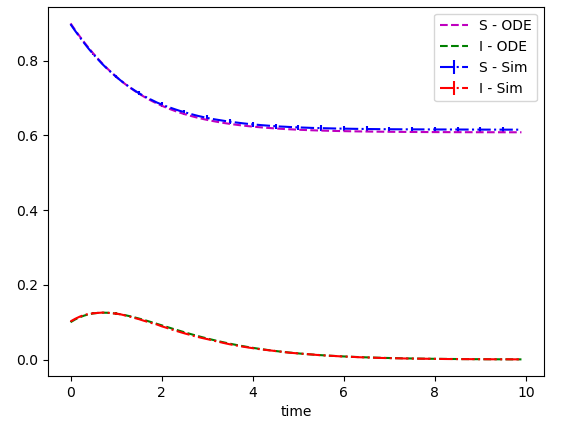}
         \caption{${\beta} \equiv 1$,   $\text{ER}(n,2/n)$.}
        
     \end{subfigure}
     \hfill
     \begin{subfigure}[b]{0.49\textwidth}
         \centering
         \includegraphics[width=\textwidth]{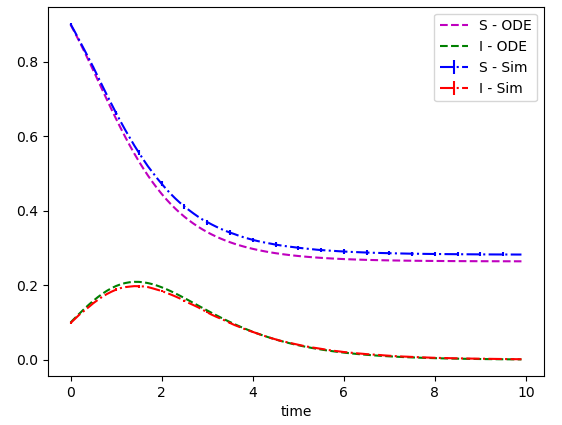}
         \caption{${\beta} \equiv 0.5$,   $\text{ER}(n,5/n)$.}
        
     \end{subfigure}
     \hfill
     \begin{subfigure}[b]{0.49\textwidth}
         \centering
         \includegraphics[width=\textwidth]{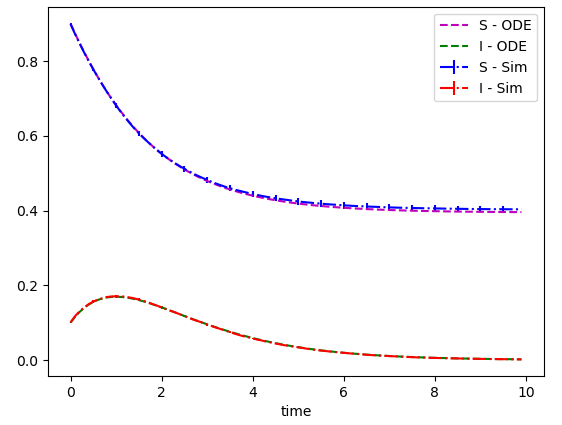}
         \caption{${\beta}\equiv1$,  $\text{CM}_n(\delta_3)$.}
     \end{subfigure}
     \hfill
     \begin{subfigure}[b]{0.49\textwidth}
         \centering
         \includegraphics[width=\textwidth]{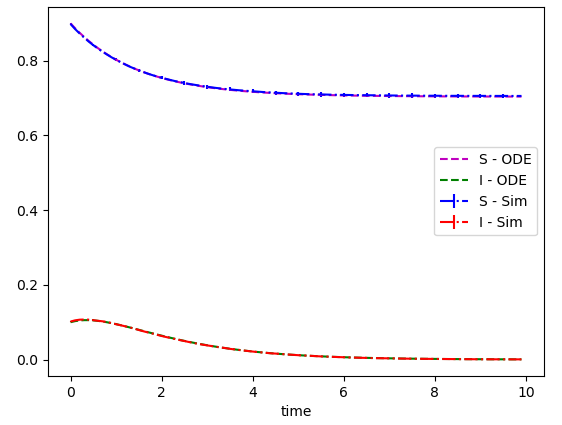}
         \caption{ ${\beta}\equiv0.5$,  $\text{CM}_n(\delta_3)$.}
     \end{subfigure}
        \caption{Time evolution of the fraction of susceptible ($S$) and infected ($I$) individuals on a finite graph ($n=250$), with constant $\beta$, obtained through simulations (Sim), along with the asymptotically exact values (ODE) given by Theorem \ref{thm:MAIN_SIR}. The simulations are obtained through $500$ iterations, resampling the random graphs at each iteration, and are plotted with $95\%$ confidence intervals.}
        \label{fig:ODEs}
\end{figure}  

\begin{remark}\label{remark:iid} 
    For simplicity, we restrict our attention to i.i.d. initial conditions, though the techniques in our proofs extend to more general initial conditions, as long as they satisfy certain symmetry properties between the laws of the initial states and that of the random graphs, and satisfy the Markov random field property mentioned above.  In the case where the limit tree $\Tmc$ is the $\kappa$-regular tree $T_\kappa$, the symmetry conditions correspond to the law of $X^{T_\kappa}(0)$ being isomorphism invariant, see \cite[Remark 3.16]{Lacker2021marginal}. 
\end{remark}

\begin{remark}\label{remark:broader_class_G}
    We also mention that, while Theorem \ref{thm:MAIN_SIR} is stated for (sparse) ER and CM graphs, our techniques extend to a broader class of graphs, namely to any graph sequence $\{G_n\}_{n\in\N}$  that converges \textit{locally weakly in probability} to a \text{UGW} tree. All results in this paper hold if we replace Assumption \ref{assu:graph_sequence} with the assumption that for some $\offspring\in\Pmc(\natzero)$ with finite third moment and a $\text{UGW}(\offspring)$ tree  $\Tmc$, 
$$ \frac{1}{n}\sum_{v\in G_v} \one_{\{B^{G_n}_r(v) \simeq H\} } \xrightarrow[n\rightarrow\infty]{p}\P(B^\Tmc_r(\root) \simeq H))$$
for every $r\in\natzero$ and every rooted graph $H$, where $\simeq$ denotes graph isomorphism, and  $B_r^G(v)$ is a ball of radius $r$ around $v\in G$, that is, the subgraph induced by all vertices in $G$ that are at most $r$ edges away from $v$.
\end{remark}

As mentioned in the Introduction, in the special case when the infection and recovery rates $\beta$ and $\rho$ are constant in time and $G_n$ is the configuration model, an 
ODE approximation similar to \eqref{eq:f_If_SF_I}  was proposed in \cite{Volz2008SIR, Volz2009epidemic} and shown to be asymptotically 
exact in \cite{Decreusefond2012large,janson2015law}.    However, Theorem \ref{thm:MAIN_SIR} applies to the more general setting of time-varying rates, which is very relevant for applications, e.g., \cite{chen2020time,hong2020estimation,london1973recurrent,dushoff2004dynamical}, and more general graph classes (see Remark \ref{remark:broader_class_G}). Further, an advantage of our approach is that it allows for several important generalizations, including non-exponential recovery times, as elaborated upon in Remark \ref{rmk:generalize_times} below, the SEIR model, presented in Section \ref{sec:SEIR}, and further extensions, discussed in Remark \ref{rmk:general_process} below. 

\begin{remark}\label{rmk:generalize_times}
     A large part of the literature on the SIR process focuses on the case where recovery times are exponential random variables, that is, each individual recovers at some rate $\rho$ regardless of how long they have been infected, and the methods exploit this Markovian structure. If recovery times are not exponential, the resulting SIR dynamics are not Markov, and this makes their analysis significantly more challenging. In contrast, the local convergence tools that we used in the proof of Theorem \ref{thm:MAIN_SIR} can still be used in this setting. Specifically, the hydrodynamic result in \cite{Ganguly2022hydrodynamic} is still valid and shows that the fraction of individuals in each of the SIR states on a finite locally-tree like graph can be approximated by the root particle dynamics of the non-Markovian SIR process on the infinite tree. Further, a version of the conditional independence property of Proposition \ref{prop:cond_ind_prop} can be established, the marginal root dynamics can be characterized as a piecewise deterministic Markov process, and its law characterized as the solution to a certain PDE. 
A complete analysis is deferred to future work.
\end{remark}

\subsection{SEIR Model}\label{sec:SEIR}
 In this section, we extend our limit results to the  Susceptible-Exposed-Infected-Recovered (SEIR) process.
The SEIR process is a model of epidemics in which each individual can be in one of four possible states: in addition to the three states $S,\ I, \ R$, of the SIR model, an individual can also be in the exposed state $E$, when it has contracted the disease but is not yet able to infect its neighbors.  

We define $\stateSS:=\{S,E,I,R\}$.
As in the case of the SIR model, the SEIR model on a (possibly random) graph $G$ can be modelled as a locally interacting Markov chain. We denote this process by $\Xbar^G$. 
The SEIR process is governed by the graph $G$ and three functions  $\beta,\rho,\lambda :[0,\infty)\rightarrow(0,\infty)$,  with $\beta$ and $\rho$, as for the SIR model, representing the infection and recovery rates, and $\lambda$ now representing the time-dependent rate at which an individual transitions from having been exposed to being infectious. We assume that the initial states are i.i.d. with $\P(\Xbar_v^G(0)=S)=s_0,$ $  \P(\Xbar_v^G(0)=E)=e_0$ and $\P(\Xbar_v^G(0)=I)=i_0$ for some $s_0\in(0,1)$ and $e_0,i_0\in[0,1]$ such that $s_0+e_0+i_0=1$.
At time $t$, an individual $v$ jumps from $S$ to $E$ at the rate $\beta_t \infects(\Xbar^G_{\partial_v}(t-))=\beta_t \sum_{w\in\partial_v}\one_{\{\Xbar^G_{w}(t-)=I\}}$, from $E$ to $I$ at the rate $\lambda_t$, and from $I$ to $R$ at the rate $\rho_t$. No other jumps are possible. Equipping $\stateSS$ with the ordering $S<E<I<R$, the SEIR process is non-decreasing in the same sense as Remark \ref{rmk:monotonicty}.

Throughout the rest of the paper, we make the following assumption.

\begin{assumption}\label{assu:beta_lambda_rho} The functions $\beta$, $\lambda$ and $\rho$ are continuous and there exist constants $c_1,c_2\in(0,\infty)$ such that  
    \begin{equation}
        c_1< \liminf_{t\rightarrow\infty} \min(\beta_t,\ \rho_t,\ \lambda_t)<\limsup_{t\rightarrow\infty} \max(\beta_t,\ \rho_t,\ \lambda_t)< c_2.
    \end{equation} 
\end{assumption}

\subsubsection{Asymptotic Characterization of SEIR dynamics}\label{sec:SEIR-statement_dynamics}
Given a finite graph $G$, we let
\begin{align}
    \begin{split}
        \sbar^{G}(t)&:= \frac{1}{|G|}\sum_{v\in G} \one_{\{\Xbar_v^{G}(t)=S\}},
        \\ \ebar^{G}(t)&:= \frac{1}{|G}\sum_{v\in G} \one_{\{\Xbar_v^{G}(t)=E\}},
        \\ \ibar^{G}(t) &:= \frac{1}{|G|}\sum_{v\in G} \one_{\{\Xbar_v^{G}(t)=I\}}. 
    \end{split}
\end{align}
We start by establishing the existence and uniqueness of the solution to a certain system of ordinary differential equations that we use in our main result. 
\begin{proposition}\label{prop:gggG}
    Suppose that $\offspring\in\Pmc(\natzero)$ has a finite third moment and let $s_0\in(0,1)$ and $e_0,i_0\in[0,1]$ satisfy $s_0+e_0+i_0=1$. Then there exists a unique solution $(g_S,\ g_E,\  g_I, \ G_I)$ to the following system of ODEs:
 \begin{align} \label{eq:gggG-Odes}
        \begin{cases}
             \dot g_S =  {\beta} g_S g_I\left(1- \frac{\sum_{k=0}^\infty k \sizeb(k)e^{- k G_I}}{\sum_{j=0}^\infty \sizeb(j)e^{- j G_I}}\right),
            \\ \dot g_E = {\beta} g_S g_I \frac{\sum_{k=0}^\infty k \sizeb(k)e^{- k G_I}}{\sum_{j=0}^\infty \sizeb(j)e^{- j G_I}} -g_E({\lambda} -{\beta} g_I),
            \\ \dot g_I= {\lambda} g_E - g_I({\rho}+{\beta}-{\beta} g_I),
            \\ \dot G_I= \beta g_I,
        \end{cases}
    \end{align}
with initial conditions
    \begin{align} \label{eq:gggG-Odes_initial}
        \begin{cases}
            G_I(0)=0,
            \\ g_m(0) =  s_0\one_{\{m=S\}} + e_0\one_{\{m=E\}} + i_0\one_{\{m=I\}}, & m\in\stateSS.
        \end{cases}
    \end{align}

\end{proposition}
The proof of Proposition \ref{prop:gggG} is similar to that of Proposition \ref{prop:f_If_SF_I}. A brief outline is given at the end of Appendix \ref{sec:odes-well-posed}.

 Given $g_S,\ g_E,\ g_I,\ G_I$ as in Proposition \ref{prop:gggG} and $M_\offspring$ as in \eqref{eq:mgf},  define \begin{align}\label{eq:sbar_infty}
        \begin{split} 
            & \sbar^{(\infty)}(t):=s_0 M_\offspring(-G_I(t)),
            \\& \ebar^{(\infty)}(t):= e^{-\int_0^t \lambda_udu}\left(e_0+ s_0\int_0^t M'_\offspring(-G_I(u)) G_I'(u) e^{\int_0^u \lambda_\tau d\tau }du\right),
            \\ & \ibar^{(\infty)}(t):= e^{-\int_0^t\rho_udu}\left( i_0 + \int_0^t\lambda_ue^{\int_0^u (\rho_s-\lambda_s)ds}\left(e_0+s_0\int_0^uM'_\offspring(-G_I(\tau)) G_I'(\tau)) e^{\int_0^\tau \lambda_sds}d\tau \right)du\right).
        \end{split} 
    \end{align}
We can now state our characterization of the large $n$ dynamics of the SEIR process.

\begin{theorem} \label{thm:SEIR_main}
    Suppose that the sequence of random graphs $\{G_n\}_{n\in\N}$ and $\offspring\in\P(\natzero)$ satisfy Assumption \ref{assu:graph_sequence}.  Let $\sizeb$ be the size-biased version of $\offspring$, as defined in \eqref{eq:sizab}, suppose $\sbar^{G_n}(0)\rightarrow s_0,$ $\ebar^{G_n}(0)\rightarrow e_0$, and $\ibar^{G_n}(0)\rightarrow i_0,$ with $s_0\in(0,1)$ and $s_0+e_0+i_0=1$, and set $\sbar^{(\infty)},$  $\ebar^{(\infty)}$ and $\ibar^{(\infty)}$ be as defined in \eqref{eq:sbar_infty}.  Then, as $n\rightarrow\infty $,
    \begin{align*}
        \begin{split}
             \sbar^{G_n}(t) \xrightarrow{p}\sbar^{(\infty)}(t),
       \qquad \ebar^{G_n}(t) \xrightarrow{p}\ebar^{(\infty)}(t), \qquad \ibar^{G_n}(t) \xrightarrow{p} \ibar^{(\infty)}(t),
        \end{split}
    \end{align*}
    uniformly   for $t\in [0,\infty)$.
\end{theorem}
The proof of Theorem \ref{thm:SEIR_main} is given in Section \ref{sec:SEIR-main-proof}, and follows a similar approach as for the SIR model, although the details are more involved. 
% and partly relegated to Appendix \ref{sec:proof-SEIR}.

In Figure \ref{fig:SEIR} we compare our asymptotically exact approximation to values of $\sbar^{G_n}, \ \ebar^{G_n}$ and $\ibar^{G_n}$ for an \Erdos \  graph obtained by Monte Carlo simulations ($500$ iterations, plotted with $95\%$ confidence intervals). Once again, our approximation closely tracks the simulation results, even for relatively small $n$. 

\begin{figure}
     \centering
     \begin{subfigure}[b]{0.49\textwidth}
         \centering
         \includegraphics[width=\textwidth]{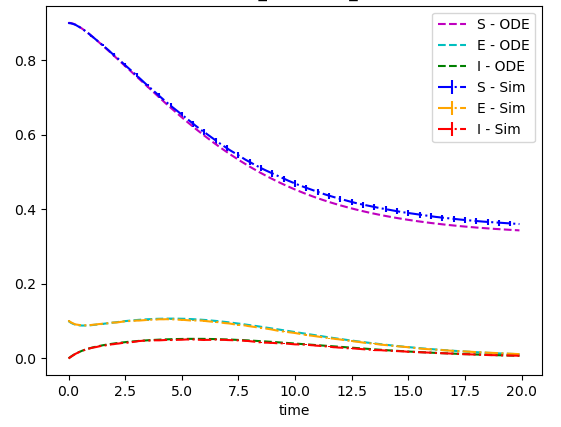}
         \caption{$n=200$, $\lambda\equiv 0.5$.}
        
     \end{subfigure}
     \begin{subfigure}[b]{0.49\textwidth}
         \centering
    \includegraphics[width=\textwidth]{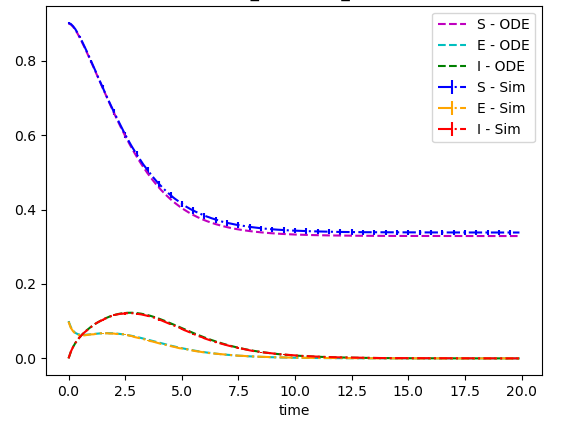}
         \caption{$n=200 $, $\lambda \equiv2$.}
     \end{subfigure}
     \centering
     \begin{subfigure}[b]{0.49\textwidth}
         \centering
         \includegraphics[width=\textwidth]{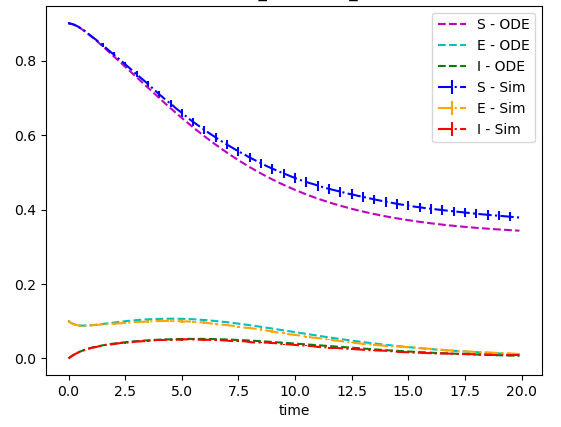}
         \caption{$n=100$, $\lambda \equiv 0.5$.}
        
     \end{subfigure}
     \begin{subfigure}[b]{0.49\textwidth}
         \centering
    \includegraphics[width=\textwidth]{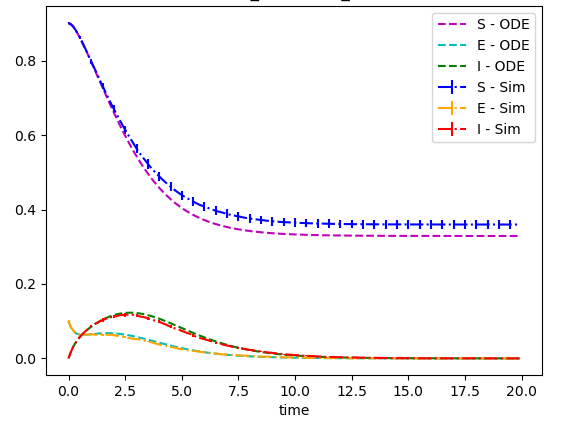}
         \caption{$n=100 $, $\lambda \equiv2$.}
     \end{subfigure}
        \caption{Time evolution of the fraction of susceptible (S), exposed (E) and infected (I) individuals on an $\text{ER}(n,3/n)$ graph, with  $\beta \equiv \rho \equiv  1$. We compare the asymptotically exact values (ODE) given by Theorem \ref{thm:SEIR_main} with the fraction obtained through Monte Carlo simulations (Sim). Simulations are obtained through $500$ iterations, and are shows with $95\%$ confidence intervals.}
        \label{fig:SEIR}
\end{figure}  

\begin{remark} \label{rmk:general_process}
The result in Theorem \ref{thm:SEIR_main} can be further extended to more general compartmental models that are widely used in the epidemiology literature in order to account for different viral strains and treatment options, for example, see \cite{duchamps2023general,foutel2022from,he2020SEIR,mwalili2020SEIR,hyman1999differential}. These allow for a susceptible state $S$ and   $m\in\N$  post-infection states $\{I_{1}$, $I_2$, … $I_m\}$. Supposing that each individual's transitions among post-infection states do not depend on the states of its neighbors, under Assumption \ref{assu:graph_sequence} and continuity assumptions analogous to Assumption \ref{assu:beta_lambda_rho}, the hydrodynamic result in \cite{Ganguly2022hydrodynamic} holds. If in addition one assumes that no transitions from post-infection states to state $S$ are possible, a version of the independence property of Proposition \ref{prop:cond_ind_prop} can be established, thus leading to a result analogous to Theorem \ref{thm:SEIR_main}. We defer a full account of this general setting to future work. 
\end{remark}

\section{Results on Outbreak Size}

 \label{sec:outbreak}
An important quantity of interest in the study of epidemic dynamics is the outbreak size, which is the fraction of individuals ever infected, in the interval $[0,\infty)$. By the monotonicity of the SIR and SEIR processes (Remark \ref{rmk:monotonicty}), the outbreak size is equal to  $1$ minus the limit, as $t\rightarrow\infty$, of the fraction of susceptible individuals at time $t$. In Section \ref{sec:SIR-outbreak} and Section \ref{sec:outbreak-SEIR}, we characterize the large-time behavior for the SIR and SEIR processes respectively, as the size of the population approaches infinity. In Section \ref{sec:mf-compare} we compare our asymptotically exact estimate of the outbreak size with a mean-field approximation for the special case of the SIR process on random regular graphs with constant infection and recovery rates.
\subsection{Outbreak Size for SIR Model}
\label{sec:SIR-outbreak} 
 Given a sequence of graphs $\{G_n\}_{n\in\N}$ satisfying Assumption \ref{assu:graph_sequence}, we let $s^{G_n}(\infty):=\lim_{t\rightarrow\infty}s^{G_n}(t)$ for $n\in\N$. We compute the limit of this quantity as $n\rightarrow\infty$, by first showing that $\lim_{n\rightarrow\infty}s^{G_n}(\infty)=\lim_{t\rightarrow\infty}s^{(\infty)}(t)$, where $s^{(\infty)}$, given in \eqref{eq:s_i_infinity}, is the hydrodynamic limit of the fraction of susceptible individuals, by Theorem \ref{thm:MAIN_SIR}. We recall that $M_\nu$ denotes the moment generating function of $\nu\in\Pmc(\natzero)$.

\begin{theorem}\label{thm:SIR-outbreak}
        Let $\{G_n\}_{n\in\N}$ and $\offspring\in\Pmc(\natzero)$ satisfy Assumption \ref{assu:graph_sequence}. Let $\sizeb$ be the size-biased version of $\offspring$, as defined in \eqref{eq:sizab}.  Then, assuming that $\lim_{n\rightarrow\infty}s^{G_n}(0)=s_0\in(0,1)$,
        \begin{equation*}
\lim_{n\rightarrow\infty}\lim_{t\rightarrow\infty}s^{G_n}(t)=\lim_{t\rightarrow\infty}s^{(\infty)}(t)= s_0M_\offspring\left(-\int_0^\infty \beta_u f_I(u)du\right),
        \end{equation*}
        where $f_I$ is defined by \eqref{eq:f_If_SF_I}-\eqref{eq:ffF-initial}. 
        Moreover, $\Finfty:= \int_0^\infty \beta_u f_I(u)du$ is finite and satisfies
    \begin{equation}\label{eq:Finfty_time_vary}
        \Finfty+\log(M_{\sizeb}(-\Finfty))-\log\left(1 -e^{\Finfty}\int_0^\infty e^{-\int_0^u \beta_\tau f_I(\tau) d\tau} \rho_u f_I(u) du\right) +\log(s_0)=0,
    \end{equation}
   Furthermore, if there exists $r\in(0,\infty)$ such that $\rho_t/\beta_t=r$ for all $t\in[0,\infty)$, then 
    equation \eqref{eq:Finfty_time_vary} 
    is equivalent to
    \begin{equation}\label{eq:Finfty_const}
        \Finfty+\log(M_{\sizeb}(-\Finfty))-\log(1+r(1-e^{\Finfty}) )+\log(s_0)=0,
    \end{equation}
    which has a unique strictly positive solution $\Fmc$.
        \end{theorem}
        The proof of Theorem \ref{thm:SIR-outbreak} is given in Section \ref{sec:outbreak}.
 \begin{remark}
 When the ratio $\rho_t/\beta_t$ is constant in time, the final outbreak size depends on $\rho$ and $\beta$ only through their ratio. This is well known when $\beta$ and $\rho$ are both constant, and in that case it is common in the SIR literature to fix $\rho\equiv1$ with no loss of generality, by re-scaling time. Theorem \ref{thm:SIR-outbreak} shows that, when the ratio  $\rho/\beta$ is not constant, the ratio no longer determines the outbreak size, and instead  the time evolution of both $\beta$ and $\rho$ influence the outbreak size.
Figure \ref{fig:ratios} illustrates this phenomenon. It plots $s^{(\infty)}(t)$, defined in \eqref{eq:s_i_infinity}, which by Theorem \ref{thm:MAIN_SIR} is the  large-$n$ asymptotic fraction of susceptible individuals, for two SIR processes with the same ratio $r_t=\rho_t/\beta_t$ for all $t\geq 0$, though different $\beta$ and $\rho$, which lead to dramatically different outbreak sizes.
 \end{remark}

\begin{figure}
    \floatbox[{\capbeside\thisfloatsetup{capbesideposition={right,center},capbesidewidth=0.5\textwidth}}]{figure}[\FBwidth]
{\caption{Large-$n$ limit of the fraction of susceptible individuals (i.e.,  $s^{(\infty)}(t)$ from \eqref{eq:s_i_infinity}) for the SIR process on a uniform $3$-regular graph,  as a function of time. For either curves, the ratio $\rho_t/\beta_t$ is equal to $1.5+sin(\pi t)$. In one case, $\rho_t$ increases linearly from $0.5$ to $1.5$ for $t\in[0,10]$ and it then stays constant. In the other, $\rho_t$  decreases linearly from $1.5$ to $0.5$ for $t\in[0,10]$, and it then stays constant.}
        \label{fig:ratios}}
{\includegraphics[width=0.5\textwidth]{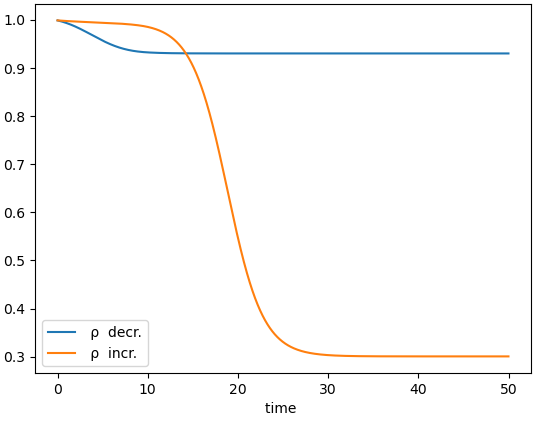}}\
\end{figure}

 Next, for each time-dependent $\beta$ and $\rho$ we identify
  constant infection and recovery rates that lead to the same outbreak size. These effective rates are unique only up to multiplication by the same constant, and so we identify them by their ratio.
 For given $\beta,\ \rho$, we define $\Psi_{\beta,\rho}:[0,\infty)\rightarrow[0,\infty]$ as 
 \begin{equation}\label{eq:Psi_beta_rho}
     \Psi_{\beta,\rho}(z)=z+\log(M_{\sizeb}(-z))-\log\left(1 -e^{z}\int_0^\infty e^{-\int_0^u \beta_\tau f_I(\tau) d\tau} \rho_u f_I(u) du\right) +\log(s_0),
 \end{equation}
 where we set $\log(x)=-\infty$  for any $x\leq 0$, and
where $f_I$ is defined by  \eqref{eq:f_If_SF_I}-\eqref{eq:ffF-initial} for some fixed $s_0\in(0,1)$ with the given $\beta$ and $\rho$. For $r\in(0,\infty)$, we also define 
$\Psi_{r}:[0,\infty)\rightarrow[0,\infty]$ as 
 \begin{equation}\label{eq:Psi_r}
     \Psi_{r}(z)=z+\log(M_{\sizeb}(-z))-\log\left(1 +r(1-e^z)\right) +\log(s_0).
 \end{equation}
 The following result shows that for every pair of rate functions $\beta$ and $\rho$ satisfying Assumption \ref{assu:beta_rho}, there exists a constant $\rhat:=\rhat(\beta,\rho)$ so that the outbreak size of an SIR process with rates $\beta$ and $\rho$, and that of a SIR process with constant infection rate $1$ and constant recovery $\rhat$ are the same (as $n\rightarrow\infty$). In particular, we observe that this is not achieved by naively replacing $\rho$ and $\beta$ with their respective average (over time) values, nor by taking $\rhat$ to be the (time) average of $\rho_t/\beta_t$.
\begin{lemma} \label{prop:effective_rate}
Let $\offspring\in\Pmc(\natzero)$ have a finite third moment, and suppose that $\rho,\ \beta$ satisfy Assumption \ref{assu:beta_rho}. Let $\Finfty_{\beta,\rho}:=\int_0^\infty \beta_t f_I(t)$, where $f_I$ is defined by \eqref{eq:f_If_SF_I}-\eqref{eq:ffF-initial}.  Then there exists a unique $\rhat\in(0,\infty)$ such that $\Psi_{\rhat}(\Finfty_{\beta,\rho})=0$, namely 
\begin{equation} \label{eq:rbar}
    \rhat=\frac{\int_0^\infty e^{-\int_0^t\beta_uf_I(u)du} \rho_t f_I(t) dt  }{ 1-e^{-\int_0^\infty \beta_u f_I(u)du} }.
\end{equation}
\end{lemma}
 \begin{proof} 
We start by showing that $\rhat\in(0,\infty)$. We know that $\Finfty_{\beta,\rho}= \int_0^\infty\beta_t f_I(t)dt<\infty$ by Theorem \ref{thm:SIR-outbreak}. By Assumption \ref{assu:beta_rho} and \eqref{eq:f_If_SF_I}, $t\rightarrow \beta_t f_I(t)$ is continuous, non-negative, and bounded away from zero near $t=0$, and so $\Finfty_{\beta,\rho}>0$. Letting $c_1,c_2\in(0,\infty)$ be constants such that \eqref{eq:assumption_bdd} holds, we have 
\begin{align*}
    \begin{split}
        \int_0^\infty e^{-\int_0^t\beta_uf_I(u)du} \rho_t f_I(t) dt  = \int_0^\infty e^{-\int_0^t\beta_uf_I(u)du} \frac{\rho_t}{\beta_t}\beta_t f_I(t) dt  \leq \frac{c_2}{c_1}(1-e^{-\Finfty_{\beta,\rho}})<\infty.
    \end{split}
\end{align*}
Similarly, note that $\int_0^\infty e^{-\int_0^t\beta_uf_I(u)du} \rho_t f_I(t) dt > c_1 c_2^{-1}(1-e^{-\Finfty_{\beta,\rho}})>0$.
Setting $\rhat$ as in \eqref{eq:rbar}, by \eqref{eq:Psi_r}, we have \begin{equation*}
    \Psi_{\rhat}(z)= z+\log(M_{\sizeb}(-z))-\log\left(1 + \frac{\int_0^\infty e^{-\int_0^t\beta_uf_I(u)du} \rho_t f_I(t) dt  }{ 1-e^{-\Finfty_{\beta,\rho}} }(1-e^z)\right) +\log(s_0).
\end{equation*}
Evaluating this at $z=\Finfty_{\beta,\rho}$ using \eqref{eq:Psi_beta_rho}, and observing  that $(1-e^z)/(1-e^{-z})=-e^{z}$, we have
\begin{align*}
    \begin{split}
      \Psi_{\rhat}(\Finfty_{\beta,\rho}) &= \Finfty_{\beta,\rho}+\log(M_{\sizeb}(-\Finfty_{\beta,\rho}))-\log\left(1 - e^{\Finfty_{\beta,\rho}}\int_0^\infty e^{-\int_0^t\beta_uf_I(u)du} \rho_t f_I(t) dt \right) +\log(s_0)\\&= \Psi_{\beta,\rho}(\Finfty_{\beta,\rho}),  
    \end{split}
\end{align*}
which is zero by Theorem \ref{thm:SIR-outbreak}. This shows the  existence of $\rhat\in(0,\infty)$ such that $\Psi_{\rhat}(\F_{\beta,\rho})=0$. 

For uniqueness, observe that for each $z\in(0,\infty)$ the map  $r\rightarrow\Psi_r(z)$ is non-decreasing, and strictly increasing on $\set{ r \ : \ \Psi_r(z)<\infty}$. Let $\Fmc_r$ be the unique zero of $\Psi_r$. It follows that $\Fmc_r$ is strictly decreasing in $r$ and therefore there is a one-to-one correspondence between $r$ and $\Fmc_r$.  
\end{proof}

We conclude this section with a brief discussion of periodic parameters. For simplicity, we fix $\rho\equiv 1$ and we consider periodic infection rates that could model, for instance, seasonality effects of the infectivity of a pathogen.   For a given amplitude $A\in [0,1)$, period $\omega >0$ and $\delta\in[0,1]$ we set $\beta_t=1+A\sin((t+\delta \omega) 2\pi/\omega)$. Here, $\delta$ is a parameter controlling the phase of the periodic rate at time zero. Note that if the period length is large enough compared to the average infection rate and recovery rate, the outbreak dies out before the full length of the period, and so, while the average of $\beta_t$ over the period is $1$, the average infection rate during the time the epidemic is “active” (i.e, there are individuals in state $I$) will be close to $\beta_0$. Because of this, we expect $\delta$ to have a greater impact on the outbreak size when $\omega$ is large. This is borne out by Figure \ref{fig:periods}, which plots the outbreak size as a function of $A$ for various $\delta$ and $\omega$. We see that in every case other than large $\omega$ and small $\delta$, the outbreak size is decreasing in $A$. This suggests the following conjecture, which we leave for future investigation.

\begin{figure}
     \centering
     \begin{subfigure}[b]{0.49\textwidth}
         \centering
         \includegraphics[width=\textwidth]{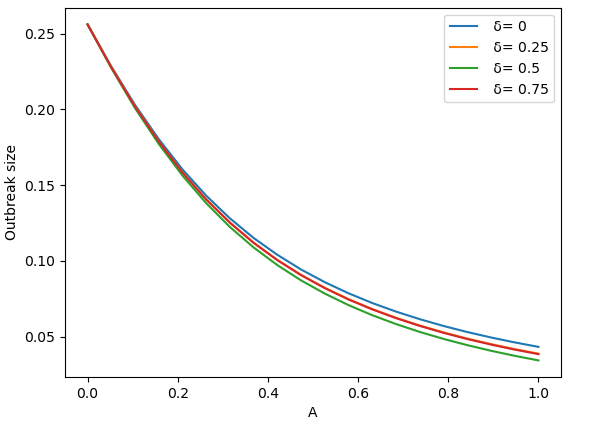}
        \caption{ $\omega=0.5$.}
        
     \end{subfigure}
     \hfill
     \begin{subfigure}[b]{0.49\textwidth}
         \centering
         \includegraphics[width=\textwidth]{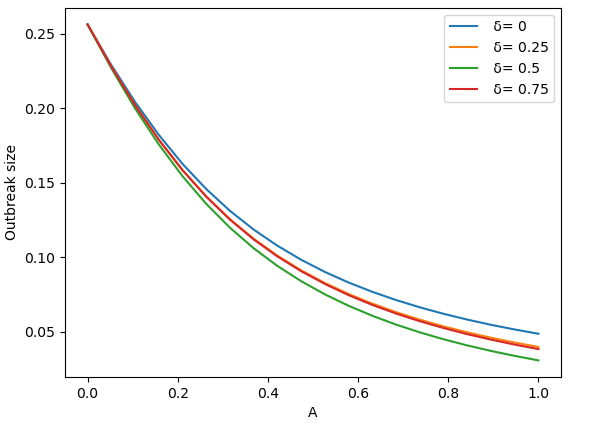}
         \caption{ $\omega=1$.}
        
     \end{subfigure}
     \hfill
     \begin{subfigure}[b]{0.49\textwidth}
         \centering
         \includegraphics[width=\textwidth]{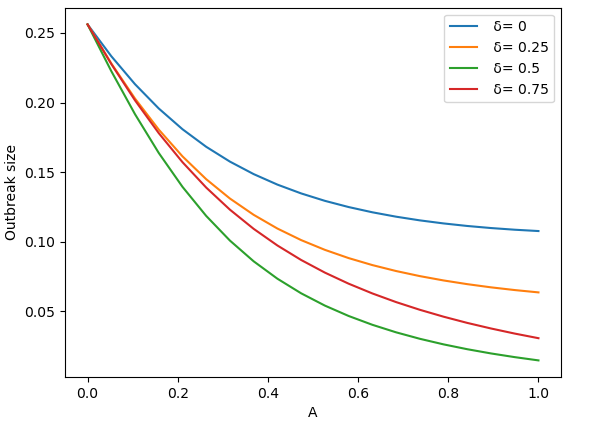}
       \caption{ $\omega=5$.}
     \end{subfigure}
     \hfill
     \begin{subfigure}[b]{0.49\textwidth}
         \centering
         \includegraphics[width=\textwidth]{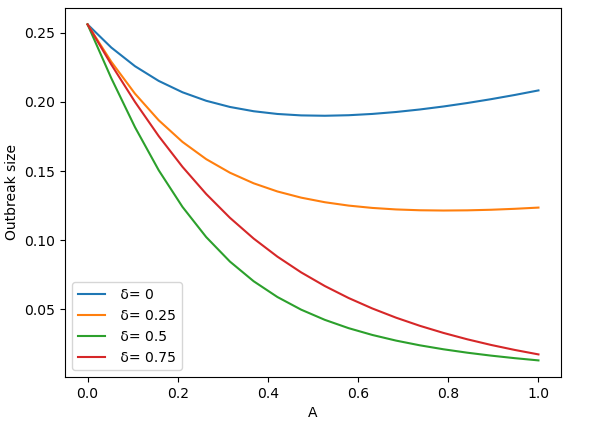}
         \caption{ $\omega=10$.}
     \end{subfigure}
        \caption{Large $n$ limit of the final size of a SIR outbreak on a $3$-regular graph with periodic infection rate, as a function of the amplitude $A$, obtained numerically from $s^{(\infty)}(t)$ as in Theorem \ref{thm:MAIN_SIR}. In all cases $\rho\equiv 1$. }
        \label{fig:periods}
\end{figure}  
\begin{conjecture} Let $A\in [0,1),\ \omega>0$ and $\delta\in[0,1]$. Define  $\beta_t=1+A\sin((t+\delta \omega) 2\pi/\omega)$. There exists $\omegabar>0$ such that, for all $\omega<\omegabar$, the asymptotic outbreak size $1-s^{(\infty)}(\infty)$ is decreasing in $A$.
\end{conjecture}

\subsection{Outbreak Size for the SEIR Model}\label{sec:outbreak-SEIR}

 We now turn to the characterization of the outbreak size of an SEIR process.  Recall the definition of  $M_\zeta$ for $\zeta\in\Pmc(\natzero)$ given in \eqref{eq:mgf}.

\begin{theorem}\label{thm:SEIR-outbreak}
    Let $\{G_n\}_{n\in\N}$ and $\offspring$ satisfy Assumption \ref{assu:graph_sequence}. Then

    \begin{equation*}
\lim_{n\rightarrow\infty}\lim_{t\rightarrow\infty}\sbar^{G_n}(t)=\lim_{t\rightarrow\infty}\sbar^{(\infty)}(t)= s_0M_\offspring\left(-\int_0^\infty \beta_u g_I(u)du\right)
    \end{equation*}
    where $\int_0^\infty \beta_u g_I(u)du=:\Finfty$ satisfies
    \begin{equation}\label{eq:SEIR-Finfty_time_vary}
        \Finfty+\log(M_{\sizeb}(-\Finfty))-\log\left(1 -e^{\Finfty}\int_0^\infty e^{-\int_0^u \beta_\tau g_I(\tau) d\tau} \rho_u g_I(u) du\right) +\log(s_0)=0,
    \end{equation}
    for $g_I$ as in \eqref{eq:gggG-Odes}.  
    Furthermore, if there exists $r\in(0,\infty)$ such that  $\rho_t/\beta_t,=r$ for all $t\in[0,\infty)$, 
    equation \eqref{eq:SEIR-Finfty_time_vary} 
    is equivalent to 
    \begin{equation}\label{eq:SEIR-Finfty}
        \Psi_r(\Fmc)=0,
    \end{equation}
    where $\Psi_r$ is defined in \eqref{eq:Psi_r}.
\end{theorem}

\begin{figure}
     \centering
     \begin{subfigure}[b]{0.49\textwidth}
   
         \centering
         \includegraphics[width=\textwidth]{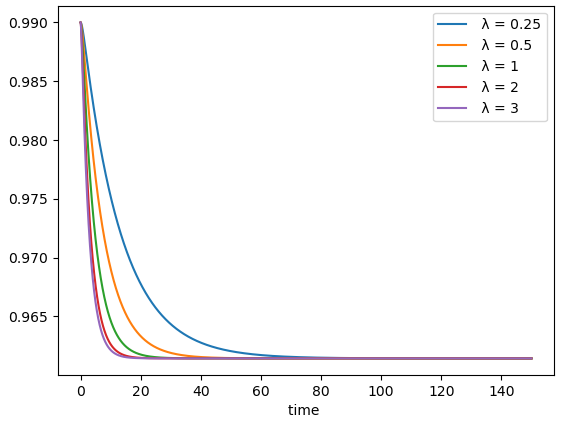}
         \caption{$\beta \equiv 0.5$, $\rho\equiv 1$.}
        
     \end{subfigure}
     \hfill
     \begin{subfigure}[b]{0.49\textwidth}
     
         \centering
         \includegraphics[width=\textwidth]{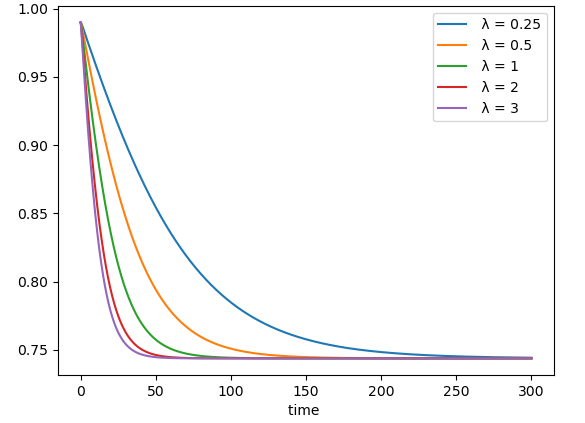}
         \caption{$\beta \equiv  1$,  $\rho\equiv 1$.}
     \end{subfigure}
      \begin{subfigure}[b]{0.49\textwidth}
         \centering
         \includegraphics[width=\textwidth]{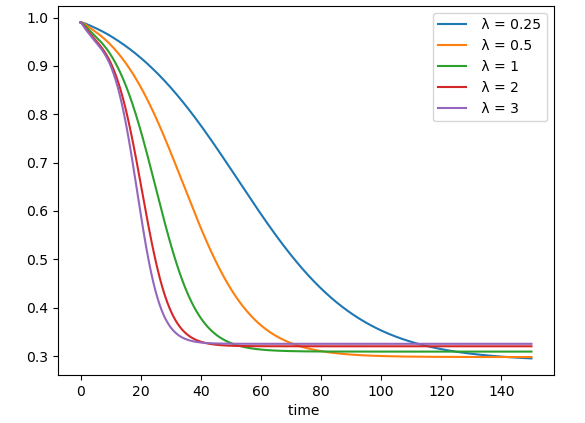}
         \caption{$\beta_t= 0.5+ 0.1\min(t,10)$, $\rho\equiv 1$.}
        
     \end{subfigure}
     \hfill
     \begin{subfigure}[b]{0.49\textwidth}
         \centering
         \includegraphics[width=\textwidth]{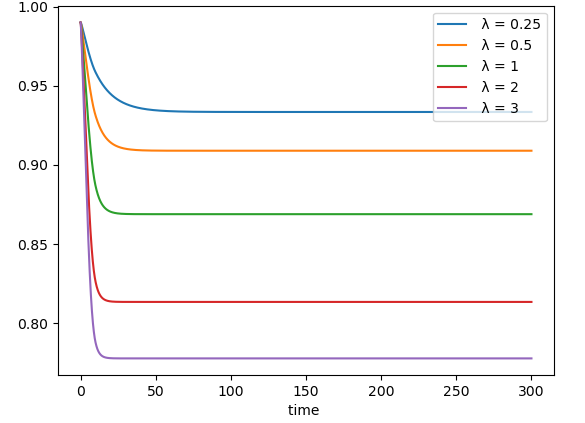}
         \caption{$\beta_t= 1.5- 0.1\max(t,10)$,  $\rho\equiv 1$.}
     \end{subfigure}
     
        \caption{Large $n$ limit fraction of susceptible individuals for SEIR processes ($\sbar^{(\infty)}(t)$ from \eqref{eq:sbar_infty}) with different values of $\lambda$. In all cases, the interaction graph is the  random $3$-regular graph, $e_0=0.01$ and $s_0=1-e_0$. In (a)  and (b) both $\beta$ and $\rho$ are constant. In (c) (resp. (d)) $\rho$ is held constant and $\beta$ grows linearly from $0.5$ to $1.5$ (resp. decreases linearly from $1.5$ to $0.5$) and is then constant.}
        \label{fig:SIER_oubtreaks}
\end{figure}  

Theorem \ref{thm:SEIR-outbreak} shows that when the ratio $t\mapsto \beta_t/\rho_t$ is constant, the final outbreak size does not depend on $\lambda$ and it coincides with the outbreak size of a SIR process with the same infection rate, recovery rate, and initial condition $s_0$. On the other hand, when the ratio is not constant, the rate $\lambda$ affects the outbreak size. Figure \ref{fig:SIER_oubtreaks} plots $\sbar^{(\infty)}(t)$, as defined in \eqref{eq:sbar_infty}, which by Theorem \ref{thm:SEIR_main} is the large $n$ limit of the  of the fraction of susceptible  individuals, for several SEIR processes on a random $3$-regular graph. For fixed constants $\beta,$ and $ \rho$, but different values of constant $\lambda$,   the time-dynamics can vary significantly, but the final fraction of susceptible individuals does not depend on $\lambda$. In contrast, when $\rho/\beta$ changes with time, the final fraction of susceptible individuals (and hence, the outbreak size) varies with $\lambda$.

In Figure \ref{fig:sEIr_vs_sIr}, we set all rates as constant, and we show that the time evolution of the sum of the fractions of infected and exposed individuals in the SEIR process can be markedly different from that of the fraction of infected individuals in an SIR process, despite the fact that the final outbreak sizes coincide.  
 We leave as a future research direction the problem of  understanding the impact of $\lambda$ on the SEIR dynamics for finite $t$ - for example, how does $\lambda$ impact the maximum number of individuals that have ever been infected in  any given time period?

\begin{figure}
     \centering
     \begin{subfigure}[b]{0.49\textwidth}
     
         \centering
         \includegraphics[width=\textwidth]{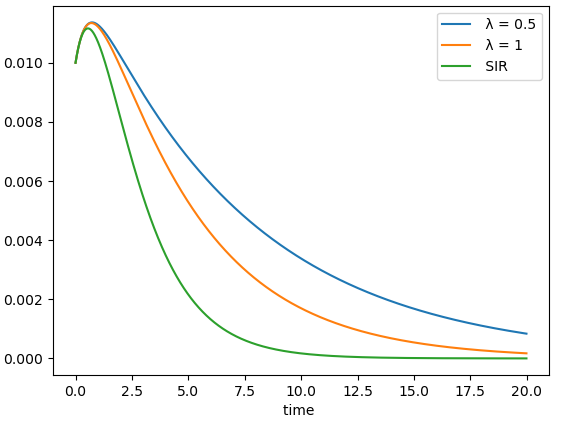}
         \caption{$\beta \equiv 0.5$, $\rho\equiv 1$.}
        
     \end{subfigure}
     \hfill
     \begin{subfigure}[b]{0.49\textwidth}

         \centering
         \includegraphics[width=\textwidth]{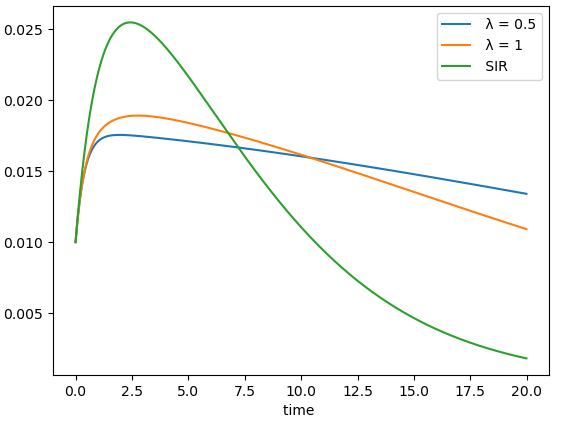}
         \caption{$\beta \equiv  1$,  $\rho\equiv 1$.}
     \end{subfigure}

        \caption{Large $n$ limit of the fraction of individuals who are exposed or infected in the SEIR process on the uniform $3$-regular (i.e., $\ebar^{(\infty)}(t)+\ibar^{(\infty)}(t)$ from \eqref{eq:sbar_infty}) for $\lambda \equiv 1$ and $\lambda\equiv0.5$, along with  the large $n$ limit of the fraction of infected individuals on the SIR process ($i^{(\infty)}(t)$ from \eqref{eq:s_i_infinity}) on the same graph and with same $\rho$ and $\beta$. Individuals are initially infected with probability $0.01$, and are otherwise susceptible.}
        \label{fig:sEIr_vs_sIr}
\end{figure}  

\subsection{Comparison with the Mean-Field approximation for the SIR model} \label{sec:mf-compare}

In this section, we restrict our attention to the SIR process on the uniform $\kappa$-regular graph, with the ratio $\rho/\beta$ being constant in time, and compare the asymptotically exact outbreak size with the corresponding  mean-field approximation. We first start by 
observing that on $\kappa$-regular graphs, the characterization \eqref{eq:Finfty_const} of the outbreak size can be simplified further as follows.

\begin{corollary} \label{cor:T_k-outbreak}
    Let $\kappa\in\N\setminus\{1\} $. Let $\{G_n\}_{n\in\N}$ be such that for every $n\in\N$, $G_n$ is chosen uniformly among all  $\kappa$-regular graphs with $n$ vertices, or equivalently  $G_n$ is a $\text{CM}_n(\delta_\kappa)$ graph. Suppose that there exists $r\in(0,\infty)$ such that $\rho_t/\beta_t=r$ for all $t\in[0,\infty)$.
    Then, it follows that
    
     \begin{equation*}
    \lim_{t\rightarrow\infty}s^{(\infty)}(t) = \sigma_\kappa.
\end{equation*} where $\sigma_\kappa\in(0,s_0]$ is the unique solution in $(0,s_0)$ of the equation
    \begin{equation}\label{eq:outbreak_Tk}
     \phi_\kappa(z):=z^{\frac{\kappa-2}{\kappa}}s_0^{\frac{2}{\kappa}}-(1+r)+rz^{-\frac{1}{\kappa}}s_0^{\frac{1}{\kappa}}=0.
\end{equation}
In particular, we have
    \begin{equation*}
        \sigma_2= \frac{s_0}{(1+r(1-s_0))^2}.
    \end{equation*}
\end{corollary}
 \begin{proof} 
Fix $\kappa\in\N\setminus\{1\}$ and set $\offspring=\delta_\kappa$. It is immediate from \eqref{eq:sizab} that the size-biased distribution $\sizeb$ is equal to $\delta_{\kappa-1}$. For any $k\in\N$, $u\in\R$, we have that $M_{\delta_k}(u)=e^{ku}$. By Theorem \ref{thm:SIR-outbreak}, the final fraction of susceptible individuals $s^{(\infty)}(\infty)$ is equal to $s_0 e^{-\kappa\Finfty}$, where $\Finfty$ is the solution of equation \eqref{eq:Finfty_const}, which for $\sizeb=\delta_{\kappa-1}$ reduces to
\begin{equation*}
    (2-\kappa)\Finfty -\log(1+r(1-e^\Finfty)) +\log(s_0)=0.
\end{equation*}
By a simple arithmetic manipulation, $\sigma_{\kappa}:=s_0 e^{-\kappa\Finfty}$ satisfies equation \eqref{eq:outbreak_Tk}. 
Uniqueness of the solution to $\phi_\kappa(z)=0$ follows since $\phi_\kappa'(z)=0$ holds for at most one value of $z$, namely for $z=(s_0^{-1/\kappa}r/(\kappa-2))^{\kappa/(\kappa-1)}$.
\quad   
 \end{proof}

Figure \ref{fig:outbreaks_n} plots the analytic values of $\sigma_\kappa$ obtained from Corollary
\ref{cor:T_k-outbreak} versus simulated values for different values of $n$ and $\kappa$. We ran $200$ iterations for each pair $(n,$ $\kappa)$, sampling a new graph at every iteration. 
As shown therein, the limit appears to be a good approximation for graphs of moderate size (namely, with $n\geq 150$). We leave for future  research the problem of finding accurate error bounds for finite $n$.

\begin{figure}
\floatbox[{\capbeside\thisfloatsetup{capbesideposition={right,center},capbesidewidth=0.5\textwidth}}]{figure}[\FBwidth]
{\caption{ Outbreak size for the SIR process with $\beta\equiv 0.5$ and $\rho\equiv 1$  obtained from Monte Carlo simulations on random $ \kappa$-regular graphs on $n$ vertices (with $95\%$ confidence intervals), compared with the asymptotically exact value given by Corollary \ref{cor:T_k-outbreak} (ODE). Simulations are obtained via $100$ iterations each, with resampling of the graph at each iteration.}
        \label{fig:outbreaks_n}}
{\includegraphics[width=0.5\textwidth]{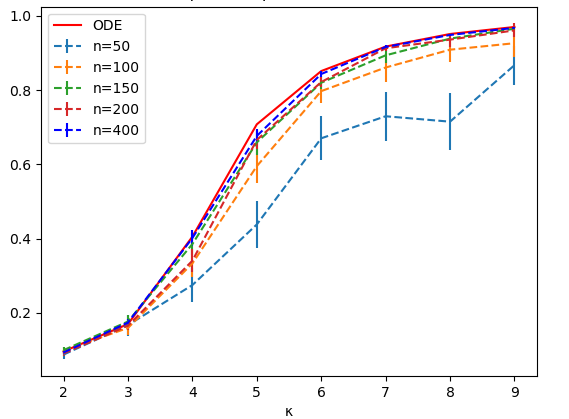}}\
\end{figure}

By Theorem \ref{thm:SIR-outbreak} and Corollary \ref{cor:T_k-outbreak}, $\sigma_\kappa$ is an asymptotically (in $n$) exact approximation of the total fraction of individuals ever infected on a SIR epidemic on a graph drawn uniformly among the $\kappa$-regular graphs on $n$ vertices.
We now compare this approximation with a scaled mean-field approximation to the SIR model on a $\kappa$-regular graph which can be formulated via the following system of ODEs, see for example \cite[Section 7]{Ganguly2022nonmarkovian}:
\begin{eqnarray*}
        \frac{ d \shat}{dt} (t) & = &  -\beta_t \kappa \shat(t)\ihat(t),
        \\ \frac{d \ihat(t)}{dt} & = &  \beta_t \kappa \shat (t) \ihat(t) -\rho_t \ihat(t),
\end{eqnarray*}
with initial conditions $\shat(0)=s_0$, $\ihat(0)=i_0=1-s_0$.  When there exists $r\in(0,\infty)$ such that $\rho_t/\beta_t=r$ for all $t$, by performing a change of variables and solving the equation $d \ihat / d \shat =-1 +\rho/ (\beta \kappa \shat)$, it can be shown  that $\lim_{t\rightarrow\infty}\shat(t)=\sigmahat_\kappa$ where $\sigmahat_\kappa=\sigmahat_\kappa(s_0)$ is the unique solution in $(0,s_0)$ of 

\begin{equation}\label{eq:mf-outbreak}
    \phihat_\kappa(\sigmahat_\kappa)=0, \qquad \text{with }\quad \phihat_\kappa(z):=s_0e^{{\frac{1}{r}\kappa(z-1)}}-z.
\end{equation}

\begin{figure}
     \centering
     \begin{subfigure}[b]{0.3\textwidth}
         \centering
         \includegraphics[width=\textwidth]{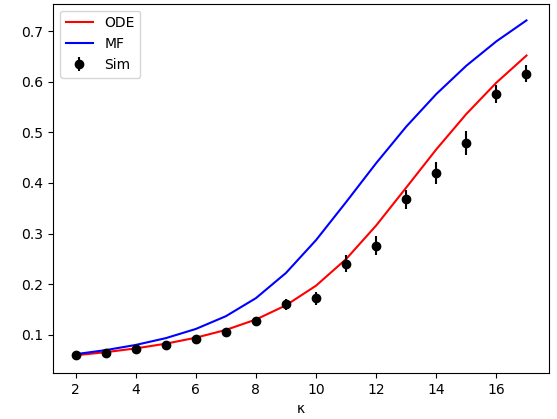}
         \caption{$\beta\equiv0.1$.}
        
     \end{subfigure}
     \hfill
     \begin{subfigure}[b]{0.3\textwidth}
         \centering
         \includegraphics[width=\textwidth]{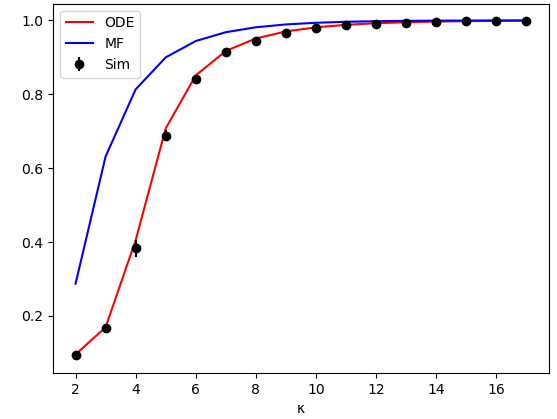}
         \caption{$\beta\equiv0.5$.}
     \end{subfigure}
     \hfill
     \begin{subfigure}[b]{0.3\textwidth}
         \centering
         \includegraphics[width=\textwidth]{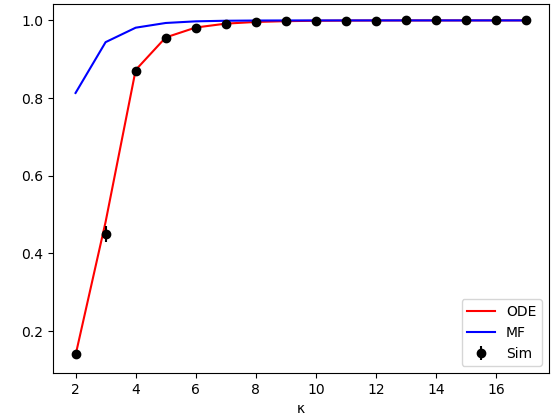}
         \caption{ $\beta\equiv1$.}
     \end{subfigure}
        \caption{Outbreak size on large ($n=400$ vertices) $\kappa$-regular configuration model graphs (Sim), with $\rho\equiv1$, and $s_0=0.9
        5$, along with the asymptotically exact value  $1-\sigma_\kappa$ from Corollary \ref{cor:T_k-outbreak}(ODE), and the mean-field approximation $1-\sigmahat_\kappa$ defined in \eqref{eq:mf-outbreak} (MF).  Simulations are obtained through $500$ iterations, resampling the graph at each iteration, and are plotted with $95\%$ confidence intervals.}
        \label{fig:s_infinity}
\end{figure}  

Our next result shows that the mean-field approximation always yields a larger estimate of the outbreak size on random regular graphs compared to the true asymptotic value. This is further illustrated in Figure \ref{fig:s_infinity}, which plots the mean-field prediction versus our prediction of the outbreak size on  random $\kappa$-regular graphs.

\begin{proposition} \label{prop:compare-meanfield}

Fix $s_0\in(0,1)$. For each $\kappa\in\N\setminus\{1\}$, let $\sigmahat_\kappa$ be as in \eqref{eq:mf-outbreak}, and $\sigma_\kappa$ be given as in Corollary \ref{cor:T_k-outbreak}. Then it follows that 
    \begin{equation} \label{eq:MF-vsLE}
         \sigmahat_\kappa < \sigma_\kappa.
    \end{equation}
    
\end{proposition}
 \begin{proof} 
Fix $s_0\in(0,1)$, $r>0$. 
Recall that $\sigma_\kappa$ is the unique solution in $(0,s_0)$ to $\phi_\kappa(z)=0$.
From the proof of  Corollary \ref{cor:T_k-outbreak}, we know that $\phi_\kappa(z)>0$ for $z\in(0,\sigma_\kappa)$  and $\phi_\kappa(z)<0$ for $z\in(\sigma_\kappa,s_0)$. Therefore, to show  \eqref{eq:MF-vsLE}, it is enough to show that $\phi_\kappa(\sigmahat_\kappa)>0$. Using the fact that $\sigmahat_\kappa= s_0 \exp(\kappa(\sigmahat_\kappa -1) / r)$, and that $e^z>1+z$ for every $z>0$, we have
    \begin{align}
        \begin{split}
\phi_\kappa(\sigmahat_\kappa)&=\phi_\kappa(s_0e^{\frac{\kappa}{r}(\sigmahat_\kappa-1)})
            \\&=  \left(s_0e^{\frac{\kappa}{r}(\sigmahat_\kappa-1)}\right)^{\frac{\kappa-2}{\kappa}}s_0^{\frac{2}{\kappa}} +rs_0^{1/\kappa} \left(s_0e^{\frac{\kappa}{r}(\sigmahat_\kappa-1)}\right)^{-1/\kappa} -1 -r
            \\ &= e^{\frac{1}{r}(\kappa-2)(\sigmahat_\kappa-1)}s_0 + re^{\frac{1}{r}(1-\sigmahat_\kappa)}-1-r
            \\ &> s_0 e^{\frac{\kappa}{r} (\sigmahat_\kappa-1)}(e^{2\frac{1}{r}(1-\sigmahat_\kappa)}-1).
        \end{split}
    \end{align}
Since the last expression is strictly positive, this concludes the proof.
\quad     \end{proof}

\section{Proofs of main results}\label{sec:main_proofs}
%%%%%%%%%%%%%%%%%%%%%%%%%%%%%%%%%%%%%%%%%%%%%%%%%%%%%%%%%%%%%%%%%%%%%%%%%%%%%%%%%%%%%%%%%%%%%%%%%%%%%%%%%%%%%%%%%%%%%%%%%%%%%%%%%%%%%%%%%%%%%%%%%%%%%%%%%%%%%%%%%%%%%%%%%%%%%%%%%%%%%%%%%%%%%%%%%%%%%%%%%%%%%%%%%%%%%%%%%%%%%%%%%%%%%%%%%%%%%%%%%%%%%%%%%%%%%%%%%%%%%%%%%
 In Section \ref{sec:S(E)IR} we introduce a parameterized family of processes that interpolates between the SIR and SEIR processes.  This allows us to prove some intermediate results simultaneously for both processes. In Section \ref{sec:proof_transient} we provide the proofs of Theorem \ref{thm:MAIN_SIR} and Theorem \ref{thm:SEIR_main}. In Section \ref{sec:proofs_outbreak} we prove Theorem \ref{thm:SIR-outbreak} and Theorem \ref{thm:SEIR-outbreak}.  Throughout,  $\{G_n\}_{n\in\N}$ is a sequence of random graphs, $\offspring\in\Pmc(\natzero)$,  $\sizeb$ is the size-biased version of $\offspring$, as defined in \eqref{eq:sizab}, and $\Tmc$ is a UGW($\offspring$) tree. We assume that $\{G_n\}_{n\in\N}$ and $\offspring$ satisfy Assumption \ref{assu:graph_sequence} and that the rates $\beta,$ $\lambda,$ $  \rho : [0,\infty)\rightarrow (0,\infty)$
satisfy Assumption \ref{assu:beta_lambda_rho}.

\subsection{The Hybrid S(E)IR Process}\label{sec:S(E)IR}
 Fix the rates $\beta,\rho,\lambda :[0,\infty)\rightarrow(0,\infty)$ as in Section \ref{sec:SEIR}, the interpolation parameter   $\alpha\in[0,1]$,  probabilities $s_0\in(0,1)$,  and $e_0,\ i_0 \in[0,1]$ with $s_0+e_0+i_0=1$.
 For a graph $G=(V,E)$, let $\xi^{{G,\alpha}}$  be a Markov chain on $\stateSS^V$ describing the evolution of interacting individuals or particles indexed by the nodes of $V$, where the state of each particle lies in the space $\stateSS=\{S,E,I,R\}$. The initial states are i.i.d. across particles, with common law $p_0$ given by
 \begin{equation}\label{eq:p_0}
     p_0(b)=s_0\one_{\set{b=S}}+ e_0\one_{\set{b=E}} + i_0\one_{\set{b=I}},
 \end{equation}
 $b\in\stateSS$.
 Given $y\in\stateSS^\ell$ for some $\ell\in\natzero$ (setting $\stateSS^0=\emptyset$) recall that $\infects(y)$ denotes the number of entries of $y$ that are equal to $I$.  At time $t\in[0,\infty)$, the jump rates for the  jump processes $\xi^{G,\alpha}_v$ representing the evolution of a particle $v\in G$ are given as follows:
 \begin{itemize}
     \item from $S$ to $E$ at rate $\alpha\beta_t\infects(\xi^{{G,\alpha}}_{\partial v}(t-))$;
     \item from $S$ to $I$ at rate   $(1-\alpha)\beta_t\infects(\xi^{{G,\alpha}}_{\partial v}(t-))$;
     \item from $E$ to $I$ at rate $\lambda_t$;
     \item from $I$ to $R$ at rate $\rho_t$.
 \end{itemize}
 No other transitions are allowed. 
 When $G$ is finite, classical results guarantee the  existence of the process $\xi^{G,\alpha}$ and the uniqueness of its law follows from standard results  about finite-state continuous time Markov chains, see for instance \cite[Proposition 4.1]{Ganguly2022hydrodynamic}.  We note that if $\alpha=1$, $\xi^{G,\alpha}$ reduces to the SEIR model, and if $\alpha=0$ (and $e_0=0$), $\xi^{G,\alpha}$ is the SIR model. Throughout, whenever $\alpha=0$, we implicitly assume that $e_0=0$, and Assumption \ref{assu:beta_lambda_rho} can be substituted with Assumption \ref{assu:beta_rho}. We refer to $\xi^{G,\alpha}$ as the S(E)IR process.
    We also observe that the process $\xi^{G,\alpha}$ is non-decreasing for every $G$ and $\alpha$, that is, for every $v\in G$ and $t,s\in[0,\infty)$,
\begin{equation}\label{eq:mono}
    t\leq s \ \Rightarrow \  \xi_v^{G,\alpha}(t)\leq \xi_v^{G,\alpha}(s).
\end{equation}

 Since we are interested in studying the limit of the S(E)IR process on locally tree-like graphs $G_n$ with $|G_n|\rightarrow\infty$ and $G_n$ converging to a limit random tree,  we need to define $\xi^{\Tmc,\alpha}$ on a possibly infinite tree $\Tmc$. 
Intuitively, $\xi^{\Tmc,\alpha}$ is a Markov jump process with the same rates as described above, but due to randomness in the tree structure, a rigorous definition (and subsequent characterization of properties) is most conveniently expressed in terms of the following (standard) Ulam-Harris-Neveu labeling  which identifies each realization of $\Tmc$ with a subgraph of the graph of all possible vertices. The latter has  vertex set $\V:=\{\root\} \cup (\cup_{k=1}^\infty \N^k)$, where $\root$ denotes the root, and edges $\{\{v,vi\}: v\in \V, i\in\N\}$, where $vi$ denotes concatenation, with the convention $\root u =u \root =u$ for all $u\in\V$. For $n\in\natzero$, we also let $\V_n:=\{\root\} \cup (\cup_{k=1}^n \N^k)$.  Given a vertex $v\in\V\setminus\{\root\}$, denote by $\pi_v$ its \textit{parent}, defined to be the unique $w\in\V$ such that there exists $k\in \N$ with $wk=v$. The \textit{children} of a vertex $v\in\V$ are defined to be the set $C^\V_v:=\{ vi\}_{i\in\N}$. 

Given a  tree $\Tmc$ with root $\root_\Tmc$, we identify it  (uniquely up to root preserving automorphisms of $\Tmc$) as a subgraph of $\V$ via a map $\Vmc$ from the vertex set of $\Tmc$ to $\V$ such that 
\begin{enumerate}[label=(\roman*)]
    \item $\Vmc(\root_\Tmc)=\root$;
    \item $\Vmc(\partial^\Tmc_\root)= \{m\in\N \ : \ m\leq d_\root^\Tmc\}$;
    \item for $v
    \in\Tmc$ at graph distance\footnote{given two vertices $v$ and $w$ in a graph $G$, the graph distance between them is the minimum number of edges on a path from $v$ to $w$.} $k\in\N$ from $\root_\Tmc$, $\Vmc(v)=u \in\N^{k}$ and $\Vmc(\partial^\Tmc_v)=\{\pi_u\}\cup\{um \ : \ m\in\N,  m\leq d^\Tmc_v-1\}$.
    \end{enumerate}
In order to represent elements in $\stateSS^\Tmc$ as  marks on $\V$,
we consider a new mark $\extra$, and define $\stateSS_{\extra}:=\stateSS\cup\{\extra\}$. Given $x\in\stateSS^\Tmc$, we extend it to an element in $ (\stateSS_{\extra})^\V$  by setting $x_w=\extra$ for all $w\in\V \setminus \Tmc$. Whenever we consider a graph $\Tmc\subset \V$ and $v\in\Tmc$, we use $\partial_v$ and $d_v$  denote neighborhoods and degrees with respect to $\Tmc$. We use $\partial^\V A$  to refer to the boundary in $\V$ of a set $A\subset \V$, and set $\partial^\V_v=\partial^\V\{v\}$ for $v\in\V$. 
Given an interval $\Smc \subset \R$  and $\Mmc$ a metric space, let $\Dmc(\Smc:\Mmc)$ be the space of càdlàg functions\footnote{right continuous functions with finite left limits at every $t$ in the interior of $\Smc$.}  equipped with the Skorokhod topology. For $x$ a (possibly random) element in $\Dmc(\Smc:\Mmc)$ and $t\in \Smc$, we write $x[t]:=\{x(s)\ : s\in \Smc\cap (-\infty,t]\}$ and $x[t):=\{x(s)\ : \ s\in \Smc\cap (-\infty,t)\}$.
 Throughout, we write $\Dmc:=\Dmc([0,\infty):\stateSS])$, and we set $\Dmc_\extra$ to be the union of $\Dmc$ and the single element consisting of the constant-$\extra$ function.

For $y\in\stateSS_\extra^\infty$ we write $\infects(y)=|\set{m\ : \ y_m=I }|$. Also, for simplicity, we identify the states $(S,E,I,R)$ with the vector $(0,1,2,3)$, and the set of possible jumps with $\Jmc=\{1,2\}$.
The jump rate function $q_\alpha \ : \Jmc \times (0,\infty) \times \Dmc_\extra \times  \Dmc_\extra^\infty \rightarrow \R_+$ is then given by
\begin{align}\label{eq:S(E)IR-rate}
    \begin{split}
        &q_\alpha(1, t,x, y) =\one_{\{x(t-)=S\}} \alpha{\beta_t} \infects (y(t-))+\one_{\{x(t-)=E\}} \lambda_t + \one_{\{x(t-)=I\}} {\rho_t},
        \\& q_\alpha(2,t,x , y) =\one_{\{x(t-)=S\}} (1-\alpha){\beta_t} \infects (y(t-)).
    \end{split}
\end{align}
\begin{remark}\label{rmk:q_reduces}
     When $\alpha=0$, $q_\alpha$ defined in \eqref{eq:S(E)IR-rate} reduces to the jump rate function of the SIR process as described in Section \ref{sec:SIR}, and when $\alpha=1$, $q_\alpha$ coincides with the jump rate function of the SEIR process as described in Section \ref{sec:SEIR}.
\end{remark}
Given $j\in\Jmc$ and $v\in\V$, we define $j^{(v)}\in\{0,1,2\}^\V$ by $j^{(v)}_w=j\one_{\{v=w\}}$ for all $w\in\V.$
We define $\xi^{\Tmc,\alpha}$ as a continuous time Markov chain on $\stateSS_\extra^{\V}$ with jump directions $\{j^{(v)}\}_{j\in\Jmc,v\in\V}$ and corresponding jump rates at time $t$ given by $q_\alpha(j,t,\xi_v^{\Tmc,\alpha}, \xi^{\Tmc,\alpha}_{\partial^\V_v})$.
The initial state of the process is given by 
$\xi^{\Tmc,\alpha}_v(0)= z^{p_0}_v$, where $z^{p_0}=\{z^{p_0}_v\}_{v\in\V}$ satisfies the following assumption.
\begin{assumption} \label{assu:initial_z}
    Let $z^{(1)}=\{z^{(1)}_v\}_{v\in\V}$ be a collection of  i.i.d.  $\stateSS$-valued random variables with common law $p_0$ given by \eqref{eq:p_0}. Let
    $ z^{(2)}$ be a $\{1,\extra\}^\V$-valued random variables independent of $z^{(1)}$ such that the subgraph of $\V$ induced by $\{v \ : z^{(2)}_v\neq \extra\}$ is equal in law to a UGW$(\offspring)$ tree. 
    The $\stateSS^\V$-valued random variable $z^{p_0}$ satisfies
    \begin{equation*}
        z ^{p_0}_v=z^{(1)}_v\one_{\{z_v^{(2)}\neq\extra\}}+ \extra \one_{\set{z_v^{(2)}=\extra}}, \qquad v \in\V.
    \end{equation*}
\end{assumption}

We can easily recover the graph $\Tmc$ from the process $\xi^{\Tmc,\alpha}$ as follows:  
\[  \Tmc = \Tmc (\xi^{\Tmc, \alpha}) := \{v \in V: \xi^{\Tmc,\alpha}_v(0) \neq \extra\} = \{v \ : \z^{p_0}_v \neq \extra\}.   \]
Since the graph $\Tmc$ can be infinite, it is no longer immediate that the process $\xi^{\Tmc,\alpha}$ with the intuitive description above is well defined (see \cite[Appendix A]{Ganguly2022hydrodynamic}). However, since $\Tmc$ is a UGW$(\offspring)$ with $\offspring$ having a finite second moment (see Assumption \ref{assu:graph_sequence}), this is guaranteed by the following result proved in  \cite{Ganguly2022hydrodynamic}, which also characterizes $\xi^{\Tmc,\alpha}$ as the unique in law solution of a certain jump SDE.
\begin{lemma}\label{lem:S(E)IR-exist}
The S(E)IR process $\xi^{\Tmc,\alpha}$ exist and is unique in law. Furthermore, its law is the unique solution to the SDE     ,
 \begin{equation}\label{eq:xi-SDE}
        \xi^{\Tmc,\alpha}_v(t) =\
             z^{p_0}_v + \sum_{j=1,2}\int_{(0,t]\times [0,\infty)} j \one_{\left\{r<q_\alpha\left(j, s, \xi^{\Tmc,\alpha}_v,  \xi^{\Tmc,\alpha}_{\partial^\V_v}\right)\right\}} \textbf{N}_v(ds, \ dr), \quad   v\in\V, \ t\in[0,\infty),
    \end{equation}
    %\one_{\{z^{p_0}_v\neq\extra\} }
    where $z^{p_0}$ is a $\stateSS_\extra^\V$-valued random element satisfying Assumption \ref{assu:initial_z} and
     $\textbf{N}= \set{\textbf{N}_v \ : v \in \V  }$ are i.i.d. 
     Poisson point processes on $(0,\infty) \times [0,\infty)$ with intensity measure equal to Lebesgue measure, independent of $z^{p_0}$.
\end{lemma}
\begin{proof} Existence and uniqueness in law of the solution to the SDE \eqref{eq:xi-SDE} follows from \cite[Theorem 4.2]{Ganguly2022hydrodynamic} on observing that
    Assumption \ref{assu:beta_lambda_rho} implies  \cite[Assumption 1]{Ganguly2022hydrodynamic}, and that by \cite[Proposition 5.1]{Ganguly2022hydrodynamic}, the UGW($\offspring$) tree $\Tmc$ is finitely dissociable in the sense of \cite[
Definition 5.1]{Ganguly2022hydrodynamic}. 
\end{proof}

We now define  
\begin{align}
    \begin{split}
        X^{\Tmc}:=\xi^{\Tmc,0},
        \\\Xbar^{\Tmc}:= \xi^{\Tmc,1}.
    \end{split}
\end{align}
\begin{remark}\label{rmk:SIR/SEIR-S(E)IR}
By Remark \ref{rmk:q_reduces} and the uniqueness in law established in Lemma \ref{lem:S(E)IR-exist}, $X^\Tmc_\Tmc$ and $\Xbar^\Tmc_\Tmc$ are the SIR and SEIR processes on the possibly infinite random tree $\Tmc$, in a sense consistent with the definitions of the SIR and SEIR processes on finite graphs given in Section \ref{sec:SIR} and Section \ref{sec:SEIR}, respectively. 
\end{remark}

%%%%%%

\subsection{Proofs of Transient Results}\label{sec:proof_transient}

The proof of Theorem \ref{thm:MAIN_SIR} is presented in Section \ref{sec:proof_odes}. The proof  of Theorem \ref{thm:SEIR_main} uses similar techniques, and is thus only outlined in Section \ref{sec:proof_odes}, with details relegated to Appendix \ref{sec:proof-SEIR}. Both proofs rely on four preliminary results first presented in Section \ref{sec:preliminary_res}. The first ingredient (Theorem \ref{thm:hydrodinamic}) is a convergence result from 
\cite{Ganguly2022hydrodynamic}, 
which shows that the limits of the fractions $|\{v\in G_n: \ \xi^{\Tmc,\alpha}_v(t)=b\}|/|G_n|$, $b\in\stateSS$, coincide with the root marginal probabilities of the limiting S(E)IR dynamics on the graph $\Tmc$ that arises as the local limit of the graph sequence  
$\{G_n\}_{n\in\N}$. The second ingredient is a projection result (Proposition \ref{prop:filtering}) that identifies the law of the marginal dynamics on $\Tmc$ in terms of a certain ({\em a priori} non-Markovian) jump process with somewhat implicit jump rates. This result is a generalization  of similar projection  results obtained in \cite{Ganguly2023marginal,Ganguly2023characterization}. 
The third and fourth results (Proposition \ref{prop:cond_ind_prop} and Proposition \ref{prop:symmetry}) identify key conditional independence and symmetry properties of the dynamics  to explicitly identify  the jump rates of the marginal dynamics.

\begin{remark}\label{rmk:compare_Work_Ganguly} 
For a general class of  interacting particle systems (IPS) on UGW trees $\Tmc$ whose offspring satisfies suitable moment conditions,  which  in particular includes the SIR and S(E)IR processes,  we expect that the marginal dynamics of the IPS on the root and its neighborhood   can be described autonomously in terms of a certain (non-Markovian) stochastic process. Indeed, in the special case when $\Tmc$ is a $\kappa$-regular tree, such a result is established in  
\cite{Ganguly2023marginal} (see also \cite{Ganguly2022nonmarkovian}) by 
 appealing to a  Markov random field property for the trajectories of the process  proved  in \cite[Theorem 3.7]{Ganguly2022interacting} (see also \cite{lacker2021locally,Lacker2021marginal} for corresponding results for interacting diffusions).   
  The current work  goes beyond regular trees to include a   large class of  UGW trees, and also establishes a much stronger conditional independence property of the trajectories  for the S(E)IR process when compared to  general IPS.  The latter is then used to  show that for the S(E)IR process, the root marginal dynamics is in fact  a Markovian process (see Remark \ref{rem-Markov}), and thus its law can be  described by a system of ODEs (namely, the forward Kolmogorov equations describing the evolution of the law of the Markov process).  \end{remark}

We remind the reader that the standing assumptions made at the beginning of Section \ref{sec:main_proofs} are in effect throughout.  
%%%%%
\subsubsection{Preliminary Results}\label{sec:preliminary_res} We start by stating the convergence result.
\begin{theorem}
     \label{thm:hydrodinamic} 
For every $n\in\N$ and $b\in\stateSS$, set
\begin{equation}
    p_b^{n,\alpha}(t) := \frac{1}{|G_n|}\sum_{v\in G_n} \one_{\{\xi^{G_n,\alpha}_v(t)=b\}}.
\end{equation}
For every $t\in[0,\infty)$,
as $n\rightarrow\infty$, 
\begin{equation}
    (p_b^{n,\alpha}(t))_{b\in\stateSS}\rightarrow  (\P(\xi^{\Tmc,\alpha}_\root(t)=b))_{b\in\stateSS},
\end{equation}
where the convergence is in probability.
\end{theorem}
\begin{proof}
   The statement follows from
   \cite[Corollary 4.7]{Ganguly2022hydrodynamic} on observing that Assumption \ref{assu:beta_lambda_rho} implies \cite[Assumption 1]{Ganguly2022hydrodynamic}, Assumption \ref{assu:graph_sequence}, along with \cite[Corollary 5.16]{Ganguly2022hydrodynamic} implies that the graph sequence $\{G_n\}_{n\in\N}$ and $\Tmc$ are a.s. finitely dissociable in the sense of \cite[Definition 5.11]{Ganguly2022hydrodynamic}, and that \cite[Assumption 2]{Ganguly2022hydrodynamic} holds trivially since the state $\stateSS_\extra$ is discrete.
\end{proof}

In view of Theorem \ref{thm:hydrodinamic}, our next goal is to characterize the law of the root marginal of $\xi^{\Tmc,\alpha}$. 
% Following the general philosophy introduced in \cite{Lacker2021marginal,Ganguly2022nonmarkovian,Ganguly2023marginal,Ganguly2023characterization}, 
We first apply a projection result that characterizes the law of any marginal $\xi_U^{\Tmc,\alpha}$ for $U\subset \V$ in terms of a certain  jump process $\project^{\Tmc,\alpha}[U]$.

\begin{proposition}\label{prop:filtering} 
 For every finite $U\subset\V$, $v\in U$, and $j\in\{1,2\}$, there exists a Borel measurable function $\qhat^{\offspring,\alpha}_{v,j}[U]: [0,\infty)\times\Dmc([0,\infty),\stateSS_\extra^U)\rightarrow [0,\infty)$ such that 
 \begin{enumerate}
     \item the function $t\rightarrow \qhat^{\offspring,\alpha}_{v,j}[U](t,x)$ is càglàd\footnote{left continuous with finite right limits at every $t\in[0,\infty)$.} for all $x\in \Dmc([0,\infty),\stateSS_\extra^U)$;
     \item the function $t\rightarrow \qhat^{\offspring,\alpha}_{v,j}[U](t,x)$ is predictable in the sense that for any $t\in[0,\infty)$ and $x,x'\in \Dmc([0,\infty),\stateSS_\extra^U)$, $\qhat^{\offspring,\alpha}_{v,j}[U](t,x)=\qhat^{\offspring,\alpha}_{v,j}[U](t,x')$ whenever $x[t)=x'[t)$;
     \item for every $v\in U$, the stochastic process $t\rightarrow \qhat^{\offspring,\alpha}_{v,j}[U](t,\xi^{\Tmc,\alpha}_U)$ is a modification\footnote{Given two stochastic processes $Y=(Y_t, t\geq0)$ and  $\Yhat=(\Yhat_t, t\geq0)$ defined on the same probability space $(\Omega,\Fmc,\P)$, $Y$ is a modification of $\Yhat$ if for every $t\geq0$, $\P(Y(t)=\Yhat(t))=1$.} of the process $\{\E[ q_\alpha(j, t, \xi^{\Tmc,\alpha}_v,\xi^{\Tmc,\alpha}_{\partial^\V_v}) \ | \ \xi^{\Tmc,\alpha}_U[t)], \ t\in[0,\infty) \}$. 
 \end{enumerate}
    Furthermore, $\law(\xi^{\Tmc,\alpha}_U)=\law(\project^{\Tmc,\alpha}[U])$, where $\project^{\Tmc,\alpha}[U])$ is the  pathwise unique solution  to the following jump SDE \begin{equation}\label{eq:xibar-SDE}
        \project_v^{\Tmc,\alpha}[U](t) =
             z^{p_0}_v + \sum_{j=1,2}\int_{(0,t]\times [0,\infty)} j \one_{\left\{r<\qhat_{v,j}^{\offspring,\alpha}[U]\left(s,\ \project^{\Tmc,\alpha}[U]\right)\right\}} \textbf{N}_v(ds, \ dr), \qquad v\in U, \ t\in[0,\infty),
    \end{equation}
    where  $z^{p_0}$ is a $\stateSS_\extra^\V$-valued  random variable satisfying Assumption \ref{assu:initial_z}, and $\textbf{N}= \set{\textbf{N}_v \ : v \in U  }$ are i.i.d.  Poisson point processes on $(0,\infty) \times [0,\infty)$ with intensity measure equal to Lebesgue measure, independent of $z^{p_0}$.
\end{proposition}
 \begin{proof} 
 In the case when $\Tmc$ is a deterministic $\kappa$-regular tree, this was proved in Lemma 9.1 and Proposition 9.2 of \cite{Ganguly2022nonmarkovian}; see also 
\cite{Ganguly2023marginal}. Using a general result from 
\cite[Corollary 4.11]{Ganguly2022interacting}, this can be extended 
to a class of Galton-Watson trees that include the ones considered in Assumption \ref{assu:graph_sequence}; 
we refer the reader to \cite{Ganguly2023characterization} for full details.
 \end{proof}
Using Proposition \ref{prop:filtering}, the law of $\xi^{\Tmc,\alpha}_{\{\root,1\}}$ can be characterized in terms of a jump process $\project^{\Tmc,\alpha}[{\set{\root,1}}]$. However, the jump rates of the latter process are {\it a priori} path-dependent and not very explicit. We now identify two additional properties that allow us to simplify the form of these jump rates and thereby show that $\project^{\Tmc,\alpha}[{\set{\root,1}}]$ is in fact a nonlinear Markov process (see Remark \ref{rem-Markov}), that is, a (time-inhomogeneous)  Markov process whose transition rates depend not only on the current state but also on the law of the current state. 

%The proofs of our main results rely on the following statement, which is a proven in greater generality in \cite[Theorem 3.7]{Ganguly2022interacting}, and which we  use in the sequel. For a set $A\subset \V$, we recall that $\partial^\V A := \{v\in \V\setminus A\ : \ \exists w\in a \text{ such that } w\in \partial_v^\V\}$ is its \textit{boundary}, and we define  $\partial^{\V,2} A := \partial^\V A \cup \partial^\V (A\cup \partial^\V A)$ to be its \textit{double boundary}. 
%\begin{lemma}\label{lem:2MRF}\footnote{\red{in the end check if we use this anywhere}}
%For all $A\subset\V$ with $|\partial^{\V,2}A|<\infty $, we have 
%\begin{equation}\label{eq:2MRF}
 %   \xi^{\Tmc,\alpha}_A[t) \ \perp \xi^{\Tmc,\alpha}_{\V\setminus A}[t) \ | \xi^{\Tmc,\alpha}_{\partial^{\V,2} A}[t), \qquad \qquad \forall\  t\in[0,\infty).
%\end{equation} 
%\end{lemma}

%We refer to \eqref{eq:2MRF} as the 2-Markov Random Field property of the trajectories of the S(E)IR process, or, in short, the 2-MRF property. 
 For a set $B\subset \V$, we let $S_B$ denote the vector in $\stateSS_\extra^B$ whose every coordinate is equal to $S$.
\begin{proposition}\label{prop:cond_ind_prop}
    For every $t\in[0,\infty)$, $A\subset \V$ with $|\partial^\V A| < \infty$, and $B\subset \V\setminus A$,
     \begin{equation}
        \xi_A^{\Tmc,\alpha}[t) \perp \xi_B^{\Tmc,\alpha}[t) | \xi^{\Tmc,\alpha}_{\partial^\V A}(t-)=S_{\partial^\V A}.
    \end{equation}
    Moreover, for every $v\in \V$, the processes  $\xi^{\Tmc,\alpha} _{vi}[t)$, $ 1\leq i \leq d_v-\one_{\{v\neq\root\}}$ are conditionally independent given  $\xi^{\Tmc,\alpha}_v(t-)=S$ and the degree $d_v$ of $v$.  
\end{proposition}
The proof of Proposition \ref{prop:cond_ind_prop} is given in Section \ref{sec:cond_ind_proof}.

    \begin{proposition} \label{prop:symmetry}
For every $b\in\stateSS$, $\alpha\in[0,1]$ and  $t\in[0,\infty)$, the conditional probability
       $\P(\xi^{\Tmc,\alpha}_{\vtil m}(t-)= b\ | \ \xi^{\Tmc,\alpha}_{\vtil}(t-)=S, \ \vtil\in\Tmc)$ does not depend on the choice of $\vtil\in\V$ and $m\in\N$.
    \end{proposition}

The proof of Proposition \ref{prop:symmetry} proceeds by exploiting the conditional independence property in Proposition \ref{prop:cond_ind_prop}  along with symmetry properties and well-posedness  of the SDE \eqref{eq:xi-SDE} to show that $\law(\xi^{\Tmc,\alpha}_{\vtil m}[t) | \xi^{\Tmc,\alpha}_{\vtil}(t-)=S)=\law(\xi^{\Tmc,\alpha}_{1}[t) | \xi^{\Tmc,\alpha}_{\root}(t-)=S)$ for all $\vtil\in\V$ and $m\in\N$. The details are relegated to the end of Section \ref{sec:cond_ind_proof}.

We conclude this section with an elementary result we use repeatedly in the sequel. %Its proof is left to the reader.

\begin{lemma}\label{lem:cond-identity-lemma}
   Let $(\Omega, \ \Fmc,\ \P)$ be a probability space, and suppose that  $A,$ $ A',$ $ B,$ $ B'\in \Fmc$ with $\P(B\cap B')>0$ and $\P(A'\cap B')>0$. Then, 
\begin{align}
    \begin{split}
        \label{eq:cond_prob_fact}
    \P(A\cap A' \ | \   B\cap B')= \P(A'\ | \  B')\frac{\P(A\cap B\ | \   A'\cap B')}{\P(B\ | \  B')}. 
    \end{split}
\end{align}

\end{lemma}
\begin{proof}
Let $A,$ $ A',$ $ B,$ $ B'$ be as in the statement of the lemma. By the definition of conditional probability, and some simple arithmetic manipulation, 
\begin{align}
    \begin{split}
    \P(A\cap A' \ | \   B\cap B')&= 
    \frac{\P(A\cap A'\cap B\cap B')}{\P(B\cap B')}
    \\&=\frac{\P(A\cap B \ | \   A' \cap B')\P(A'\cap B')}{\P(B\ | \  B')\P(B')}
    \\ &=\P(A'\ | \  B')\frac{\P(A\cap B\ | \   A'\cap B')}{\P(B\ | \  B')}. 
    \end{split}
    \end{align}
\end{proof}

\subsubsection{Proof of Theorem \ref{thm:MAIN_SIR}} \label{sec:proof_odes}
We can now complete the proof of our main result for the SIR process by characterizing the time marginals of $\xi_\root^{\Tmc,\alpha}$ for the special case $\alpha=0$, which by Remark \ref{rmk:SIR/SEIR-S(E)IR} is equal to the marginal at the root of the SIR process on the possibly infinite tree $\Tmc$. 
For $a,b\in\stateS$, $t\in [0,\infty)$ and $k\in\{m\in\natzero\ : \ \offspring(m)>0\}$,  define 
\begin{equation}\label{eq:probs}
    \begin{alignedat}{2}
     &P^\offspring_{a,b;k}(t)& &:=\P(X^\Tmc_\root(t)=b \ |\  X^\Tmc_\root(0)=a,\ d_\root=k),
    \\ & P^\offspring_{a,b}(t)& &:=\P(X^\Tmc_\root(t)=b \ |\  X^\Tmc_\root(0)=a),
  \\ &f^\offspring_a(t)& &:= \P(X^\Tmc_1(t)=a \ | \ X^\Tmc_\root(t)=S,\ 1\in \Tmc  ),
    \end{alignedat}
\end{equation}
where $d_v$ is the degree of the vertex $v\in\Tmc$, and where we recall that $1\in\Tmc$ is equivalent to $X_1^\Tmc(0)\neq\extra$. When clear from context, we omit the dependence on $\offspring$ and simply write $P_{a,b;k}, \ P_{a,b}$ and $f_a$.
 
        \begin{theorem}\label{thm:SIR-fs}
            Let $f_S$ and $f_I$ be as in \eqref{eq:probs}, and set $F_I(t):=\int_0^t\beta_s f_I(s)ds$  for $t\in[0,\infty)$. Then, $(f_S$, $f_I$, $F_I)$ solves the ODE system \eqref{eq:f_If_SF_I}-\eqref{eq:ffF-initial}.
        \end{theorem}
        \begin{proof}
     Throughout the proof, in order to simplify the notation we write $X$ in lieu of $X^{\Tmc}=\xi^{\Tmc,0}$, the SIR process on $\Tmc$, and $q$ in lieu of $q_{0}$ for the jump rates defined in \eqref{eq:S(E)IR-rate}. We start by observing that, by Assumption \ref{assu:initial_z}, $f_S(0)=s_0$ and $f_I(0)=i_0=1-s_0$. Since, clearly $F_I(0)=0$, the initial condition \eqref{eq:ffF-initial} are  established.  
By the fundamental theorem of calculus, $\dot F_I(t) =\beta_t f_I(t)$, which is the third equation in \eqref{eq:f_If_SF_I}.

       We now turn to the derivation of the evolution of $f_I$ and $ f_S$. This requires us to simultaneously track the evolution of two nodes, $\root$ and $1$,   since $f_I(t)$ and $ f_S(t)$ are conditional probabilities associate with the joint law of $X_\root(t)$ and $X_1(t)$. 
       To start with, we apply the projection result of Proposition \ref{prop:filtering}, with $\alpha  = 0$ and $U = \{\root, 1\}$, to conclude that the joint marginal 
       $X_{\root,1}$ has the same law as the jump process 
       $\project^{\Tmc,\alpha}[\{\root,1\}]$ on $\stateS_\extra\times\stateS_\extra$ 
        that has predictable 
       jump rates \begin{equation}\label{eq:qhat-def}
           \hat{q}_{v}(t,x) := \hat{q}_{v,j(x_v)}^{\offspring,0}[\{\root,1\}](t,x),
       \end{equation} $v \in \{\root, 1\}$, $x \in \Dmc([0,\infty),\stateS_\extra^2)$ and $j(x_v)=1+\one_{\{x_v(t-)=S\}}$, 
       which satisfy,  for every $t \geq 0,$ almost surely\footnote{The dependence $j(x_v)$ 
            of the allowed jump  on the state is a notational nuisance that is a mere artifact of our using a common framework to analyze both the SIR and SEIR processes. Indeed, this is because when we use the common (ordered) state space $\{S,E,I,R\}$ 
            for both processes, then the SIR process allows 
            only jumps of size $2$ from the state S (going from S to I and skipping over E), and only jumps of size $1$ from the  state I (going from I to R).} 
       \begin{equation}\label{eq:qhat}
            \hat{q}_{v}(t,X_{\root, 1}) = \E[q(j(X_v),t, X_v, X_{\partial_v^{\mathbb{V}}})|X_{\root, 1}[t)], \quad v \in \{\root, 1\}. 
\end{equation}

       Next, we use the specific form of $q$, as defined in \eqref{eq:S(E)IR-rate} and Propositions \ref{prop:cond_ind_prop} and \ref{prop:symmetry} to obtain a more explicit description of $\qhat_v$, $v\in\{\root,1\}$. Since the probabilities $f_I(t)$ and $f_S(t)$ are conditioned on $X_\root(t)=S$ and on $X_1(t)\neq\extra$ (and using the fact that a particle that is in state $R$ remains in that state for all subsequent times), we only need to consider the jump intensities $\qhat_v(t,X_{\root,1})$, $v\in\{\root,1\}$,
       on the events $\{X_{\root,1}(t-)=(S,S)\}$ and $\{X_{\root,1}(t-)=(S,I)\}$.
       
Define $B_1(t):=\beta_t\E[\infects(X_{\partial^\V_1\setminus\{\root\}}(t-)) | X_1(t-)=S]$. Recalling the definition of $q=q_0$ from \eqref{eq:S(E)IR-rate}, $B_1(t)$ is the cumulative conditional rate at which the children of $1$ infect $1$ at time $t$, given $X_1(t-)=S$ (which also implies $1\in\Tmc$). Similarly, let $B_{\root}(t):=\beta_t\E[\infects(X_{\partial^\V_\root\setminus\{ 1 \}}(t-)) | X_\root(t-)=S]$ be the cumulative conditional rate at which the neighbors of the root other than vertex $1$ infect the root at time $t$, given $X_\root(t-)=S$. By  Proposition \ref{prop:cond_ind_prop},  for $v,w\in\{\root,1\}$ with $w\neq v$,
\begin{equation}\label{eq:B_v}
    B_v(t)= \beta_t\E[\infects(X_{\partial^\V_v\setminus\{ w\}}(t-)) | X_v(t-)=S, \ X_w(t-)].
\end{equation}

Using \eqref{eq:S(E)IR-rate} and \eqref{eq:B_v}, on the event $\{X_\root(t-)=S\}$, $$\qhat_\root(t,X_{\root,1})= \beta_t\one_{\{X_1(t-)=I\}}+B_\root(t).$$ Similarly, on the event $\{X_\root(t-)=S\}$, $$\qhat_1(t,X_{\root,1})=B_1(t)\one_{\{X_1(t-)=S\}}+\rho_t \one_{\{X_1(t-)=I\}}.$$ Therefore, we can treat $X_{\root,1}$ as a two particle jump processes driven by Poisson noises with intensity measure equal to Lebesgue measure, whose  jumps and jump rates from the states $(S,S)$ and $(S,I)$ can be summarized as follows:
    \begin{align*}
        &\text{Jump: }  & 
  &\text{Rate at time $t$:}
  \\       (S,S) &\rightarrow (S,I)  & 
  & B_1(t)
  \\  (S,S) &\rightarrow (I,S)  & 
  & B_{\root}(t)
  \\ (S,I) &\rightarrow (I,I)  & 
  &{\beta_t} + B_{\root}(t)
  \\ (S,I) &\rightarrow (S,R)  & 
  &{\rho_t},
       \end{align*}
with all other jump rates being equal to zero.
       Next, we fix $h>0$ and  we obtain expressions for $f_I(t+h),\ f_S(t+h)$ in terms of $f_I(t),$ $ f_S(t), $ $ h,$ $  {\beta_t},$ $ {\rho_t}$, and $\sizeb$.
We first consider $f_S(t+h)=\P(X_1(t+h)=S\ | \ X_\root(t+h)=S, \ 1\in\Tmc)$ defined in \eqref{eq:probs}. Using the monotonicity of the SIR dynamics, we can write 
\begin{equation} \label{eq:proof_fs_f_S}
    f_S(t+h)=\P(X_1(t+h)= S,\ X_1(t)= S \ | \ X_\root(t+h)=S,\ X_\root(t)=S,\ 1\in\Tmc ).
\end{equation}
By an application of Lemma \ref{lem:cond-identity-lemma}, with $A=\{X_1(t+h)=S\} $, $A'=\{X_1(t)=S\}$, $B=\{X_\root(t+h)=S\}$ and $B'=\{X_\root(t)=S, 1 \in\Tmc\}$, we obtain 
    \begin{equation}\label{eq:f_S}
        f_S(t+h)= f_S(t)\frac{\P(X_\root(t+h)=S,\ X_1(t+h)=S,  | \ X_\root(t)=S, \ X_1(t)=S,  \ 1\in\Tmc)}{\P(X_\root(t+h)=S\ | \ X_\root(t)=S, \ 1\in\Tmc)}.
    \end{equation}
    Since $B_1(t)+B_\root(t)$ is the rate at which $X_{\root,1}(t)$ leaves the state $(S,S)$, the numerator in  the right-hand side of \eqref{eq:f_S} is equal to $1-h(B_1(t)-B_\root(t))+o(h)$. For the denominator, observe that the rate $\qhat_\root(t,X_{\root,1})$ on the event $\{X_\root(t-)=S,\  X_1(t-)\neq \extra\}=\{X_\root(t-)=S, \ 1\in\Tmc\}$ is equal to
    \begin{align*}
        \begin{split}
   & \E[q(1,t,X_\root,X_{\partial^\V_\root}) \ | X_\root(t-)=S, 1\in\Tmc ] 
    \\ = &  \beta_t\E[\infects(X_1(t-)) \ | X_\root(t-)=S, \ 1\in\Tmc ] + \beta_t\E[\infects(X_{\partial^\V_\root\setminus\{1\}}(t-)) \ | X_\root(t-)=S, \ 1\in\Tmc ] 
    \\ = & \beta_t f_I(t) + B_\root(t),
        \end{split}
    \end{align*}
    where the first equality follows from \eqref{eq:S(E)IR-rate} with $\alpha=0$, and the second follows from the definition of $f_I$ in \eqref{eq:probs} and by \eqref{eq:B_v} (on observing that the event $\{1\in\Tmc\}$ is $X_1(t)$-measurable). Therefore, it follows that
\begin{equation}
        f_S(t+h)= f_S(t)\frac{1-h(B_1(t)+B_{\root}(t))+o(h)}{1-h({\beta_t} f_I(t)+B_{\root}(t))+o(h)},
    \end{equation}
which implies that
\begin{align} \label{eq:proof_fs_endf_S}
    \begin{split}
        f_S(t+h)-f_S(t)&= f_S(t)\frac{1-hB_1(t) -hB_{\root}(t)-1+h{\beta_t} f_I(t) +hB_{\root}(t)+o(h)}{1-h({\beta_t} f_I(t)+B_{\root}(t))+o(h)}
        \\ &= f_S(t)\frac{h{\beta_t} f_I(t) -hB_1(t)+o(h)}{1+o(1)}. 
    \end{split}
\end{align}
In turn, this implies 
\begin{equation}\label{eq:F_S_proof}
    \dot f_S(t)= \beta_tf_S(t)f_I(t)- f_S(t) B_1(t). 
\end{equation}

Similarly, recalling that $f_I(t+h)=\P(X_1(t+h)=I\ |\ X_\root(t+h)=S, \ 1\in\Tmc)$ from \eqref{eq:probs}, using the fact that a particle that at time $t+h$ is in state $I$ could only have been in states $S$ or $I$ at time $t$, and using a similar derivation as in \eqref{eq:proof_fs_f_S}-\eqref{eq:proof_fs_endf_S}, 
\begin{align}
    \begin{split}
f_I(t+h)&=\sum_{a=S,I}f_a(t)\frac{\P(X_\root(t+h)=S, \ X_1(t+h)=I|\ X_\root(t)=S , \ X_1(t)=a,\ 1\in\Tmc)}{\P(X_\root(t+h)=S|\ X_\root(t)=S ,1\in\Tmc)}\\ &= \frac{f_S(t)(hB_1(t)+o(h)) + f_I(t)(1-h({\rho_t}+B_{\root}(t)+{\beta_t})+o(h))}{1-h(f_I(t){\beta_t}+B_{\root}(t)) +o(h)},
    \end{split}
\end{align}
      which implies that
       \begin{align}
    \begin{split}
&f_I(t+h)-f_I(t)=(1+o(1))(h f_S(t)B_1(t) -hf_I(t)({\rho_t}+{\beta_t} -{\beta_t} f_I(t)) +o(h)).
    \end{split}
\end{align}
  It follows that
    \begin{equation}\label{eq:f_I_proof}
        \dot f_I=f_S B_1 -f_I({\rho}+{\beta}-{\beta} f_I).
    \end{equation}
    
    In view of \eqref{eq:F_S_proof} and \eqref{eq:f_I_proof} all that is left to find is an expression for $B_1(t)$, the conditional rate at which the children of vertex $1$ infect vertex $1$ at time $t$, given $X_1(t)=S$, in terms only of ${\beta_t},\ {\rho_t},\ \sizeb,$ and $f_I(t)$.  
    By  Proposition \ref{prop:cond_ind_prop}, $X_{\partial_1\setminus{\root}}(t)$ is conditionally independent of $X_\root(t)$ given $X_1(t)=S$. Also by Proposition \ref{prop:cond_ind_prop}, $\{X_{1i}(t)\}_{i=1,...,k}$, are conditionally i.i.d.  given $X_1(t)=S$ and $d_1=k+1$, and by Proposition \ref{prop:symmetry}, $$\P(X_{1i}(t)=I\ |\ X_{1}(t)=S,\ d_1=k+1)= \P(X_{1i}(t)=I\ |\ X_{1}(t)=S, \ 1i\in\Tmc) =f_I(t),$$
    for $i=1,...,k$.
    This implies that 
    \begin{align}\label{eq:B1_1}
        \begin{split}
            B_1(t)&= \beta_t\E[\infects(X_{\partial^\V_1\setminus\{\root\}}(t-)) \ | \ X_1(t-)=S]
            \\ &=  \beta_t\E[ \E[\infects(X_{\partial^\V_1\setminus\{\root\}}(t-)) | X_1(t-)=S, \ d_1=k+1] \ |\  X_1(t-)=S]
            \\ & =  \beta_t f_I(t)\E[ \E[d_1-1 | X_1(t-)=S, \ d_1=k+1] \ |\  X_1(t-)=S]
            \\&= \beta_t f_I(t) (\E[d_1 \ -1 |\ X_1(t-)=S]).
        \end{split}
    \end{align}
    
     Next, we find a more explicit description of the conditional expectation in the last line of \eqref{eq:B1_1}.  Let $\bar{\Nmc}=\{k\in\natzero \ : \ \sizeb(k+1)>0\}$. For $k\in\bar{\Nmc}$, define
    \begin{equation}
        \condi_k:=\P(X_1(t)=S\ | \ 1\in\Tmc, \ d_1=k+1).
    \end{equation}
Then, observing that $X_1(t)=S$ implies that $1\in\Tmc$,
\begin{equation}\label{eq:d|X1}
    \E[d_1-1\ | X_1(t-)=S]=\sum_{k\in\bar{\Nmc}}k\frac{\P(d_1=k+1\ | \ 1\in\Tmc)}{\P(X_1(t-)=S\ |\ 1\in\Tmc)} \condi_k(t-).
\end{equation}
By \eqref{eq:S(E)IR-rate}, the conditional rate at which the individual at $1$ is infected, given that $X_1(t-)=S$ and $d_1=k+1$, is
% \footnote{I think: \blue{And another application of Proposition \ref{prop:filtering} with $\alpha=0$ and $U=\{1\}\cup\{1\ell\}_{\ell=1}^k+1$}}
\begin{align}\label{eq:rate_X1}
    \begin{split}
& \beta_t\E[\infects(X_{\partial^\V_1}(t)-)) \ | \ X_1(t-)=S, \ d_1=k+1 ]
        \\ &= \beta_t\E[\infects(X_{\partial^\V_1\setminus{\{\root\}}}(t-))) \ | \ X_1(t-)=S, \ 
        d_1=k+1 ] +  \beta_t\E[\infects(X_{\root}(t-))) \ |\ X_1(t-)=S ]
        \\ &= \beta_t k f_I(t) + \beta_t\E[\infects(X_{\root}(t-))) \ |\ X_1(t-)=S ]
    \end{split}
\end{align}
where the second equality follows from Proposition \ref{prop:cond_ind_prop} and Proposition \ref{prop:symmetry}, and  the first equality follows from an application of Proposition \ref{prop:cond_ind_prop} with $A=\{\root\}\cup\{\ell v\}_{ \ell\in\N\setminus\{1\}, v\in\V}$, $\partial^\V A=\{1\}$ and $B=\{1m\}_{m\in\N}$. Setting $\iota(t):=\beta_t\E[\infects(X_{\root}(t-))) \ |\ X_1(t-)=S ]$,  using the monotonicty \eqref{eq:mono} of the SIR process in the first equality and \eqref{eq:rate_X1} in the second, we have
\begin{align}
    \begin{split}
        \condi_k(t+h)=&\P(X_1(t+h)=S \ | X_1(t)=S \ d_1=k+1)\condi_k(t)
        \\=&(1-hk\beta_tf_I(t)-h\iota(t) )\condi_k(t),
    \end{split}
\end{align}
and it follows that 
\begin{equation*}
    \dot \condi_k(t)=-(k\beta_t f_I(t)+ \iota(t))\condi_k(t)
\end{equation*}
and, since $\condi_k(0)=s_0$ by  Assumption \ref{assu:initial_z},
\begin{equation}
    \label{eq:\condi}
    \condi_k(t)=s_0e^{-k\int_0^t\beta_sf_I(s)ds + \int_0^t\iota(s)ds}.
\end{equation}

Next, observing that  $\P(d_1=k+1 \ | \ 1\in\Tmc)=\sizeb(k)$ since $\Tmc$ is a UGW$(\offspring)$, and
\begin{align*}
    \begin{split}
    \P(X_1(t-)=S \ | \ 1\in\Tmc) &= \sum_{k\in\bar{\Nmc}} \P(X_1(t-)=S \ | d_1=k+1, 1\in\Tmc)\P(d_1=k+1 \ |1\in\Tmc )
        \\ &= \sum_{k\in\bar{\Nmc}} r_k(t)\sizeb(k),
    \end{split}
\end{align*}
the expression in \eqref{eq:d|X1} can be rewritten as
\begin{align}\label{eq:d_1_final}
    \begin{split}
        \E[d_1-1\ | X_1(t-)=S] = &\sum_{k\in\bar{\Nmc}}k\frac{\P(d_1=k+1\ | \ 1\in\Tmc)}{\P(X_1(t)=S\ |\ 1\in\Tmc)} \condi_k(t).
        \\ =& \sum_{k\in\bar{\Nmc}}k\frac{\sizeb(k)}{\sum_{\ell\in\bar{\Nmc}}\sizeb(\ell)\condi_\ell} \condi_k(t)
        \\ =& \sum_{k\in\bar{\Nmc}}k \sizeb(k)\frac{ s_0e^{-k\int_0^t\beta_sf_I(s)ds + \int_0^t\iota(s)ds} }{\sum_{\ell\in\bar{\Nmc}}\sizeb(\ell)s_0e^{-\ell\int_0^t\beta_sf_I(s)ds + \int_0^t\iota(s)ds}} 
        \\ =& \frac{ \sum_{k\in\bar{\Nmc}}k \sizeb(k)e^{-k\int_0^t\beta_sf_I(s)ds} } {\sum_{\ell\in\bar{\Nmc}}\sizeb(\ell)e^{-\ell\int_0^t\beta_sf_I(s)ds}}, 
    \end{split}
\end{align}
where in the third equality we used \eqref{eq:\condi}. Combining  \eqref{eq:d_1_final} and \eqref{eq:B1_1}, and recalling that $F_I(t):=\int_0^t\beta_sf_I(s)ds$, we obtain
    \begin{align}\label{eq:B_1}
        \begin{split}
            B_1(t)=  \beta_t f_I(t)  \frac{\sum_{k=0}^\infty k\sizeb(k) e^{- k F_I(t)}}{ \sum_{\ell=0}^\infty \sizeb(\ell) 
  e^{- \ell F_I(t)} } .
        \end{split}
    \end{align}
    As desired, this expresses $B_1(t)$ purely in terms of $\ \sizeb, \ f_I $ and $F_I(t)$. Combining \eqref{eq:B_1} with \eqref{eq:F_S_proof} and \eqref{eq:f_I_proof} establishes the first and second equation of \eqref{eq:f_If_SF_I}, thus concluding the proof.  \end{proof}

\begin{remark}\label{rem-Markov}
In the proof of Theorem \ref{thm:SIR-fs} we showed that the jump rate $\qhat_v(t,X^\Tmc_{\root,1})$ as defined in \eqref{eq:qhat-def} is not path dependent on the event $\{X_\root(t-)=S\}$. By a similar argument that appeals to Proposition \ref{prop:cond_ind_prop}, one can show that $\qhat_v(t,X^\Tmc_{\root,1})$ is also not path dependent on the event $\{X_\root(t-)=S\}^c$, thereby showing that $X_{\root,1}$ is a  Markov  process.  The analogue of the latter property  can also be shown to  hold for the discrete-time SIR process using a similar (in fact, simpler) proof.  Numerical evidence supporting this property for the discrete-time SIR process on trees was first provided in \cite{WortThesis18}. 
\end{remark}

\begin{theorem} \label{thm:SIR_ODE}
Let $f_I$ be as in \eqref{eq:probs} and set $\Nmc:=\{m\in\natzero\ : \ \offspring(m)>0\}$. Then  $((P_{S,S;k}$, $ P_{S,I;k})_{k\in\Nmc}$, $ P_{I,I})$, as defined in \eqref{eq:probs}, is the unique solution to the following system of ODEs:
\begin{equation} \label{eq:SIR-ODE}
    \begin{cases}
        \dot P_{S,S;k}= -{\beta} k f_I P_{S,S;k}, & k\in\Nmc,
        \\  \dot P_{S,I;k}= {\beta} k f_I P_{S,S;k} -{\rho} P_{S,I;k} , & k\in\Nmc,
        \\ \dot P_{I,I}=- {\rho} P_{I,I},

    \end{cases}
\end{equation}
    with initial conditions 
    \begin{equation} \label{eq:SIR-ODE-initial}
    \begin{cases}
        P_{S,S;k}(0)= 1,  & k\in\Nmc,
        \\   P_{S,I;k}(0)= 0,  &   k\in\Nmc,
        \\  P_{I,I}(0)= 1,
    \end{cases}
\end{equation}
\end{theorem}

 \begin{proof}
    Throughout the proof, in order to simplify the notation we write $X$ in lieu of $X^{\Tmc}$.
    By Assumption \ref{assu:beta_rho}, the fact that $f_I$ defined in \eqref{eq:probs} is continuous (since by Theorem \eqref{thm:SIR-fs} it is characterized in terms of the solution of the ODE system \eqref{eq:f_If_SF_I}-\eqref{eq:ffF-initial}) and the fact that the ODE system \eqref{eq:SIR-ODE} is linear, the initial value problem \eqref{eq:SIR-ODE}-\eqref{eq:SIR-ODE-initial} has a unique solution.
    Clearly, from \eqref{eq:probs}, the initial conditions \eqref{eq:SIR-ODE-initial} hold. Next, we show that \eqref{eq:SIR-ODE} is satisfied.

We start by considering $P_{S,S;k}$. Fix $t\geq0$,  $h>0$ and $k\in\Nmc$. Since $s_0=\P(X_\root(0)=0)>0$ and $d_\root$ is independent of $X_\root(0)$, $\P(X_\root(0)=S, \ d_\root)>0$.  From \eqref{eq:probs}, noting that $X_{\partial _\root}(t)=y\in\stateS^{k}$ implicitly implies $d_\root=k,$ we have 
\begin{align}\label{eq:proof_sir_ode_0}
            \begin{split}
                & P_{S,S;k}(t+h) 
                \\ &=\P(X_\root(t+h)=S \ | \ X_\root(0)=S, \ d_\root=k)
                \\ &= \sum_{y\in\stateS^{k}} \frac{\P(X_\root(t+h)=S,\ X_\root(0)=S ,  \ X_{ \partial_\root}(t)=y) }{\P(X_{\root}(0)=S, \ d_\root=k  )} 
                \\&= \sum_{y\in\stateS^{k}} \frac{\P(X_\root(t+h)=S ,  \  X_\root(t)=S, \ X_{ \partial_\root}(t)=y ) }{\P(X_{\root}(0)=S, \ d_\root=k  )} 
                 \\ &= \sum_{y\in\Smc_{k,t}}\P(X_\root(t+h)=S | X_\root(t)=S, X_{{\partial_\root}}(t)=y) \P(X_\root(t)=S, X_{ {\partial_\root}}(t)=y |  X_{\root}(0)=S, d_\root=k) ,
            \end{split}
        \end{align}
        where $\Smc_{k,t}:=\{y\in\stateS^k\ : \ \P(X_\root(t)=S, \ X_{\partial_\root}(t)=y)>0\}$, and the monotonicity \eqref{eq:mono} of the SIR process is used in the third and fourth equality. Since the jump rate of a susceptible individual whose neighbors' states are equal to $y\in\stateS^{k}$ is equal to ${\beta_t}\infects (y)$, we have that \begin{equation}\label{eq:proof_sir_ode_i}
        \P(X_\root(t+h)=S \ |\ X_\root(t)=S, X_{{\partial_\root}}(t)=y)= 1-h{\beta_t}\infects (y)+o(h).
    \end{equation}
    The expression on the right-hand side of \eqref{eq:proof_sir_ode_i} does not depend on the exact values of the $k-\Imc(y)$ elements of $y$ that are not equal to $I$. Thus, substituting the expression in \eqref{eq:proof_sir_ode_i} into the last line of \eqref{eq:proof_sir_ode_0} and rewriting the sum to be over the number of infected neighbors of $\root$,
        \begin{align*}
            \begin{split}
               P_{S,S;k}(t+h) =  &\sum_{j=0}^k (1-h {\beta_t} j +o(h)) \ \P(X_\root(t)=S,\ \infects (X_{ {\partial_\root}}(t))=j\ | \ X_{\root}(0)=S, d_\root=k) 
                \\ =& \sum_{j=0}^k (1-h {\beta_t} j +o(h)) \ \P(\infects(X_{ {\partial_\root}}(t))=j\ | X_\root(t)=S, \ X_{\root}(0)=S,\  d_\root=k) \ P_{S,S;k}(t) 
                \\ =& \sum_{j=0}^k (1-h {\beta_t} j +o(h)) \ \P(\infects(X_{ {\partial_\root}}(t))=j\ |  X_\root(t)=S, \ d_\root=k) \ P_{S,S;k}(t),
            \end{split}
        \end{align*}
        where in the last equality we used the monotonicity of the SIR process \eqref{eq:mono}.
        Applying Proposition \ref{prop:cond_ind_prop} with $\alpha=0$, it follows that $\{X_{i}(t) \ : \ i\sim \root \}$ are conditionally i.i.d. given $\{X_{\root}(t)=S, d_\root=k\}$. Furthermore, for $k\geq 1$ and $m\in\N \cap [0, k]$, by Proposition \ref{prop:symmetry} and another application of Proposition \ref{prop:cond_ind_prop} with $A=\set{mv}_{v\in\V}$, $\partial^\V A=\set{\root}$ and $B=\N\setminus\set{m}$, and observing that $d_\root=\sum_{\ell=1}^\infty\one_{\set{X_\ell(0)\neq\extra}}$, we have that  $$\P(X_m(t)=I \ |\ X_\root(t)=S,\  d_\root=k)= \P(X_m(t)=I \ |\ X_\root(t)=S,\  m\in\Tmc) =f_I(t),$$ where $f_I$ is as in \eqref{eq:probs}. Therefore, conditional on $X_\root(t)=S$ and  $d_\root=k$, $\infects(X_{{\partial_\root}}(t))$ has a binomial distribution with parameters $(k,f_I(t))$. It follows that, letting $Y$ be a binomial random variable with parameters $(k,f_I(t))$,
        \begin{align}\label{eq:proof_sir-ii}
        \begin{split}
              %P_{S,S;k}(t+h) &= P_{S,S;k}(t)  \sum_{j=0}^k (1-h {\beta_t} j +o(h)) {k \choose j} (f_I)^j(1-f_I)^{k-j} \ \\ &= 
               P_{S,S;k}(t+h) & =P_{S,S;k}(t) (\E[1-h {\beta_t} Y] +o(h)) 
              \\ & = P_{S,S;k}(t)(1-h{\beta_t} k f_I + o(h)).
        \end{split}
        \end{align}
         This implies 
         \begin{align}
         \lim_{h\rightarrow0^+} \frac{P_{S,S;k}(t+h)-P_{S,S;k}(t)}{h}=   \lim_{h\rightarrow0^+} \frac{(1-h{\beta_t} k f_I(t) +o(h)-1)P_{S,S;k}(t)}{h}= -{\beta_t} k f_I(t)P_{S,S;k}(t),    
         \end{align}
         which proves the first equation in \eqref{eq:SIR-ODE}

         The derivations of the ODEs for $ P_{S,I;k}$ and $ P_{I,I}$ are similar, and are thus only outlined below. As in the last line of \eqref{eq:proof_sir_ode_0}, we start by writing

\begin{equation}\label{eq:P_SIk}
    P_{S,I;k}(t+h) = \Qmc_I(h) + \Qmc_S(h),
\end{equation}
         where for $b=I$ and $b=S$
         \begin{align*}
               &\Qmc_b(h)
               = \sum_{j=0}^k \P(X_\root(t+h)=I , \ X_\root(t)=b,\ \infects (X_{{\partial_\root}}(t))=j \ |\    X_{\root}(0)=S, \ d_\root=k).
         \end{align*}
           
        Recalling the definition of the SIR rates $q_0$ as in \eqref{eq:S(E)IR-rate} and using arguments similar to what used to derive \eqref{eq:proof_sir_ode_i}-\eqref{eq:proof_sir-ii}, $\Qmc_S(h)=(h{\beta_t} k f_I(t) +o(h) )P_{S,S;k}(t)$ and 
        \begin{align}
            \begin{split}
                 \Qmc_I(h)&= \sum_{j=0}^k (1-{\rho_t} h +o(h))  \P(X_\root(t)=I,\ \infects (X_{ {\partial_\root}}(t))=j\ |\  X_{\root}(0)=S, \ d_\root=k) 
                 \\ &= (1-{\rho_t} h +o(h))  \sum_{j=0}^k \P(X_\root(t)=I,\ \infects (X_{ {\partial_\root}}(t))=j\ |\  X_{\root}(0)=S, \ d_\root=k ) 
                 \\ &= (1-{\rho_t} h +o(h))
                 \P(X_\root(t)=I \ | \ X_{\root}(0)=S, \ d_\root=k)
                 \\&=(1-{\rho_t} h + o(h)) P_{S,I;k}(t).
            \end{split}
        \end{align}
       Substituting the last two displays into \eqref{eq:P_SIk}, we obtain $P_{S,I;k}(t+h)-P_{S,I;k}(t)= hk {\beta_t} f_I(t) P_{S,S;k}(t)- {\rho_t} h P_{S,I;k}(t) + o(h)$,which implies the second equation in \eqref{eq:SIR-ODE}.
       
       Next, to obtain the ODE for $P_{I,I}$ note that by definition of the jump rate \eqref{eq:S(E)IR-rate}, setting $\Nmc:=\set{k\in\natzero\ : \ \offspring(k)>0}$,
       \begin{align}
       \begin{split}
         &  P_{I,I}(t+h)
         \\ &=\sum_{k\in \Nmc }\sum_{y\in\stateS^k}\P(X_\root(t+h)=I  | X_\root(t)=I, X_{\partial_\root}(t)=y) \P(X_\root(t)=I,\ X_{\partial_\root}(t)=y |  X_{\root}(0)=I) 
         \\& = (1-h {\rho_t} + o(h)) \sum_{k\in \Nmc }\sum_{y\in\stateS^k}\P(X_\root(t)=I,\ X_{\partial_\root}(t)=y\ | \ X_{\root}(0)=I) 
          \\& = (1-h {\rho_t} + o(h)) P_{I,I}(t),
       \end{split}
       \end{align}
        which implies the third equation in \eqref{eq:SIR-ODE} and concludes the proof.

        \end{proof}
We can combine the results above to prove Theorem \ref{thm:MAIN_SIR}.

 \begin{proof}[Proof of Theorem \ref{thm:MAIN_SIR}]
By Theorem \ref{thm:hydrodinamic}, $\lim_{n\rightarrow\infty}s^{G_n}(t)=\P(X^\Tmc_\root(t)=S)$ and $\lim_{n\rightarrow\infty}i^{G_n}(t)=\P(X^\Tmc_\root(t)=I)$. By Theorem \ref{thm:SIR_ODE}, we can characterize the transition rates of  $X^{\Tmc}_\root(t)$  defined in \eqref{eq:probs} as the solution to the ODE system \eqref{eq:SIR-ODE}-\eqref{eq:SIR-ODE-initial}. Let $f_I$ and $ f_S$ be as in \eqref{eq:probs}, and $F_I(t)=\int_0^t\beta_sf_I(s)ds$, $t\in[0,\infty)$ .Then we can solve the ODE system  \eqref{eq:SIR-ODE}-\eqref{eq:SIR-ODE-initial} as follows:
\begin{equation*}
    \begin{alignedat}{2}
         & P_{S, S ;k}(t) & &= e^{-kF_I(t)},
    \\ & P_{S, I ;k}(t) & &= e^{-\int_0^t\rho_u du} \int_0^t k e^{-kF_I(u)} e^{\int_0^u\rho_\tau d \tau} \beta_uf_I(u)d u,
       \\ &P_{I, I}(t)& &=e^{-\int_0^t\rho_u d u}.
    \end{alignedat}
\end{equation*}
In view of \eqref{eq:probs}, by averaging over $d_\root$, that is, by multiplying each of the quantities above by $\offspring(k)=\P(d_\root=k)$ and summing over $k\in\N$, we conclude that
\begin{equation*}
    \begin{alignedat}{2}
        & P_{S, S}(t) & &= M_\offspring(-F_I(t)),
    \\ & P_{S, I}(t) & &= e^{-\int_0^t\rho_u du} \int_0^t M'_\offspring (-F_I(u)) e^{\int_0^u\rho_\tau d \tau} \beta_uf_I(u)du,
    \end{alignedat}
\end{equation*}
where $M'_\offspring(z)=\sum_{k=0}^\infty k \offspring(k)e^{kz}$, and the exchange in order of summation and integration is justified by the fact that every term is non-negative. By Theorem \ref{thm:SIR-fs}, $(f_S$, $f_I$, $F_I)$ solve the ODE system \eqref{eq:f_If_SF_I}-\eqref{eq:ffF-initial}.
Finally, since $\P(X^\Tmc_\root(t)=S)=s_0P_{S,S}(t)$ and $\P(X^\Tmc_\root(t)=I)= s_0 P_{S,I}(t) + i_0P_{I,I}(t)$, equation \eqref{eq:s_i_infinity} follows. This completes the proof.
\end{proof}

\subsubsection{Proof of Theorem \ref{thm:SEIR_main}}\label{sec:SEIR-main-proof}
Now, we turn our attention to the SEIR process. Since its derivation is similar to that of Theorem \ref{thm:MAIN_SIR}, we relegate most of the details to Appendix \ref{sec:proof-SEIR}. For $a,b\in\stateSS$,  $t\in[0,\infty)$ and $k\in\{m\in\natzero \ : \ \offspring(m)>0\}$, define
\begin{equation}\label{eq:SEIR-probs}
    \begin{alignedat}{2}
        &Q^\offspring_{a,b;k}(t)& &\coloneqq\P(\Xbar^{\Tmc}_\root(t)=b\ |\ \Xbar^{\Tmc}_\root(0)=a, \ d_\root=k) ,
        \\ &Q^\offspring_{a,b}(t)& &\coloneqq\P(\Xbar^{\Tmc}_\root(t)=b \ |\  \Xbar^{\Tmc}_\root(0)=a),
      \\ & g^\offspring_{a}(t)& &\coloneqq \P(\Xbar^{\Tmc}_1(t)=a \ |\ \Xbar^{\Tmc}_\root(t)=S, \ 1\in \Tmc).
    \end{alignedat}
\end{equation}
When clear from the context, we omit the dependence on $\offspring$ and  write $Q_{a,b},\ Q_{a,b;k}$ and $g_{a}$.

\begin{theorem}\label{thm:gggG}
            Let $g_S$, $g_E$ and $g_I$ be as in \eqref{eq:SEIR-probs} and set $G_I(t):=\int_0^t\beta_s g_I(s)ds$  for $t\in[0,\infty)$. Then, $(g_S$, $g_E$, $g_I$, $G_I)$ solves the ODE system \eqref{eq:gggG-Odes}-\eqref{eq:gggG-Odes_initial}.
        \end{theorem}
The proof of Theorem \ref{thm:gggG} is given in Appendix \ref{sec:proof-SEIR}.

\begin{theorem} \label{thm:SEIR_ODE}
   
Let $g_I$ be as in \eqref{eq:SEIR-probs} and set $\Nmc:=\{m\in\natzero\ : \ \offspring(m)>0\}$.  Then $ ((Q_{S,i;k})_{,i\in\stateSS\setminus\{R\}, k\in\Nmc}$, $ Q_{E,E}$, $Q_{E,I}$, $Q_{I,I})$ is the unique solution to the following system of ODEs:
    
 \begin{align} \label{eq:ODE-SEIR}
        \begin{cases}
            \dot Q_{S,S;k}=-{\beta} k g_I Q_{S,S;k},  &  k \in\Nmc,
            \\ \dot Q_{S,E;k}= {\beta} k g_I Q_{S,S;k}- {\lambda} Q_{S,E;k}, &  k  \in\Nmc,
            \\ \dot Q_{S,I;k}={\lambda} Q_{S,E;k}-{\rho} Q_{S,I;k}, &  k \in\Nmc,
            \\ \dot Q_{E,E}= -{\lambda} Q_{E,E},
            \\ \dot Q_{E,I}={\lambda} Q_{E,E}- \rho Q_{E,I} ,
            \\ \dot Q_{I,I}=-{\rho} Q_{I,I},
        \end{cases}
    \end{align}

with initial conditions
    \begin{align} \label{eq:ODE-SEIR_initial}
        \begin{cases}
             Q_{a,b}(0)=\one_{\{a=b\}},   & a,b\in\stateSS,
            \\
             Q_{a,b;k}(0)=\one_{\{a=b\}},   & a,b\in\stateSS, \  k\in\Nmc.
        \end{cases}
    \end{align}
\end{theorem}
The proof of Theorem \ref{thm:SEIR_ODE} is given in Appendix \ref{sec:proof-SEIR}.
We conclude this section by outlining how the last two theorems are used to prove Theorem \ref{thm:SEIR_main}.

 \begin{proof}[Proof of Theorem \ref{thm:SEIR_main}]
By Theorem \ref{thm:hydrodinamic}, $\lim_{n\rightarrow\infty}\sbar^{G_n}(t)=\P(\Xbar^\Tmc_\root(t)=S)$, $\lim_{n\rightarrow\infty}\ebar^{G_n}(t)=\P(\Xbar^\Tmc_\root(t)=E)$, and $\lim_{n\rightarrow\infty}\ibar^{G_n}(t)=\P(\Xbar^\Tmc_\root(t)=I)$. By Theorem \ref{thm:SEIR_ODE}, we can characterize the transition rates of  $\Xbar^{\Tmc}_\root(t)$, given in \eqref{eq:SEIR-probs} as the solution to the system of ODEs \eqref{eq:ODE-SEIR}-\eqref{eq:ODE-SEIR_initial}. We can solve these ODEs to obtain an expression for $Q_{S,S;k},\ Q_{S,E;k},$ in terms of $g_S,\ g_E$ and $g_I$ (defined in \eqref{eq:SEIR-probs}) which, along with $G_I(t)=\int_0^tg_I(s)ds$ for $t\in[0,\infty)$, by Theorem \ref{thm:gggG}, solve the ODEs \eqref{eq:gggG-Odes}-\eqref{eq:gggG-Odes_initial}. Observing that $Q_{S,b}(t)=\sum_{k\in\natzero}\offspring(k)Q_{S,b;k}(t)$ for all $t\in[0,\infty)$ and $b\in\{S,E\}$, and noting that
\begin{equation*}
    \begin{alignedat}{2}
       & \P(\Xbar^\Tmc_\root(t)=S)& &=s_0 Q_{S, S}(t),
            \\ &\P(\Xbar^\Tmc_\root(t)=E)& &=s_0 Q_{S, E}(t) + e_0 Q_{E, E}(t)
            \\ & \P(\Xbar^\Tmc_\root(t)=I)& &=s_0 Q_{S, I}(t) + e_0 Q_{E, I}(t) + i_0  Q_{I, I}(t)
    \end{alignedat}
\end{equation*}
establishes the theorem.
   \end{proof}

\subsection{Proofs related to the Outbreak Size} \label{sec:proofs_outbreak}
In this section, we prove Theorem \ref{thm:SIR-outbreak} and Theorem \ref{thm:SEIR-outbreak}, which characterize the large $n$ limit of the total fraction of individuals still susceptible at the end of an  SIR or SEIR outbreak on the locally-tree like graph sequences we consider. Recall the standing assumptions made at the beginning of Section \ref{sec:main_proofs}.
We start by introducing some notation to simplify the exposition. First, define 
\begin{equation}
    \label{eq-undd}
  \underline{d}_\sizeb:=\min\{d\in\natzero\ : \ \sizeb(d)>0\}, 
\end{equation}
and recalling that $M_\sizeb$  is the moment generating function of $\sizeb$, set
\begin{align}\label{eq:def_Phi}
    \begin{split}
        \Phi(z):=\Phi_\sizeb(z):= \frac{M_\sizeb'(- z)}{M_\sizeb(- z)},
        \quad z \in [0,\infty). 
    \end{split}
\end{align}
For all $z\in[0,\infty)$ we have $M_\sizeb(-z)=\E_\sizeb[e^{-dz }]\leq \E_\sizeb[1]=1$. Furthermore, for $z\in[0,\infty)$, $M_\sizeb'(-z)= \sum_{k=1}^\infty k e^{-k z} \sizeb(k) $, where the interchange of the sum and derivative is justified because $ k e^{-k z} \leq k$ and  $\sizeb$ has finite mean. We start with an elementary lemma.

\begin{lemma} \label{lemm:Phi_props}
     $\Phi_\sizeb:[0,\infty) \rightarrow [0,\infty)$ is continuous and satisfies the following properties:
    \begin{enumerate}[label=(\roman*)]
        \item $\Phi_\sizeb(0)=\E_\sizeb[d]$;
        \item $\lim_{z\rightarrow\infty}\Phi_\sizeb(z)=\underline{d}_\sizeb$;
        \item $\Phi_\sizeb(z)$ is non-increasing in $z\in[0,\infty)$, and strictly decreasing if for every $j\in\natzero$, $\sizeb\neq \delta_{j}$.
    \end{enumerate}
    \end{lemma}
     \begin{proof} 
     The property \textit{(i)}  follows immediately from the relation $\Phi_\sizeb(0)=\E_\sizeb[d]/\E_\sizeb[1]=\E_\sizeb[d]$. 
     
     The stated continuity of $\Phi_\sizeb$ follows from the dominated convergence theorem and the fact that $\sizeb$ has finite mean, which follows from \eqref{eq:sizab} and Assumption \ref{assu:graph_sequence}. 
        In turn, by the  dominated convergence theorem, the latter implies that  $\lim_{z\rightarrow\infty}\E_\sizeb[d e^{-d z}]=0$. If $\underline{d}_\sizeb=0$ then $\lim_{z\rightarrow\infty}\E_\sizeb[e^{-d z}]=\sizeb(0)>0$, and by \eqref{eq:def_Phi} it follows that $\lim_{z\rightarrow\infty} \Phi_\sizeb(z)=0=\underline{d}_\sizeb$. On the other hand, if $\underline{d}_\sizeb>0$, then
        \begin{align}\label{eq:Phi_sizeb(infty)}
\lim_{z\rightarrow\infty}\Phi_\sizeb(z)=  \lim_{z\rightarrow\infty}\frac{\underline{d}_\sizeb e^{- \underline{d}_\sizeb z } + \sum_{j=\underline{d}_\sizeb+1}^\infty j e^{-jz}}{ e^{- \underline{d}_\sizeb z } + \sum_{j=\underline{d}_\sizeb+1}^\infty  e^{-jz}} = \lim_{z\rightarrow\infty} \frac{\underline{d}_\sizeb e^{- \underline{d}_\sizeb z }}{ e^{- \underline{d}_\sizeb z }}= \underline{d}_\sizeb.
        \end{align}
This proves \textit{(ii)}.

Next, observe that \begin{equation*}
    \Phi_\sizeb(z)=-\frac{d}{dz} \log M_\sizeb(-z)= \frac{d}{d(-z)} \log M_\sizeb(-z).
\end{equation*}
Since the moment generating function of any measure in $\Pmc(\R)$ is log-convex (which follows from an application of H\"older's inequality), and strictly log-convex unless the measure is equal to $\delta_x$ for $x\in\R$,  \textit{(iii)} follows. 
\end{proof}

We now prove Theorem \ref{sec:SIR-outbreak}.
 \begin{proof}[Proof of Theorem \ref{thm:SIR-outbreak}]
Let $f_I$ and $f_S$ be as in \eqref{eq:probs} and set $F_I(t)=\int_0^t\beta_sf_I(s)ds$.
By \eqref{eq:s_i_infinity}, $s^{(\infty)}(t)=s_0 M_\offspring(-F_I(t))$. By the dominated convergence theorem,  $z \mapsto M_\offspring(-z)=\sum_{k=0}^\infty\offspring(k)e^{-z k}$ is continuous on $(0,\infty)$. 

We now turn to the study of  the large-time limit of $F_I$. 

By Theorem \ref{thm:SEIR_main}, $(f_S,$ $f_I,$  $F_I)$ satisfy the ODE system  \eqref{eq:f_If_SF_I}.
For any $a\in[0,1]$ and $ b\in(0,\infty)$, the point $(a,0,b)$ is a fixed point of the system. We claim that as $t\rightarrow\infty$, $(f_S(t),$ $ f_I(t),$ $ F_I(t))$ converges to one such point, and then identify the corresponding $b$ as the solution of an equation.
Near any $t\geq 0$ such that $f_I(t)>0$, $F_I$ is strictly increasing,
and thus it is invertible.  Let $F_I(\infty):=\lim_{t\rightarrow \infty}F_I(t)$, which exists since $F_I$ is non-decreasing. We can change variables for $F\in[0,F_I(\infty)] $ and write $\x(\f):= f_I( F_I^{-1}(\f))$ and  $\y(\f):= f_S( F_I^{-1}(\f))$. We write $\beta^\ast$ (resp. $\rho^\ast$) for the composition of $\beta$ (resp. $\rho$) with $F_I^{-1}$.  Recalling the definition of $\Phi$ in \eqref{eq:def_Phi}, we rewrite the first two equations in \eqref{eq:f_If_SF_I} as 
\begin{equation} \label{eq:ODE-xyf}
    \begin{cases}
        \y'= \y (1-\Phi)
        \\ \x'= \y \Phi  - (1+\frac{\rho^\ast}{\beta^\ast}) +\x,
    \end{cases}
\end{equation}

Since $F_I(0)=0$, and $f_S(0)=s_0$, we can solve the first equation to obtain $\log (\y(\f)/s_0)= \f+\log M_\sizeb(-\f)$, which is equivalent to 
\begin{equation}\label{eq:yF}
    y(F)=s_0M_\sizeb(-F)e^F.
\end{equation}  
Substituting this into the second equation in \eqref{eq:ODE-xyf}, we obtain a linear ODE for $\x$. Recalling that $x(0)=f_I(0)=i_0$ and that $i_0+s_0=1$, we solve this equation to obtain
\begin{align}
    \begin{split}
    \label{eq:solve-x}
    \x(\f) &=i_0e^{\f} + e^{\f}\int_0^Fs_0M_\sizeb(-z)\Phi(z)dz - e^{\f}\int_0^{\f} e^{-z}\left(1+\frac{\rho^\ast(z)}{\beta^\ast(z)}\right)dz
    \\ & = i_0e^{\f} + e^{\f}s_0 (1-M_\sizeb(-F)) - e^{\f}\int_0^{\f} e^{-z}\left(1+\frac{\rho^\ast(z)}{\beta^\ast(z)}\right)dz
    \\& = e^{\f}-\y(\f)  -e^{\f} (1-e^{-\f}) - e^{\f}\int_0^{\f}e^{-z}\frac{\rho^\ast(z)}{\beta^\ast(z)}dz     
    \\ &= 1-\y(\f) - e^{\f}\int_0^{\f}e^{-z}\frac{\rho^\ast(z)}{\beta^\ast(z)}dz,
    \end{split}
\end{align} 
where in the second line we used the fact that $M_\sizeb(0)=1$, and in the first and third line we applied \eqref{eq:yF}.

We now claim that \eqref{eq:solve-x} shows that $F_I(\infty)<\infty$.   Since $F_I(t)=\int_0^t \beta_s f_I(s)ds$ and $\beta$ satisfies Assumption \ref{assu:beta_rho}, this implies that  $\lim_{t\rightarrow\infty}f_I(t)=0$.
First, observe that, if there exists $s\in[0,\infty)$ such that $f_I(s)=0$, then, by \eqref{eq:f_If_SF_I}, $f_I(t)=0$ for all $t\geq s$. 
Next, suppose for the sake of contradiction that $F_I(\infty)=\infty$. Then, for all $t\geq 0$, $f_I(t)>0$. 
By Assumption \ref{assu:beta_rho}, it then follows that $\int_0^t e^{-F_I(s)}\rho_s f_I(s) ds>0$ for all  $t>0$. By definition,  $f_S(t)\in[0,1]$, and so $\y(\f)\in[0,1]$ for all $\f\in[0,F_I(\infty))$. In particular, $\liminf_{\f\rightarrow\infty}\y(\f)\geq0$. But letting  $\f \rightarrow \infty$, 
\eqref{eq:solve-x} then implies that $\lim_{\f\rightarrow\infty}\x(\f)=\lim_{t\rightarrow \infty} f_I(t)=-\infty$, which is a contradiction. Therefore, we conclude that $F_I(\infty)<\infty$ and, thus,  $\lim_{t\rightarrow\infty} f_I(t)=0$.

Since $\lim_{t\rightarrow\infty}f_I(t)=0$, by setting $\x(F_I(\infty))=0$ in  \eqref{eq:solve-x}, we obtain 
\begin{equation} \label{eq:y_infinity}
    \y(F_I(\infty))=1-e^{F_I(\infty)}\int_0^{F_I(\infty)} e^{-z}\frac{\rho^\ast(z)}{\beta^\ast(z)}dz =1-e^{F_I(\infty)}\int_0^{\infty} e^{-\int_0^u \beta_\tau f_I(\tau)d\tau}\rho_u f_I(u)du.
\end{equation}
When combined, \eqref{eq:yF} and \eqref{eq:y_infinity} establish \eqref{eq:Finfty_time_vary}. 

If there exists   $r\in(0,\infty)$ such that $\rho_t/\beta_t=r$ for all $t$, then the integral in the rightmost expression in \eqref{eq:Finfty_time_vary} is equal to $r (1-e^{-F_I(\infty)})$, and thus  \eqref{eq:Finfty_time_vary} reduces to \eqref{eq:Finfty_const}.
Let $\Psi_r$ be given by \eqref{eq:Psi_r}.
        Using the fact that moment generating functions are log-convex, it follows that $\Psi_r$ is convex. Furthermore, $\Psi_r$ is continuous on $[0,\log(1+1/r))$, $\Psi_r(0)=\log(s_0)<0$ and $\lim_{z\rightarrow (\log (1 +1/r))^-}\Psi_r(z)=\infty$. Therefore, \eqref{eq:Finfty_const} has a unique positive solution.
This concludes the proof.
  \end{proof}

We conclude this section by providing a similar characterization of the outbreak size for the SEIR process.

 \begin{proof}[Proof of Theorem \ref{thm:SEIR-outbreak}] Let $g_S,$ $g_E,$ $g_I$ be as in \eqref{eq:SEIR-probs}, and set $G_I(t):=\int_0^t\beta_s g_I(s)ds$ for $t\in[0,\infty)$.
Note that $\sbar^{(\infty)}(t)=s_0 M_\offspring(-G_I(t))$ by \eqref{eq:sbar_infty}, and by the dominated convergence theorem, $M_\offspring$ (the moment generating function of $\offspring$) is continuous on $(-\infty,0)$. 

We now study the large-time limit of $G_I$.
By Theorem \ref{thm:gggG},  $(g_S,$ $ g_E,$ $  g_I,$ $ G_I)$ satisfy the system of ODEs \eqref{eq:gggG-Odes}.
Near any $t\geq 0$ such that $g_I(t)>0$, $G_I$ is strictly increasing, and, therefore invertible. Let $G_I(\infty):=\lim_{t\rightarrow \infty}G_I(t)$, which exists since $G_I$ is non-decreasing. We can change variables for $G\in[0,G_I(\infty)] $, write $\x(G):= g_I( G_I^{-1}(G))$, $\z(G):= g_E( G_I^{-1}(G))$ and  $\y(G):= g_S( G_I^{-1}(G))$. We write $\beta^\ast$ (resp. $\rho^\ast$, $\lambda^\ast$) for the composition of $\beta$ (resp. $\rho$, $\lambda$) with $G_I^{-1}$.  By \eqref{eq:gggG-Odes}, letting  apostrophe denote differentiation with respect to $G$, we have
\begin{equation} \label{eq:ODE-xxyg}
    \begin{cases}
        \y'= \y (1-\Phi)
        \\ \z'= \y \Phi  -  \frac{\z\lambda^\ast}{\x\beta^\ast}+\z
        \\ \x'= \frac{\z\lambda^\ast }{\x\beta^\ast}-(1+\frac{\rho^\ast}{\beta^\ast})+\x.
    \end{cases}
\end{equation}

If we let $\bar{\x}=x+z$, then $\y,\bar{\x}$ satisfy 
\begin{equation} \label{eq:ODE-barxyg}
    \begin{cases}
        \y'= \y (1-\Phi)
        \\ \bar{\x}'= \y \Phi  -(1+\frac{\rho^\ast}{\beta^\ast})+\bar{\x}
    \end{cases}
\end{equation}
with $\y(0)=s_0$, $\bar{\x}(0)=1-s_0$, which is the same initial value problem as \eqref{eq:ODE-xyf}. The same argument as that used in the proof of Theorem \ref{thm:SIR-outbreak} can then be used to conclude the proof. 
   \end{proof}

For the sake of completeness, we also include here the special case of the $2$-regular tree (i.e., the infinite line graph), with constant $\rho$ and $\beta$, where we can obtain an explicit expression for $\int_0^t \beta_s f_I(s)ds$ for all $t\in[0,\infty]$. 
\begin{proposition}
Let $T_2=$UGW$(\delta_2)$, and suppose that there exist $r,\ b>0$ such that for all $t\in [0,\infty)$, $\rho_t = r$ and $\beta_t= b$. Then, for all $t\in[0,\infty)$,
    
\begin{align}\label{eq:PSSK-prop}
    \begin{split}
        P_{S,S}(t)=\left(\frac{(1-s_0)e^{- t(b (1-s_0)+r)}+ \frac{r} {b}}{ 1-s_0+\frac{r}{b}}\right)^2,
    \end{split}
\end{align}
    and, hence, 
    \begin{equation}\label{eq:T2_proof}
        \lim_{t\rightarrow \infty} \P(X^{T_2}_\root(t)=S)= s_0 \left(\frac{1}{1+(1-s_0)\frac{b}{r}}\right)^{2}.
    \end{equation} 
\end{proposition}
 \begin{proof} 
    Let $P_{S,S}$, $f_I$ and $f_S$ be as in \eqref{eq:probs}. By Theorem \ref{thm:SIR_ODE},  $\dot P_{S,S}= - b  2 f_I P_{S,S}$, and therefore 
    \begin{equation}\label{eq:PSS-k=2}
        P_{S,S}(t)=\exp(-2 b  \int_0^t f_I(s)ds).
    \end{equation} Setting $\offspring=\delta_2$ and, thus,   $\sizeb=\delta_1$, the first equation in \eqref{eq:f_If_SF_I} reduces to $\dot f_S(t)=0$. Since $f_S(0)=s_0$ and $f_I(0)=i_0=1-s_0$, the second equation in \eqref{eq:f_If_SF_I} reduces to 
\begin{equation}\label{eq:fI_ODE_k=2}
    \dot f_I(t)= -( b  f_I(0)+ r) f_I(t) + b  (f_I)^2.
\end{equation}
This is a Bernoulli equation that can be solved explicitly. The constant $0$ function is a solution. 

For the rest of this proof, we  assume that $f_I(0)\in (0,1)$. Let $ a=-( b  f_I(0)+ r)$, so that \eqref{eq:fI_ODE_k=2} is $\dot f_I =  a f_I + b  f_I^2$. Let $\tau:=\inf\{t>0\ : f_I(t)=0 \}$. For $t\in[0,\tau)$, we can divide both sides of the ODE by $(f_I)^2$.

\begin{equation}\label{eq:ode_k=2}
    \frac{\dot f_I}{(f_I)^2}-
    \frac{ a}{f_I}= b .
\end{equation}

If we set $u(t)=1/f_I(t)$, for $t\in[0,\tau)$,  then $\dot u = - (f_I)^{-2} \dot f_I$ and the ODE in \eqref{eq:ode_k=2} takes the form
\begin{equation*}
    \dot u + a u =-b.
\end{equation*}
This is a linear equation whose explicit solution is
\begin{equation}\label{eq:u(t)}
    u(t)=\frac{ b }{ a}(e^{-t a}-1 )+ e^{-t a}u(0),
\end{equation}
which does not blow up in finite time, and therefore $\tau=\infty$.
Since $f_I(t)=1/u(t)$, \eqref{eq:u(t)} implies 
\begin{equation*}
    f_I(t)=\frac{f_I(0)(f_I(0)+\frac{ r}{ b })}{f_I(0)+\frac{ r}{ b } e^{( b  f_I(0)+ r)t}},
\end{equation*}
which can be integrated to conclude that
\begin{equation*}
\int_0^tf_I(s)d s=t\left(f_I(0)+\frac{ r}{ b }\right) -\frac{1}{ b }\log\left(f_I(0)+\frac{ r}{ b }e^{(b  f_I(0)+ r)t}\right) + \frac{1}{ b }\log\left(f_I(0)+\frac{ r}{ b }\right).
 \end{equation*}
This, combined with \eqref{eq:PSS-k=2}, yields \eqref{eq:PSSK-prop}. Since $\P(X_\root^{T_2}(t)=S)=s_0P_{S,S}(t)$, letting $t\rightarrow\infty$, we obtain \eqref{eq:T2_proof}. \end{proof}

\section{Proof of the Conditional Independence Property} \label{sec:cond_ind_proof}

In Section \ref{sec:proof_ind}, we prove the conditional independence property  stated in Proposition \ref{prop:cond_ind_prop} and the symmetry property stated in Proposition \ref{prop:symmetry}.  The proof relies on a certain change of measure result established in \cite{Ganguly2022interacting}, which we first summarize in Section \ref{sec:change-measure}. Throughout, $\offspring\in\Pmc(\natzero)$ has finite third moment, $\Tmc$ is a UGW($\offspring$) tree, $\alpha\in[0,1]$ is an interpolation parameter, the rates $\beta$, $\lambda$ and $\rho$ satisfy Assumption \ref{assu:beta_lambda_rho}, and $\xi^{\Tmc,\alpha}$ is the hybrid S(E)IR process solving  \eqref{eq:xi-SDE}, with initial states satisfying Assumption \ref{assu:initial_z}.
\subsection{A Radon–Nikodym derivative} \label{sec:change-measure}
We let $\mu:=\law(\xi^{\Tmc,\alpha})$ and $\mu_{t-}=\law(\xi^{\Tmc,\alpha}[t))$ for $t\in(0,\infty)$.  
Given $U\subset\V$ and $t\in(0,\infty)$, let $\Dmc^U_{t-}:=\Dmc([0,t):\stateSS_\extra^U)$ be the set of càdlàg functions $[0,t)\rightarrow \stateSS_\extra^U$. We start with two technical definitions. 

\begin{definition}
Given $y\in\Dmc_\extra$ and $t\in[0,\infty)$ we let $\text{Disc}_t(y):=\{s\in(0,t]\ : \ y(s-)\neq y(s)\}$.
    We say that $x\in\Dmc_\extra^\V$ is \textit{proper} if for every $u,v\in\V$ and $t\in[0,\infty)$, $\text{Disc}_t(x_v)\cap \text{Disc}_t(x_u) =\emptyset$.
\end{definition}

\begin{definition}
    Fix $U\subset \V$ finite,  $t\in(0,\infty)$ and suppose that $x\in\Dmc_{t-}^U$ is proper and $\text{Disc}_\infty(x):=\cup_{s\in(0,\infty)}\text{Disc}_{s}(x)$ can be ordered as a strictly increasing sequence $\{t_k(x)\}$. Then the \textit{jump characteristics} of $x$ are the elements $\set{ (t_k(x),j_k(x),v_k(x))}\subset (0,\infty)\times \Jmc \times U$ where for each $k\in\N$ with $k\leq |\text{Disc}_\infty(x)|$, $v_k=v_k(x)$ is a vertex in $U$ such that $x_{v_k}$ is discontinuous at time $t_k(x)$, and $j_k(x)$ is the size of the jump $x_{v_k}(t_k(x))-\lim_{h\rightarrow0^+}x_{v_k}(t_k(x)-h)$.
\end{definition}
Given $U\subset \V$ and $t\in(0,\infty)$, we also define a function $\psi$ from the set of functions $[0,t)\rightarrow \stateSS$ into $\{0,1\}$ by
\begin{equation}
    \label{eq:psi}
    \psi(x)= \one_{\set{x\in \Dmc_{t-}^U}} \one_{\{\{v\in U \ :\ x_v(0)\neq \extra\} \text{ is a locally finite tree}\}}.
\end{equation}
We also recall that $\V_n=\{\root\} \cup (\cup_{k=1}^n \N^k)$.

We now state a change of measure result that is established, for general interacting jump processes, in \cite{Ganguly2022interacting}. In the sequel, the exact definition of the reference processes $\xihat^n$ presented below will not be important, and so we state the following proposition to summarize some key properties we use.

\begin{proposition} \label{prop:derivatives}
 For each $n\in\N \setminus\set{1},$ there exists an $\stateSS_\extra^\V$-valued process $\xihat^n:=\xihat^{n,\alpha}$ such that for any $t\in(0,\infty)$, $A,B\subset \V$ with 
 $\partial^\V A \subset \VV_{n-1}$  and $(A\cup\partial^\V A)\cap B=\emptyset$, 
\begin{equation}\label{eq:cond_ind_reference}
     \xihat^n_A[t) \perp \xihat^n_B[t) \ | \ \xihat^n_{\partial^\V A}[t).
 \end{equation}
 Furthermore,  $\xihat^n_{\V_n}$ is almost surely proper and its jump characteristics $\{(t^n_i,v^n_i,j^n_i)\}$ are well-defined.  Moreover, for every $t\in(0,\infty)$,   $\muhat^n_{t-} :=\law(\xihat^n[t))$  has the property that, almost surely
 \begin{equation}\label{eq:derivative}
     \frac{d\mu_{t-}}{d\muhat_{t-}^n}(\xihat^n[t))=\psi(\xihat^n[t)])\exp\left( -\sum_{ \substack{v\in \V_n\\ j=1,2}}\int_{(0,t)}(q_\alpha(j,s,\xihat^\partial_v,\xihat^n_{\partial_v})-1)ds\right)\prod_{0<t_i^n<t} q_\alpha(j^n_i,t_i^n,\xihat^n_{v_j^n},\xihat^n_{\partial_{v_i^n}}),
 \end{equation}
 where $\psi$ is defined in \eqref{eq:psi}, and $q_\alpha$ is given in \eqref{eq:S(E)IR-rate}.
\end{proposition}
\begin{proof}
    An explicit definition of the processes $\xihat^n$ as a solution of a SDE related to \eqref{eq:xi-SDE} is given in \cite[(4.3)]{Ganguly2022interacting} by substituting the rate function $q_\alpha$ to the rate functions $r^v$ used therein. Assumption 4.1 in \cite{Ganguly2022interacting}, i.e., the well-posedness of $\xihat^n$, follows from an application of \cite[Theorem C.2]{Ganguly2022hydrodynamic} on observing that \cite[Assumption C.1]{Ganguly2022hydrodynamic} holds by Assumption \ref{assu:graph_sequence}, the definition of $q_\alpha$ in \eqref{eq:S(E)IR-rate}, and the form of the driving noises in \eqref{eq:xi-SDE}.
    Assumption \ref{assu:beta_lambda_rho}
implies that  \cite[Assumption 3.1, Assumption 3.4]{Ganguly2022interacting} holds with $q_\alpha$ in place of $r^v$. 

Then, \cite[Proposition 4.4]{Ganguly2022interacting} establishes \eqref{eq:cond_ind_reference}.  By \cite[Lemma 4.8]{Ganguly2022interacting}, $\xihat^n_{\V_v}$ is proper. Finally,  \eqref{eq:derivative} holds by \cite[Corollary 4.11]{Ganguly2022interacting}.
\end{proof}

Given $n\in\N\setminus\{1\}$, $A\subset\V_n$ and $t>0$ we define
\begin{equation}
    \label{eq:E_n}
    \Emc_n^t(A ):=\{\xihat^n_{\partial^\V A}(t-)=S_{\partial^\V A }\}=\bigcap_{v\in\partial^\V A}\{\xihat^n_{v}(t-)=S\}.
\end{equation}
\subsection{Proof of Proposition \ref{prop:cond_ind_prop}}\label{sec:proof_ind} We start by establishing a factorization result for the Radon-Nikodym derivative established in \ref{sec:change-measure}. We recall that $\partial^\V_v$ denote the neighborhood in $\V$ of $v\in\V$. We also set $C_v:=C_v^\V=\{vk\}_{k\in\N}$. 

\begin{lemma} \label{lemm:density_factors}   Let $n\in\N\setminus\{1\}$ and fix
    $A,B\subset \V$ with $\partial^\V A \subset \VV_{n-1}$  and such that $A,\ \partial^\V A$ and $B$ form a partition of $\V$.  Let $\muhat^n_{t-}$ and $ \xihat^n_{t-} $ be as in Proposition \ref{prop:derivatives}. Then there exist measurable functions $\ftil_1:\Dmc_{t-}^{A\cup\partial^\V A}\rightarrow[0,\infty)$ and   $\ftil_2:\Dmc_{t-}^{B\cup\partial^\V A}\rightarrow[0,\infty)$
     such that for every $t\in(0,\infty),$
 \begin{equation}\label{eq:derivative-factors}
     \frac{d\mu_{t-}}{d\muhat^n_{t-}}(\xihat^n[t)) = \ \ftil_1(\xihat^n_{A\cup \partial^\V A}[t))\  \ftil_2(\xihat^n_{B\cup \partial^\V A}[t)), \qquad \text{a.s. on } \{\xihat^n_{\partial^\V A}(t-)=S_{\partial^\V A}\}.
 \end{equation} 
 
 \end{lemma}
  \begin{proof} 
  Fix $n,\ A,\ B$ as in the statement of the lemma, and $t\in(0,\infty)$. Set $\Emc_n:=\Emc_n^t(A)$, where the latter is defined in \eqref{eq:E_n}. For $v\in\V_n$ and $x\in \Dmc_{t-}^{\V_n}$  proper define
  \begin{equation}\label{eq:gamma_v}
\gamma_v(x_v,x_{{\partial^\V_w}}):=\left[  \prod_{0<t_k(x_v)<t}q_\alpha(j_k(x_v),t_k(x_v),x_{v},x_{{\partial^\V_w}})\right] e^{ - \sum_{j=1,2}\int_{(0,t)}\left(q_\alpha(j,s,x_v,x_{{\partial^\V_w}})-1\right)ds},
     \end{equation}
where $ \set{(t_k(x_v),v,j_k(x_v))}$ are the jump characteristics of $x_v$. When $x\in\Dmc^{\V_n}_{t-}$ is not proper, set $ \gamma_v(x_v,x_{{\partial^\V_w}}):=0$. Also, for $v\in\V$,  $y: [0,t)\rightarrow \stateSS_{\extra}$, $b\in\stateSS_{\extra}$ and $z\in (\stateSS_{\extra})^\infty$, define 
\begin{equation}  \label{eq:psi_v}
    \psi_v(y,b,z):=\begin{cases}
        1-\one_{\{y(0)\neq \extra,\  b=\extra\}} & \text{if } v=\root, \ y\in \Dmc_{t-}, \ |\set{ \kappa \in \N  : \ z_\kappa\neq \extra}| < \infty,
        \\ \one_{\{y(0)\neq \extra\}} & \text{if } v\neq\root, \ y\in \Dmc_{t-}, \ |\set{ \kappa \in \N  : \ z_\kappa\neq \extra}| < \infty,
        \\ 0 & \text{otherwise.}
    \end{cases}
\end{equation}

By Proposition \ref{prop:derivatives}, the jump characteristics of $\xihat^n_{\V_n}$ are almost surely well-defined. On the event  that they exist, the jump characteristics of $\xihat^n_{\V_n}$ are a disjoint union of those of $\xihat^n_v$ for $v\in\V_n$.  We can then rewrite \eqref{eq:derivative} as
\begin{equation*}
    \frac{d\mu_{t-}}{d\muhat^n_{t-}}(\xihat^ n[t)])= \prod_{v \in \VV_n} \gamma_v(\xihat^{\partial^\V_w}[t),\xihat^n_{{\partial^\V_w}}[t)) \ \psi_v(\xihat^n_v[t),\xihat^n_{\pi_n}(0),\xihat^n_{C_v}(0))\ \text{a.s.}
\end{equation*}
Since $A,\ \partial^\V A$ and $B$ forms a  partition of $\V$, we can  further decompose the right-hand side as
\begin{equation}\label{eq:der-proof}
    \frac{d\mu_{t-}}{d\muhat^n_{t-}}(\xihat^n[t))=  \prod_{F\in\{A,\partial^\V A, B\}} \left( \prod_{v\in F\cap \V_n} \gamma_v(\xihat^n_{{\partial^\V_w}}[t),\xihat^n_{{\partial^\V_w}}[t)) \ \psi_v(\xihat^n_v[t)],\xihat^n_{\pi_v}(0),\xihat^n_{C_v}(0)) \right) \ \text{a.s.},
\end{equation}
where  for ease of notation in the sequel, we set $\pi_\root=\root$, which we can be done in \eqref{eq:der-proof} since $\psi_\root(y,b,z)$ does not depend on $b$.
The product in the inner bracket is a function of $\xihat^n_{A\cup \partial^\V A}$ when  $F=A$, and a function of $\xihat^n_{B \cup \partial^\V A}$ when $F=B$. Thus, to prove \eqref{eq:derivative-factors} it suffices to show that for each $w\in\partial^\V A$, there
exist measurable functions $\ftil^w_1:\Dmc_{t-}^{A\cup\partial^\V A}\rightarrow[0,\infty)$ and   $\ftil^w_2:\Dmc_{t-}^{B\cup\partial^\V A}\rightarrow[0,\infty)$ such that almost surely on $\Emc_n$  
\begin{equation} \label{eq:w-factors}
    \gamma_w(\xihat^n_{{\partial^\V_w}}[t),\xihat^n_{{\partial^\V_w}}[t)) \ \psi_w(\xihat^n_w[t)],\xihat^n_{\pi_w}(0),\xihat^n_{C_w}(0)) = \ftil^w_1(\xihat^n_{A\cup\partial^\V A}[t)) \ \ftil_2^w(\xihat^n_{B\cup \partial^\V A}[t)) .
\end{equation}

By the monotonicity of the S(E)IR dynamics given in \eqref{eq:mono}, on $\Emc_n$ the set of times $\{t_i(\xihat^n_w) < t \ : w\in\partial^\V A \}$ given by the jump characteristics of $\xihat^n_w$ is empty. Hence, almost surely, \begin{equation}
\label{eq:prod1}\one_{\Emc_n}\prod_{t_i(\xihat^n_w)<t}q_\alpha(j_i(\xihat^n_w),t_i(\xihat^n_w),\xihat^n_w,,\xihat^n_{{\partial^\V_w}}) = \one_{\Emc_n}.
\end{equation}
Recalling the identification of the states $(S,E,I,R)$ with $(0,1,2,3)$ and the definition  $q_\alpha$ in \eqref{eq:S(E)IR-rate}, for $w\in\partial^\V A$ and $s\in(0,t)$, on the event $\Emc_n$ we have 
\begin{align}\label{eq:der-proof-2}
    \begin{split}
q_\alpha(j,s,\xihat^n_w,\xihat^n_{{\partial^\V_w}})&= {\beta_t}\infects (\xihat^n_{{\partial^\V_w}}(s-))(\alpha\one_{\{j=1\}\}}+(1-\alpha)\one_{\set{j=2}})
        \\ &={\beta_t} \left(\sum_{\wbar \sim w}\one_{\{\xihat^n_{\wbar}(s-)=2\}}\right)(\alpha\one_{\{j=1\}\}}+(1-\alpha)\one_{\set{j=2}}).
    \end{split}
\end{align} 
For convenience of notation, set $\alpha(j):=\alpha\one_{\{j=1\}}+(1-\alpha)\one_{\set{j=2}}$. Then for $w\in\partial^\V A$, using  first  \eqref{eq:gamma_v} and \eqref{eq:prod1}, and then  \eqref{eq:der-proof-2} and the fact that $\partial^\V_{w}\subset A\cup B$,
\begin{align}\label{eq:fact_gamma_w}
    \begin{split}
        \one_{\Emc_n} & \gamma_w(\xihat^n_w,\xihat^n_{{\partial^\V_w}})
         \\ &=\one_{\Emc_n} \exp\left(-\sum_{j=1,2}\alpha(j)\int_{(0,t)}\left(\infects (\xihat^n_{{\partial^\V_w}}(s-))-1\right)ds\right)
\\ & =\one_{\Emc_n} \exp\left(\int_{(0,t)} \left(\infects(\xihat^n_{{\partial^\V_w}\cap A}(s-)) + \infects (\xihat^n_{{\partial^\V_w}\cap B }(s-)) -1\right)ds \right)
\\& =\one_{\Emc_n}  \exp\left(-\int_{(0,t)}\left(\infects(\xihat^n_{{\partial^\V_w}\cap A}(s-)) \right)ds \right) \exp\left(-\int_{(0,t)} \left(\infects(\xihat^n_{{\partial^\V_w}\cap B }(s-)) -1\right)ds\right),
    \end{split}
\end{align}
which shows that each $\gamma_w$ term in \eqref{eq:w-factors} admits the desired factorization. It only remains to show that the same holds for the $\psi_w$ term  in \eqref{eq:w-factors}. To this end, note that for $w\in\partial^\V A \setminus \{\root\}$, by \eqref{eq:psi_v},
\begin{align}\label{eq:fact_psi_w}
    \begin{split}
    &\psi_w(\xihat^n_w[t)],\xihat^n_{\pi_w}(0),\xihat^n_{C_w}(0))
        \\ &= \one_{\{\xihat^n_w[t)\in\Dmc_{t-}\}}  \one_{\{|\set{ v \in C_w  : \ \xihat^n_{w}(0)\neq \extra}| < \infty\}} (1- \one_{ \{\xihat^n_w(0)\neq \extra, \ \xihat^n_{\pi_w}(0)=\extra\}} )
         \\ &= \one_{\{\xihat^n_w[t)\in\Dmc_{t-}\}}  \one_{\{|(A\cup\partial^\V A)\cap\set{ v \in C_w  : \ \xihat^n_{v}(0)\neq \extra}| < \infty\}}\one_{\{|(B\cup\partial^\V A)\cap\set{ v \in C_w  : \ \xihat^n_{v}(0)\neq \extra}| < \infty\}} \one_{ \{\xihat^n_w(0)\neq \extra,  \xihat^n_{\pi_w}(0)=\extra)\}^c}
        \\ &= 
\psitil^{(1)}_w(\xihat^n_w[t),\xihat^n_{\pi_w}[t))  \
\psitil^{(2)}_w(\xihat^n_{A\cup\partial^\V A}[t))\  
        \psitil^{(3)}_w(\xihat^n_{B\cup \partial^\V A}[t)),
    \end{split}
\end{align}
where $\psitil^{(1)}_w$ is the product of the first and last term in the penultimate line of the display, and $\psitil^{(i)}_w$, $i=2,3$, is the $i$-th term of the  penultimate  line of the display.
Since $w\in\partial^\V A$ and $\V$ is a tree,  either $\{w,\pi_w\}\subset B\cup\partial^\V A$ or  $\{w,\pi_w\}\subset A\cup\partial^\V A$. Hence, the last line of \eqref{eq:fact_psi_w} factors as desired.
Similarly, if $\root\in\partial^\V A$, by \eqref{eq:psi_v} we have
\begin{align}\label{eq:fact_psi_root}
    \begin{split}
        & \psi_\root(\xihat^n_{\root}[t)],\xihat^n_{\pi_\root}(0),\xihat^n_{C_\root}(0))
        \\ &=  \one_{\{\xihat^n_{\root}[t)\in\Dmc_{t-}\}}  \one_{\{|\set{ v \in C_\root  : \ \xihat^n_{\root}(0)\neq \extra}| < \infty\}} \one_{\{\xihat^n_{\partial_\root}(0)\neq \extra\}}
         \\ & =  \one_{\{\xihat^n_{\root}[t)\in\Dmc_{t-}\}}  \one_{\{|(A\cup\partial^\V A)\cap\set{ v \in C_\root  : \ \xihat^n_{v}(0)\neq \extra}| < \infty\}}\one_{\{|(B\cup\partial^\V A)\cap\set{ v \in C_\root  : \ \xihat^n_{v}(0)\neq \extra}| < \infty\}} 
         \one_{\{\xihat^n_{\root}(0)\neq \extra\}}
        \\ & = \psitil^{(4)}(\xihat^n_{\root}[t))\ \psitil^{(5)}(\xihat^n_{A\cup\partial^\V A}[t))\  
        \psitil^{(6)}(\xihat^n_{B\cup \partial^\V A}[t)), 
    \end{split}
\end{align}
where $\psitil^{(4)}$ is the product of the first and last term in the penultimate line of the display, and $\psitil^{(i)}$ with $i=5,6$ is the $(i-3)$-th term of the penultimate line of the display.  Together \eqref{eq:fact_gamma_w}, \eqref{eq:fact_psi_root} and \eqref{eq:fact_psi_w} prove  \eqref{eq:w-factors} and, hence, $\one_{\Emc_n} d\mu_{t-}/d\muhat_{t-}^n$ admits the factorization stated in \eqref{eq:derivative-factors}.
     \end{proof}

We conclude this section by proving Proposition \ref{prop:cond_ind_prop}.
\begin{proof}[Proof of Proposition \ref{prop:cond_ind_prop}]
    Throughout the proof we fix $\alpha\in[0,1]$,  $\offspring\in\Pmc(\natzero)$ and $\Tmc=\text{UGW}(\offspring$), and we omit the dependence of $\xi^{\Tmc,\alpha}$ on them.
    Let $\{A,B,\partial^\V A\}$ be a partition of $\V$ with $\partial^\V A$ being finite. Pick $n\in\N$ such that  $\partial^\V A \subset \VV_{n-1}$. We define $\Amc:\Dmc_\extra^{\partial^\V A}\rightarrow\{0,1\}$ by $\Amc(x)=\one_{\set{x_v(t-)=S, \  v\in\partial^\V A}}$. We observe that $\Amc(\xihat^n_{\partial^\V A})=\one_{\Emc_{n}^t(A)}$, where the latter is defined in \eqref{eq:E_n}. Let $W$ be a bounded, $\sigma(\xi_A[t))$-measurable random variable. Adopting the convention $0/0=0$, and using Lemma \ref{lemm:density_factors} and Bayes's theorem in the first line, and the property \eqref{eq:cond_ind_reference} in the last line, we have 
\begin{align*}
    \begin{split}
        \E_{\mu}[W \Amc(\xi_{\partial^\V A}) \ |\ \xi_{\partial^\V A}[t),\ \xi_{B}[t) ] =& \frac{ \E_{\muhat^n}[W \Amc(\xihat^n_{\partial^\V A}) \frac{d\mu_{t-}}{d\muhat^n_{t-}}\xihat^n[t)\ |\ \xihat^n_{\partial^\V A}[t), \xihat^n_{B}[t) ]}{  \E_{\muhat^n}[ \frac{d\mu_{t-}}{d\muhat^n_{t-}}\xihat^n[t) |\ \xihat^n_{\partial^\V A}[t), \xihat^n_{B}[t) ]}
        \\=  &\one_{\Emc_n^t(A)}\frac{ \E_{\muhat^n}[W  \ftil_1(\xihat^n_{A\cup \partial^\V A}[t) ) \ftil_2(\xihat^n_{B\cup \partial^\V A}[t))\ |\ \xihat^n_{\partial^\V A}[t), \xihat^n_{B}[t) ]}{  \E_{\muhat^n}[ \ftil_1(\xihat^n_{A\cup \partial^\V A}[t))  \ftil_2(\xihat^n_{B\cup \partial^\V A}[t))\ |\ \xihat^n_{\partial^\V A}[t), \xihat^n_{B}[t) ]}
        \\ =&  \one_{\Emc_n^t(A)}\frac{ \ftil_2(\xihat^n_{B\cup \partial^\V A}[t)) \E_{\muhat^n}[W  \ftil_1(\xihat^n_{A\cup \partial^\V A}[t))  \ |\ \xihat^n_{\partial^\V A}[t), \xihat^n_{B}[t) ]}{ \ftil_2(\xihat^n_{B\cup \partial^\V A}[t))  \E_{\muhat^n}[ \ftil_1(\xihat^n_{A\cup \partial^\V A}[t)) \ |\ \xihat^n_{\partial^\V A}[t), \xihat^n_{B}[t) ]}
      %  \\ =& \one_{\Emc_n^t(A)}\frac{ \E_{\muhat^n}[W  \ftil_1(\xihat^n_{A\cup \partial^\V A}[t))  \ |\ \xihat^n_{\partial^\V A}[t), \xihat^n_{B}[t) ]}{  \E_{\muhat^n}[ \ftil_1(\xihat^n_{A\cup \partial^\V A}[t)) \ |\ \xihat^n_{\partial^\V A}[t), \xihat^n_{B}[t) ]}
         \\ =& \one_{\Emc_n^t(A)}\frac{ \E_{\muhat^n}[W  \ftil_1(\xihat^n_{A\cup \partial^\V A}[t))  \ |\ \xihat^n_{\partial^\V A}[t) ]}{  \E_{\muhat^n}[ \ftil_1(\xihat^n_{A\cup \partial^\V A}[t)) \ |\ \xihat^n_{\partial^\V A}[t)]}.
    \end{split}
\end{align*}
The last quotient is $\xihat^n_{\partial^\V A}[t)$-measurable. As this holds for every bounded $\xi_A[t)$-measurable random variable $W,$ we conclude that 
\begin{equation*}
     \E_{\mu}[W \Amc(\xi_{\partial^\V A}) \ |\ \xi_{\partial^\V A}[t),\ \xi_{B}[t)] = \E_{\mu}[W \Amc(\xi_{\partial^\V A})  \ |\ \xi_{\partial^\V A}[t)].
\end{equation*}
Since $\Amc(\xi_{\partial^\V A})= \one_{\{\xi_v(t-)=S \ \forall v\in\partial^\V A\}}$, if follows that
\begin{equation} \label{eq:conditional_indep_prop_inproof}
        \xi_A[t) \perp \xi_B[t) \ |\  \xi_{\partial^\V A}(t-)=S_{\partial^\V A},
    \end{equation}
 which proves the first assertion of the proposition. 

 Next, let $v\in\V$ and $\kappa\in\natzero$. Set $\dtil_v:=d_v -\one_{\{v\neq\root\}}=|C_v\cap\Tmc|$. The event $\{\dtil_v=k\}$ is clearly the same as the event $\set{d_v=\kappa+\one_{\{v\neq\root}\}}$. In the sequel, we condition on the event $\dtil_v=\kappa$. If $\kappa=0$, the set $\{\xi_{vi}(t)\}_{i=1,...,\dtil_v}$ of children of $v$ in $\Tmc$ is empty, and if $\kappa=1$, the set of children of $v$ in $\Tmc$ is a singleton. In either case, the stated conditional independence holds trivially. Suppose that $\kappa\geq 2$ and for $i=1,...,\kappa$, we fix $x_i\in \stateSS$. For $w\in\V$, let $\T_w:=\{wu\}_{u\in\V}$ denote the subtree of $\V$ rooted at $w$. By the definition of the jump rate $q_\alpha$ in \eqref{eq:S(E)IR-rate}, and the SDE characterization \eqref{eq:xi-SDE} from Lemma \ref{lem:S(E)IR-exist}, $\dtil_v=\kappa$ if and only  $\xi(0)_{vm}\neq \extra$ for all $m\in\N$ with $m\leq\kappa$ and  $\xi_{mk}(0)=\extra$ for all $k\in\N$ with $k>\kappa$. Using first this fact, and then applying \eqref{eq:conditional_indep_prop_inproof} with $A=\cup_{i=1}^\kappa \T_{vi}$, $\partial^\V A =\{v\}$ and $B=\{v\ell\ : \  \ell\in \N\cap(\kappa,\infty))$, we have  
    \begin{align}\label{eq:pmf_children}
        \begin{split}
            &\P(\xi_{v\ell}(t)=x_{\ell}\text{ for } 1\leq\ell\leq \kappa\ |\  \xi_v(t)=S,\ \dtil_v=\kappa)
            \\ &= \P(\xi_{v\ell}(t)=x_{\ell}\text{ for } 1\leq\ell\leq \kappa\ | \ \xi_v(t)=S,\ \xi_{vm}(0)\neq \extra \ \text{for } 1\leq m \leq \kappa,\ \xi_{vk}(0)=\extra \text{ for } k > \kappa)
           % \\ &= \P(\xi_{v\ell}(t)=x_{\ell}\ | \xi_v(t)=S,\ \xi_{vm}(t)\neq \extra \ \text{for } 1 \leq m \leq \kappa,\ \xi_{vk}(t)=\extra \text{ for } k > \kappa)
            \\ &= \P(\xi_{v\ell}(t)=x_{\ell} \text{ for } 1\leq\ell\leq \kappa\ |\  \xi_v(t)=S,\ \xi_{v m}(0)\neq \extra \ \text{for } 1 \leq m \leq \kappa).
        \end{split}
    \end{align}
   For each $\ell=1,...,\kappa$,  another application of \eqref{eq:conditional_indep_prop_inproof}, with $A=\T_{v\ell}$, $\partial^\V A=\{v\} $ and $B=\{v m\ : \ m\in\N\setminus\{\ell\})$ yields
   \begin{equation*}
        \label{eq:c_iid_}
        \xi_{v\ell}[t) \perp \{\xi_{vm}[t)\}_{m\in\N\setminus\{\ell\}}\ | \xi_v(t-)=S.
    \end{equation*}   
    It follows that for $\ell\in\N\cap[1,\kappa]$ and  $M\subset\{m\in\N\ :\  m\leq\kappa \}\setminus\{\ell\}$, 
\begin{align} \label{eq:cond_ind_chil_iterate}
    \begin{split}
        &\P(\xi_{v\ell}(t)=x_\ell |\  \xi_v(t)=S,\ \xi_{vi}(t)\neq \extra \ \text{for } 1\leq i\leq \kappa,\  \xi_{vm}=x_m \ \forall m\in M)
        \\&= \P(\xi_{v\ell}(t)=x_\ell |\  \xi_v(t)=S,\ \xi_{v\ell}(t)\neq \extra)
        \\&= \P(\xi_{v\ell}(t)=x_\ell |\  \xi_v(t)=S,\ {v\ell}\in\Tmc).
    \end{split}
\end{align}
    By iteratively applying \eqref{eq:cond_ind_chil_iterate}, we obtain
     \begin{align*}
        \begin{split}
            & \P(\xi_{v\ell}(t)=x_\ell\ \ \text{for } 1\leq \ell \leq \kappa  \ |\  \xi_v(t)=S,\ \xi_{vi}(t)\neq \extra \ \text{for } 1\leq i\leq \kappa)
             \\ & = \prod_{\ell=1}^{\kappa}\P(\xi_{v\ell}(t)=x_\ell\ | \xi_v(t)=S, \ v\ell\in \Tmc),
        \end{split}
    \end{align*}
    which along with \eqref{eq:pmf_children} establishes the second assertion of the theorem and concludes the proof.
\end{proof}
 
We conclude this  section by using Proposition \ref{prop:cond_ind_prop} and properties of the SDE \eqref{eq:xi-SDE} to derive Proposition \ref{prop:symmetry}.

\begin{proof}[Proof of Proposition \ref{prop:symmetry}] Fix $\vtil\in\V\setminus\{\root\}$ and, on the event $\vtil\in\Tmc$, let $\Tmctil$ be the subtree of $\Tmc$ rooted at $\vtil$, i.e., $\Tmctil:=\Tmc\cap \T_{\vtil}$ where $\T_{\vtil}:=\{\vtil w \ : \ w\in\V\}$. By Assumption \ref{assu:initial_z}, $\Tmc_{\vtil}$ is a Galton-Watson tree with offspring distribution $\offspring$ on the event $\{\vtil\in\Tmc\}$. Recall that $\xi^{\Tmc,\alpha}$ satisfies the SDE \eqref{eq:xi-SDE}, and define the modified process $\xitil$ on $\stateSS_\extra^\V$ by  $t\in[0,\infty)$,
\begin{equation}\label{eq:xitil}
\xitil_v(t)=
    \begin{cases}
        \xi^{\Tmc,\alpha}_v(t) & v\in\T_{\vtil}
        \\ \extra & v\in\V\setminus\T_{\vtil}.
    \end{cases}
\end{equation}
    Fix $t\in[0,\infty)$, and let $\Emc^t:=\{\xi^{\Tmc,\alpha}_\vtil(t-)=S\}$. By \eqref{eq:xi-SDE} and Assumption \ref{assu:initial_z}, $\xi^{\Tmc,\alpha}_{\vtil}(t-)=S$ implies $\xi^{\Tmc,\alpha}_{\vtil}\neq\extra $ and hence,  that
    \begin{equation}\label{eq:Emctilde}
    \Emc^t=\tilde{\Emc}^t:=\{\xitil_{\vtil}
    (t-)=S,\  \vtil \in\Tmc \} = \{\xitil_{\vtil}(s)=S\ : \ s\in[0,t) \},
    \end{equation}
    where the second equality follows from the monotonicity property of the S(E)IR process, see \eqref{eq:mono}. Applying Proposition \ref{prop:cond_ind_prop} with $A=\T_{\vtil}\setminus\{\vtil\}$, $\partial^\V A=\{\vtil\}$ and B=${\V\setminus\T_{\vtil}}$, it follows that $\xi^{\Tmc,\alpha}_{\T_{\vtil}\setminus\{\vtil\}}[t)$ is conditionally independent of $\xi^{\Tmc,\alpha}_{\V\setminus\T_{\vtil}}[t)$ given $\Emc^t$. Thus, by \eqref{eq:xitil} and \eqref{eq:Emctilde},
    \begin{equation}\label{eq:laws_equal}
        \law( \xi^{\Tmc,\alpha}_{\T_{\vtil}\setminus\{\vtil\}}[t)\ | \Emc^t) = \law (\xitil^{\Tmc,\alpha}_{\V\setminus\T_{\vtil}}[t) \ | \tilde{\Emc}^t).
    \end{equation}
Next, fix $m\in\N$, and define a map $\mapvtil_{m}:\T_{\vtil}\rightarrow\V$ given by $\mapvtil_m(\vtil)=\root$ and for $v\in\V$, $\mapvtil_m(\vtil (m v))=1 v$, $\mapvtil_m(\vtil (1 v))=m v$ and $\mapvtil_m(\vtil (\ell v))=\ell v$ for all $\ell\in\N\setminus\{1,m\}$, recalling that $vw$ represent concatenation of $v,w\in\V$. Then $\mapvtil_m$ defines an isomorphism of the rooted graphs $(\T_{\vtil},\vtil)$ and $(\V,\root)$. It follows from \eqref{eq:xitil} and the form of the SDE \eqref{eq:xi-SDE} that 
\begin{equation*}
    \law(\xitil_{\T_{\vtil}} [t)\ |\  \tilde{\Emc}^t)=\law ( \xi^{\Tmc,\alpha}_{\V}[t)\ |\ \xi^{\Tmc, \alpha}_\root(t-)=S)
\end{equation*}
and in particular, 
\begin{equation}\label{eq:sym-proof_4}
    \law(\xitil_{\vtil m}[t)]\ | \xitil_{\vtil}(t-)=S, \vtil m\in\Tmctil) 
       =\law(\xi^{\Tmc,\alpha}_{1}[t)]\ | \xitil_{\root}(t-)=S, 1 \in\Tmc), 
\end{equation}
which follows from the independence of $\Tmc$ from the driving Poisson processes in the SDE \eqref{eq:xi-SDE} for $\xi^{\Tmc,\alpha}$ and the fact that $\vtil m\in\Tmc$ implies  $\vtil\in\Tmc$. By the well-posedness of the SDE \eqref{eq:xi-SDE} established in Lemma \ref{lem:S(E)IR-exist}, it follows that the left-hand side of \eqref{eq:sym-proof_4} does not depend on the choice of $\vtil \in\V$ and $m\in\N$, thus proving the proposition.
\end{proof}
 %%%%%%%%%%%%%%%%%%%%

\appendix
\section{Proofs of intermediate SEIR dynamics Results}\label{sec:proof-SEIR}
In this section, we prove Theorem \ref{thm:SEIR_ODE} and Theorem \ref{thm:gggG}, thus completing the proof of Theorem \ref{thm:SEIR_main}. 
Throughout,  $\offspring\in\Pmc(\natzero)$,  $\sizeb$ is the size-biased version of $\offspring$, as defined in \eqref{eq:sizab}, and $\Tmc$ is a UGW($\offspring$) tree. We  assume that $\offspring$ has finite third moment and, as everywhere else in the paper, we assume that the rates $\beta,\ \lambda,  \ \rho : [0,\infty)\rightarrow (0,\infty)$
satisfy Assumption \ref{assu:beta_lambda_rho}. We assume that Assumption \ref{assu:initial_z} hold.

We start by proving the ODE characterization of $g_S,$ $g_E$ and $g_I$.
     \begin{proof}[Proof of Theorem \ref{thm:gggG}] Throughout  the proof in order to simplify  notation we write $\Xbar$ in lieu of $\Xbar^\Tmc=\xi^{\Tmc,1}$, the SEIR process on $\Tmc$, and $q$ in lieu of $q_1$, the rate function defined in \eqref{eq:S(E)IR-rate}. By Assumption \ref{assu:initial_z}, $g_S(0)=s_0$, $g_E(0)=e_0$ and $g_I(0)=i_0$. Clearly, $G_I(0)=0$. Therefore, the initial conditions \eqref{eq:gggG-Odes_initial} hold. By the fundamental theorem of calculus, $\dot G_I(t)= \beta_t g_I(t)$, which is the fourth equation in \eqref{eq:gggG-Odes}.

       We now turn to the derivation of the evolution of $g_I, \ g_E $ and $ g_S$. This requires keeping track of two states simultaneously since $g_I(t)$, $ g_E(t)$ and $  g_S(t)$ are conditional probabilities associated with the joint law of ${\Xbar}_1(t)$ and ${\Xbar}_\root(t)$. To start, we apply Proposition \ref{prop:filtering} with $\alpha=1$ and $U=\{\root,1\}$ to conclude that ${\Xbar}_{\root,1}$ has the same law as the jump process on the state space $\stateSS_{\extra}\times\stateSS_{\extra}$ with jump rates $\qhat_v(t,x)=\qhat_{v,1}^{\offspring,1}[\{\root,1\}](t,x)$, $v\in\{\root,1\}$, $x\in\Dmc([0,\infty),\stateSS_\extra^2)$, which satisfy, for every $t\geq0$, almost surely
       \begin{equation}\label{eq:qhat-SEIR}
           \hat{q}_{v}(t,\Xbar_{\root, 1}) = \E[q(1,t, \Xbar_v, \Xbar_{\partial_v^{\mathbb{V}}})|\Xbar_{\root, 1}[t)], \quad v \in \{\root, 1\}.
       \end{equation} 
       
       Next, we use the specific form of $q$, as defined in \eqref{eq:S(E)IR-rate} and Propositions \ref{prop:cond_ind_prop} and \ref{prop:symmetry} to obtain a more explicit description of $\qhat_v$, $v\in\{\root,1\}$. Since the probabilities $g_a(t)$, $a\in\{S,\ E,\ I\}$  are conditioned on $\Xbar_\root(t-)=S$ and $\Xbar_1(0)\neq\extra$ (and using the fact that an individual that is in state $R$ remains in that state for all subsequent times), we only need to consider the jumps $\qhat_v(t,\Xbar_{\root,1})$, $v\in\{\root,1\}$ 
       on the events $\{\Xbar_{\root,1}(t-)=(S,S)\}$, $\{\Xbar_{\root,1}(t-)=(S,E)\}$ and $\{\Xbar_{\root,1}(t-)=(S,I)\}$.
       
       For $v,w\in\set{\root,1}$ with $v\neq w$, define $\Bbar_v(t):=\beta_t\E[\Imc(\Xbar_{\partial_v\setminus\set{w}}(t-))| \Xbar_v(t-)=S]$. By the definition of the SEIR jump rates $q=q_1$ in \eqref{eq:S(E)IR-rate}, $B_v$ is the conditional cumulative rate at which the neighbors of $v$ other that $w$ infect the individual at $v$ at time $t$. By Proposition \ref{prop:cond_ind_prop},  \begin{equation}\label{eq:Bbar_v}
\Bbar_v(t)=\beta_t\E[\Imc(\Xbar_{\partial_v\setminus\set{w}})| \Xbar_v(t)=S, \Xbar_w(t)].
       \end{equation}
       Using \eqref{eq:S(E)IR-rate}, \eqref{eq:qhat-SEIR}  and Proposition \ref{prop:filtering}, and proceeding similarly as in the proof of Theorem \ref{thm:SIR-fs}, we can treat $\Xbar_{\root,1}$ as a two particle jump process driven by Poisson noises with intensity measure equal to Lebesgue measure, whose  jumps and jump rates from the states $(S,S)$ , $(S,E)$ and $(S,I)$ can be summarized as follows.
    \begin{align*}
        &\text{Jump: }  & 
  &\text{Rate at time $t$:}
  \\       (S,S) &\rightarrow (S,E)  & 
  & \Bbar_1(t)
  \\  (S,S) &\rightarrow (E,S)  & 
  & \Bbar_{\root}(t)
   \\ (S,E) &\rightarrow (E,E)  & 
  &  \Bbar_{\root}(t)
  \\ (S,E) &\rightarrow (S,I)  & 
  &{\lambda_t}
  \\ (S,I) &\rightarrow (E,I)  & 
  &{\beta_t} + \Bbar_{\root}(t)
  \\ (S,I) &\rightarrow (S,R)  & 
  &{\rho_t},
       \end{align*}
with all other rates being equal to zero.
       Next we fix $h>0$ and $t\geq0$ and obtain expressions for $g_I(t+h),\ g_E(t+h)$, and $g_S(t+h)$ in terms of $g_I(t)$, $g_E(t)$, $g_S(t)$, $ h$, ${\beta_t},$ ${\rho_t}$, and $\sizeb$.
We first consider $g_S$, defined in \eqref{eq:SEIR-probs}. Using monotonicity of the SEIR dynamics, we can write 
\begin{equation}\label{eq:g-s-pr}
    g_S(t+h)=\P({\Xbar}_1(t+h)= S,\ {\Xbar}_1(t)= S \ | \ {\Xbar}_\root(t+h)=S,\ {\Xbar}_\root(t)=S,\ 1\in\Tmc ).
\end{equation}
By an application of Lemma \ref{lem:cond-identity-lemma} with $A=\{{\Xbar}_1(t+h)=S\} $, $A'=\{{\Xbar}_1(t)=S\}$, $B=\{{\Xbar}_\root(t+h)=S\}$, $B'=\{{\Xbar}_\root(t)=S, 1\in\Tmc\}$ we obtain
    \begin{equation}\label{eq:g_S}
        g_S(t+h)= g_S(t)\frac{\P({\Xbar}_\root(t+h)=S,\ {\Xbar}_1(t+h)=S,  | \ {\Xbar}_\root(t)=S, \ {\Xbar}_1(t)=S,  \ 1\in\Tmc)}{\P({\Xbar}_\root(t+h)=S\ | \ {\Xbar}_\root(t)=S, \ 1\in\Tmc)}.
    \end{equation}
    Since $\Bbar_1(t)+\Bbar_\root(t)$ is the rate at which $\Xbar_{\root,1}(t)$ leaves the state $(S,$ $S)$, the numerator on the right-hand side of \eqref{eq:g_S} is equal to $1-h(\Bbar_1(t)+\Bbar_\root(t))+o(h)$. 
    For the denominator, observe that the rate $\qhat_\root(t,\Xbar_{\root,1})$ on the event  $\{\Xbar_\root(t-)=S, 1\in\Tmc\}$ is equal to
    \begin{align*}
        \begin{split}
   & \E[q(1,t,\Xbar_\root,\Xbar_{\partial^\V_\root}) \ | \Xbar_\root(t-)=S, 1\in\Tmc ] 
    \\ = &  \beta_t\E[\infects(\Xbar_1(t-)) \ | \Xbar_\root(t-)=S, \ 1\in\Tmc ] + \beta_t\E[\infects(\Xbar_{\partial^\V_\root\setminus\{1\}}(t-)) \ | \Xbar_\root(t-)=S, \ 1\in\Tmc ] 
    \\ = & \beta_t g_I(t-) + \beta_t \Bbar_\root(t-),
        \end{split}
    \end{align*}
    where the first equality follows from \eqref{eq:S(E)IR-rate} with $\alpha=1$, and the second follows from the definition of $g_I$ in \eqref{eq:SEIR-probs} and by \eqref{eq:Bbar_v} (on observing that the event $\{1\in\Tmc\}$ is $\Xbar_1(t)$-measurable).  Therefore, it follows that
\begin{equation*}
        g_S(t+h)= g_S(t)\frac{1-h(\Bbar_1(t)+\Bbar_{\root}(t))+o(h)}{1-h({\beta_t} g_I(t)+\Bbar_{\root}(t))+o(h)},
    \end{equation*}
Which implies
\begin{align*}
    \begin{split}
        g_S(t+h)-g_S(t)&= g_S(t)\frac{h{\beta_t} g_I(t) -h\Bbar_1(t)+o(h)}{1+o(1)}. 
    \end{split}
\end{align*}
In turn, this implies
\begin{equation}\label{eq:g_s_proof}
    \dot g_S= g_S ({\beta} g_I -\Bbar_1).
\end{equation} 

Similarly, recalling that $g_E(t+h)=\P(\Xbar_1(t+h)=E\ | \ \Xbar_\root(t+h)=S, \ 1\in\Tmc)$ from \eqref{eq:SEIR-probs}, and using the monotonicity property \eqref{eq:mono} with $\alpha=1$, by a similar derivation as \eqref{eq:g-s-pr}-\eqref{eq:g_s_proof},
\begin{align*}
    \begin{split}
g_E(t+h)&=\sum_{a=S,E}g_a(t)\frac{\P({\Xbar}_\root(t+h)=S, \ {\Xbar}_1(t+h)=E|\ {\Xbar}_\root(t)=S , \ {\Xbar}_1(t)=a,\ 1\in\Tmc)}{\P({\Xbar}_\root(t+h)=S|\ {\Xbar}_\root(t)=S ,1\in\Tmc)}
\\ &= \frac{g_S(t)(h\Bbar_1(t)+o(h)) + g_E(t)(1-h({\lambda_t}+\Bbar_{\root}(t))+o(h))}{1-h(g_I(t){\beta_t}+\Bbar_{\root}(t)) +o(h)},
    \end{split}
\end{align*}
       and, hence,
       \begin{align*}
    \begin{split}
&g_E(t+h)-g_E(t)= (1+o(1))(h g_S(t)\Bbar_1(t) -hg_E(t)({\lambda_t} -\beta_t g_I(t)) +o(h)).
    \end{split}
\end{align*}
It follows that 
    \begin{equation}\label{eq:g_e_proof}
        \dot g_E=g_S \Bbar_1 -g_E(\lambda-{\beta} g_I).
    \end{equation} 

    Next, we see that
    \begin{align*}
    \begin{split}
g_I(t+h)&=\sum_{a=S,E,I}g_a(t)\frac{\P({\Xbar}_\root(t+h)=S, \ {\Xbar}_1(t+h)=I|\ {\Xbar}_\root(t)=S , \ {\Xbar}_1(t)=a,\ 1\in\Tmc)}{\P({\Xbar}_\root(t+h)=S|\ {\Xbar}_\root(t)=S ,1\in\Tmc)}
\\ &= \frac{g_S(t)o(h^2)+ g_E(t)(h\lambda_t) + g_I(t)(1-h(\beta+\Bbar_\root+\rho)) +o(h)}{1-h(g_I(t){\beta_t}+\Bbar_{\root}(t)) +o(h)},
    \end{split}
\end{align*}
       by the monotonicity property \eqref{eq:mono} with $\alpha=1$ and by the fact that the probability that two jumps occur in an interval of length $h$ is $o(h^2)$, since the driving noises $\Nbf_{\{\root,1\}}$ as in Proposition \ref{prop:filtering} are independent Poisson point processes with intensity measure equal to Lebesgue measure. We then have
       \begin{align*}
    \begin{split}
&g_I(t+h)-g_I(t)= (1+o(1))(h g_E(t)\lambda_t -hg_I(\beta_t+\rho_t-\beta_tg_I(t)) +o(h)),
    \end{split}
\end{align*}
which implies the third equation in \eqref{eq:gggG-Odes}.

    Recalling that $G_I(t)=\int_0^t\beta_ug_I(u)du$, by the same argument as \eqref{eq:B1_1}-\eqref{eq:d_1_final} in the proof of Theorem \ref{thm:SIR-fs}, $\Bbar_1(t)$ satisfies 
    \begin{align*}
        \begin{split}
            \Bbar_1(t) & = \beta_t g_I(t )\E[ d_1-1\ | \ \Xbar_1(t-)=S ]
            \\ &= 
            \beta_t g_I  \frac{\sum_{k=0}^\infty k\sizeb(k) e^{- k G_I(t)}}{ \sum_{n=0}^\infty \sizeb(n) 
  e^{- n G_I(t)} } .
        \end{split}
    \end{align*}
    Substituting this back into \eqref{eq:g_s_proof} and \eqref{eq:g_e_proof}, we obtain the first and second equations in \eqref{eq:gggG-Odes}. This concludes the proof.
   \end{proof}

 \begin{proof}[Proof of Theorem \ref{thm:SEIR_ODE}] 
    Throughout the proof, we simplify the notation and  write ${\Xbar}$ in lieu of ${\Xbar}^{\Tmc}$. By Assumption \ref{assu:beta_lambda_rho},  the fact that $g_I$ is continuous (which follows from Theorem 
    \ref{thm:gggG}), and the fact that the ODE \eqref{eq:ODE-SEIR} is linear, the initial value problem \eqref{eq:ODE-SEIR}-\eqref{eq:ODE-SEIR_initial} has a unique solution. Clearly by \eqref{eq:SEIR-probs} the initial conditions \eqref{eq:ODE-SEIR_initial} hold.
    
To prove \eqref{eq:ODE-SEIR} we proceed similarly as in the proof of Theorem \ref{thm:SIR_ODE}. We start by considering $Q_{S,S;k}$. Fix $t\geq0$, $h>0$ and $k$ in the support of $\offspring$. Then, using the monotonicity of the SEIR process (see \eqref{eq:mono}) in the second quality, and the fact that $\Xbar_{\partial_\root}(t)=y\in\Xbar^k$ implies that $d_\root=k$, \begin{align}\label{eq:seir-proof-1}
            \begin{split}
                & Q_{S,S;k}(t+h)
                 \\ & =\P({\Xbar}_\root(t+h)=S \ | \ {\Xbar}_\root(0)=S, \ d_\root=k)
                 \\ &= \sum_{y\in\stateSS^{k}}\frac{\P(\Xbar_\root(t+h)=S, \ \Xbar_\root(t)=S, \ \Xbar_\root(0)=S,  \ \Xbar_{\partial_\root}(t)=y)}{\P(\Xbar_\root(0)=S,\ d_\root=k)}
                \\ &= \sum_{ y\in\Smc_{k,t}}\P({\Xbar}_\root(t+h)=S |  {\Xbar}_\root(t)=S,  {\Xbar}_{ \partial_\root}(t)=y)  \P({\Xbar}_{ \root}(t)=S, {\Xbar}_{ \partial_\root}(t)=y|  {\Xbar}_{\root}(0)=S, d_\root=k),
            \end{split}
        \end{align}
where $\Smc_{k,t}:=\{y\in\stateSS^k\ : \ \P(\Xbar_\root(t)=S, \Xbar_{\partial_\root}(t)=y)>0\}$. Since by \eqref{eq:S(E)IR-rate} the jump rate of a susceptible individual with neighbors $y\in\stateSS^{k}$ is equal to ${\beta_t}\infects (y)$, it follows that 
\begin{equation}\label{eq:p1-ode_SEIR}
    \P({\Xbar}_\root(t+h)=S \ |\ {\Xbar}_\root(t)=S, \ {\Xbar}_{{\partial_\root}}(t)=y)= 1-h{\beta_t}\infects (y)+o(h).
\end{equation} 
The right-hand side does not depend on the exact states of the $k-\infects(y)$ neighbors of the root that are not in state $I$. Thus, substituting  the expression in \eqref{eq:p1-ode_SEIR} into the last line of \eqref{eq:seir-proof-1} and rewriting the sum  to be over the number of infected neighbors of $\root$, 
        \begin{align}\label{eq:Qssk}
            \begin{split}
                 Q_{S,S;k}(t+h)= &\sum_{j=0}^k (1-h {\beta_t} j +o(h)) \ \P({\Xbar}_\root(t)=S,\ \infects ({\Xbar}_{ {\partial_\root}}(t))=j\ | \ {\Xbar}_{\root}(0)=S, d_\root=k) 
                \\ =& \sum_{j=0}^k (1-h {\beta_t} j +o(h)) \ \P(\infects({\Xbar}_{ {\partial_\root}}(t))=j\ | {\Xbar}_\root(t)=S, \ {\Xbar}_{\root}(0)=S,\  d_\root=k) \ Q_{S,S;k}(t) 
                \\ =& \sum_{j=0}^k (1-h {\beta_t} j +o(h)) \ \P(\infects({\Xbar}_{ {\partial_\root}}(t))=j\ |  {\Xbar}_\root(t)=S, \ d_\root=k) \ Q_{S,S;k}(t).
            \end{split}
        \end{align}
        Letting $\alpha=1$ in Proposition \ref{prop:cond_ind_prop},  it follows that $\{{\Xbar}_{i}(t) \ : \ i\sim \root \}$ are conditionally i.i.d.  given ${\Xbar}_{\root}(t)=S$ and $ d_\root=k$. For each $m\in\N\cap[1,k]$, by Proposition \ref{prop:symmetry} and an Application of  Proposition \ref{prop:cond_ind_prop} with $A=\T_m:=\{mv\ : v\in\V\}$, the subtree rooted at $m$,  $\partial^\V A=\{\root\}$ and $B=\N\setminus\set{m}$, observing that $d_\root=\sum_{\ell\in\N}\one_{\{\Xbar_\ell(0)\neq \extra\}}$, we have that
        $$\P(\Xbar_m(t)=I \ |\  \Xbar_\root(t)=S, \ d_\root=k)=\P(\Xbar_m(t)=I\ |\  \Xbar_\root(t)=S, \ m\in\Tmc)=g_I(t),$$
        where $g_I$ is defined in \eqref{eq:SEIR-probs}.
        It follows that, conditional on ${\Xbar}_\root(t)=S$ and  $d_\root=k$, $\infects({\Xbar}_{{\partial_\root}}(t))$ has binomial distribution with parameters $(k$, $g_I(t))$. Letting $Y$ be a binomial random variable with parameters $(k,$ $g_I(t))$, it follows from \eqref{eq:Qssk} that 
        \begin{align}\label{eq:p2-ode-SEIR}
        \begin{split}
              Q_{S,S;k}(t+h)&=   Q_{S,S;k}(t)  ( 1-h {\beta_t} \E[Y] +o(h)) 
              \\ & =(1-h{\beta_t} k g_I + o(h)) Q_{S,S;k}(t),
        \end{split}
        \end{align}
         and, thus, 
         \begin{align*}
         \lim_{h\rightarrow0^+} \frac{Q_{S,S;k}(t+h)-Q_{S,S;k}(t)}{h} &=   \lim_{h\rightarrow0^+} \frac{(1-h{\beta_t} k g_I(t) +o(h)-1)Q_{S,S;k}(t)}{h}\\ &= -{\beta_t} k g_I(t)Q_{S,S;k}(t),   
         \end{align*}
         which proves the first equation in \eqref{eq:ODE-SEIR}.
         The derivation of the ODEs for $ Q_{S,E;k}$ and $ Q_{S,I,;k}$ is similar and outlined below. As in the last line of \eqref{eq:seir-proof-1} write,
         \begin{equation}\label{eq:Qsek}
             Q_{S,E;k}(t+h)= \bar{\Qmc}_E(h)+ \bar{\Qmc}_S(h),
         \end{equation}
         where, for $b=\{S,E\}$,
         \begin{equation*}
         \bar{\Qmc}_b(h)=
             \sum_{j=0}^k \P({\Xbar}_\root(t+h)=E, \ {\Xbar}_\root(t)=b, \ \infects ({\Xbar}_{{\partial_\root}}(t))=j \ |  \  {\Xbar}_{\root}(0)=S,\ d_\root=k).
         \end{equation*}
         Recalling the definition of the rates $q_1$ in \eqref{eq:S(E)IR-rate} and using arguments similar to what used to derive \eqref{eq:p1-ode_SEIR}-\eqref{eq:p2-ode-SEIR}, $\bar{\Qmc}_S=(h{\beta_t} k g_I(t) +o(h) )Q_{S,S;k}(t)$ and 
        \begin{align*}
            \begin{split}
                 \bar{\Qmc}_E(h)=& \sum_{j=0}^k (1-{\lambda_t} h +o(h))  \P({\Xbar}_\root(t)=E,\ \infects ({\Xbar}_{ {\partial_\root}}(t))=j |  {\Xbar}_{\root}(0)=S, \ d_\root=k) 
                 \\ =& (1-{\lambda_t} h +o(h))  \sum_{j=0}^k \P({\Xbar}_\root(t)=E,\ \infects ({\Xbar}_{ {\partial_\root}}(t))=j |  {\Xbar}_{\root}(0)=S, \ d_\root=k ) 
                 \\ =& (1-{\lambda_t} h +o(h))
                 \P({\Xbar}_\root(t)=E \ |  {\Xbar}_{\root}(0)=S, \ d_\root=k) 
                 \\ =& (1-{\lambda_t} h + o(h)) Q_{S,E;k}(t).
            \end{split}
        \end{align*}
       Therefore, $Q_{S,E;k}(t+h)-Q_{S,E;k}(t)= hk {\beta_t} g_I(t) Q_{S,S;k}(t)- {\lambda_t} h Q_{S,E;k}(t) + o(h)$ which implies the second equation in \eqref{eq:ODE-SEIR}.
       Proceeding similarly, we obtain the relation 
       \begin{align*}
           \begin{split}
               Q_{S,I;k}(t+h) &=\sum_{b=S,E,I}\P(\Xbar_\root(t+h)=I | \Xbar_\root(t)=b, \Xbar_\root(0)=S d_\root=k)
               \\ &=  \lambda_t h Q_{S,E;k}(t) + (1- \rho_t)h Q_{S,I;k}
           \end{split}
       \end{align*}
       and 
       \begin{align*}
           \begin{split}
               Q_{E,I}(t+h)&=\sum_{b=S,E,I}\P(\Xbar_\root(t+h)=I | \Xbar_\root(t)=b, \Xbar_\root(0)=E, d_\root=k)
               \\ &=  \lambda_t h Q_{E,E}(t) + (1- \rho_t)h Q_{E,I},
           \end{split}
       \end{align*}
       which imply the third and fifth equations in \eqref{eq:ODE-SEIR}.
   Setting $r^E_t:=\lambda_t$, and  $r^I_t:=\rho_t$, for $a\in\{E,\ I\}$ we see that
       \begin{align*}
       \begin{split}
         &  Q_{a,a}(t+h)= \\ &=\sum_{y\in\stateSS ^{k}}\P({\Xbar}_\root(t+h)=a \ |\ {\Xbar}_\root(t)=a, {\Xbar}_{\partial_\root }(t)=y) \P({\Xbar}_\root(t)=a,\ {\Xbar}_{\partial_\root}(t)=y\ | \ {\Xbar}_{\root}(0)=a) 
         \\& = (1-h {r^a_t} + o(h))  \sum_{y\in\stateSS ^{k}}\P({\Xbar}_\root(t)=a,\ {\Xbar}_{\partial_\root}(t)=y\ | \ {\Xbar}_{\root}(0)=a) 
          \\& = (1-h {r^a_t} + o(h)) Q_{a,a}(t),
       \end{split}
       \end{align*}
       which proves the fourth and sixth equations in \eqref{eq:ODE-SEIR}, thus concluding the proof.
           
     \end{proof}

\section{Proof of Proposition \ref{prop:f_If_SF_I}} \label{sec:odes-well-posed}
In this section we prove the well-posedness of the ODE system  \eqref{eq:f_If_SF_I}-\eqref{eq:ffF-initial}. We start with the following elementary result.

\begin{lemma}\label{lemm:phi_lip}
    Suppose that $\offspring\in\Pmc(\natzero)$ has a finite third moment. 
    % Then, the size biased distribution $\sizeb$, defined in \eqref{eq:sizab}, has finite second moment. 
    Then $\Phi_\sizeb$, defined in \eqref{eq:def_Phi},  is Lipschitz continuous on $[0,\infty)$.
\end{lemma}
\begin{proof}
It is easy to see that under the assumption that $\theta$ has a finite third moment, the size-biased distribution $\sizeb$, defined in \eqref{eq:sizab}, has a finite second moment.  Indeed, 
    let $\Yhat$ and $Y$ be  random variables with laws $\sizeb$ and $\offspring$, respectively. By \eqref{eq:sizab}, it is easy to see  that \begin{equation*}
        \E[\Yhat^2]= \frac{\E[Y^3]-2\E[Y^2]+\E[Y]}{\E[Y]},
    \end{equation*}
    which is finite since $\offspring$ has finite third moment. 
    
    For $z\in[0,\infty)$ note that $M_\sizeb(-z)=\E[e^{-z\Yhat}]\leq 1$,  and so by the dominated convergence theorem 
$\lim_{z\rightarrow\infty}M_{\sizeb}(-z)=\sizeb(0)$. 
    Since $\sizeb$ has finite second moment, again by the dominated convergence theorem $M_\sizeb''(-z)=\E[\Yhat^2e^{-z\Yhat}]\leq \E[\Yhat^2]<\infty$,  $M_{\sizeb}''$ is continuous on $(-\infty, 0)$ and $\lim_{z\rightarrow\infty} M''_{\sizeb}(-z)=0$. 
    Thus,   it follows from the limits established above that
      \begin{equation}
      \label{rel-Mratio} 
  \sizeb(0)>0 \quad \Rightarrow \quad   \lim_{z\rightarrow\infty} \frac{M''_{\sizeb}(-z)}{M_{\sizeb}(-z)} = \frac{0}{\sizeb(0)}= 0
   \end{equation}
    Now,  setting $\Phi:=\Phi_{\sizeb}$ for conciseness, it follows that 
    \begin{equation}
    \label{eq-Phiprime}
        \Phi'(z):=\frac{d}{dz}\Phi(z)= \frac{(M'_{\sizeb}(-z))^2- M''_{\sizeb}(-z)M_\sizeb(-z)}{M^2_{\sizeb}(-z)} = (\Phi(z))^2-\frac{M''_{\sizeb}(-z)}{M_{\sizeb}(-z)}.
    \end{equation}
   By Lemma \ref{lemm:Phi_props}, $\Phi$ is bounded on $[0,\infty)$. Furthermore, $M''_{\sizeb}(0)/M_{\sizeb}(0)=\E[\Yhat^2]$. 
   % and by the dominated convergence theorem, $\lim_{z\rightarrow\infty} M''_{\sizeb}(-z)=0$.
   Recall the quantity $\underline{d}_\sizeb  = \min\{d\in\natzero : \sizeb(d)>0\}$ introduced in \eqref{eq-undd}. 
Using \eqref{rel-Mratio} for the case $\underline{d}_\sizeb=0$ (which is equivalent to  $\sizeb(0)>0$), and a similar argument as \eqref{eq:Phi_sizeb(infty)} in Lemma \ref{lemm:Phi_props}  for the case $\underline{d}_\sizeb>0$,
   we have 
   \begin{equation} \label{eq:Phi'_lim}
       \lim_{z\rightarrow\infty} \frac{M''_{\sizeb}(-z)}{M_{\sizeb}(-z)} = \underline{d}_\sizeb^2.
   \end{equation}
   
 Together with \eqref{eq-Phiprime}, the continuity of $\Phi$ on $[0,\infty)$ and the continuity of  $M_\sizeb$ and    $M^{\prime \prime}_\sizeb$ on $(-\infty, 0]$, this implies that $\Phi'$ is uniformly bounded on $[0,\infty)$.  This completes the proof. 
\end{proof}

\begin{proof}[Proof of Proposition \ref{prop:f_If_SF_I}] By Assumption \ref{assu:beta_rho}, $\beta$ and $\rho$ are continuous in $t$. By Lemma \ref{lemm:Phi_props}, $\Phi(z)$ is continuous in $z\in[0,\infty)$. Therefore, the right-hand side of the ODE \eqref{eq:f_If_SF_I} is continuous and so by Peano's existence theorem, there exists  $\tau\in(0,\infty)$ and a solution $(f_S$,$f_I$,$F_I):[0,\tau)\rightarrow\R^3$ to  \eqref{eq:f_If_SF_I}-\eqref{eq:ffF-initial} on $[0,\tau)$.

Next, fix $s\in[0,\tau)$. We claim that $(f_S(s)$,$f_I(s)$,$F_I(s))\in[0,1]\times[0,1]\times[0,\infty)$. Since $f_S(0)=s_0\in(0,1)$ and 
$f_I(0)=1-s_0\in(0,1)$ 
and the right-hand side of the first (respectively, second)  equation in  \eqref{eq:f_If_SF_I} is equal to $0$ whenever $f_S(t)=0$ (respectively, $f_I(t) =0$), it follows that $f_S(s)\geq 0$ (respectively, $f_I(s) \geq 0$). In turn,  this implies that $\dot F_I(s) \geq 0$, and therefore that $F_I(s) >  0$ (since $F_I(0)=0$ and $\dot{F}_I(0+) = f_I(0) > 0$).   Now, by summing the first two equations in \eqref{eq:f_If_SF_I}, we obtain 
\begin{equation}
    \dot f_I(s)+ \dot f_S(s) =f_I(s) \beta_s \left ( f_S(s)+ f_I(s) - 1 -\frac{\rho_s}{\beta_s}\right).
\end{equation}
Since $f_S(0)+f_I(0)=1$, it follows that $f_S(s)+f_I(s)$ is strictly decreasing in $s$, and in particular  $f_S(s)+f_I(s)\leq 1$. 

 Finally,  by Lemma \ref{lemm:phi_lip}, the right-hand side of  \eqref{eq:f_If_SF_I} is  Lipschitz continuous in $(f_S,$ $f_I$, $F_I)$ on $[0,1]\times[0,1]\times[0,\infty)
$.  Thus, it  follows that $\tau=\infty$ and \eqref{eq:f_If_SF_I}-\eqref{eq:ffF-initial} has a unique solution for all times.   \end{proof}

Proposition \ref{prop:gggG} is proved in the same way as the proof of \ref{prop:f_If_SF_I} by first using Lemma \ref{lemm:Phi_props} and Assumption \ref{assu:beta_lambda_rho} to establish the existence of $(g_S$, $g_E,$ $g_I,$ $G_I)$ solving \eqref{eq:gggG-Odes}-\eqref{eq:gggG-Odes_initial} on $[0,\tau)$ for some $\tau\in(0,\infty)$, then  showing that such a solution stays in $[0,1]\times [0,1]\times [0,1]\times[0,\infty)$, where, by Lemma \ref{lemm:phi_lip}, the right-hand side of \eqref{eq:gggG-Odes} is Lipschitz.

\bibliography{bibliography}

%%%%%%%%%%%%%%%%%
\end{document}